\documentclass{amsart}
\usepackage[foot]{amsaddr}
\usepackage[english]{babel}
\usepackage{amsthm}
\usepackage{scalerel}
\usepackage{protosem}
\usepackage{linearb}
\usepackage{url}
\usepackage{breakurl}
\usepackage[breaklinks]{hyperref}
\usepackage{caption}
\usepackage[labelformat=simple,labelfont={}]{subcaption}

\newcommand{\introthmname}{}
\newtheorem{introthminn}{\introthmname}

\DeclareFontFamily{U}{cbgreek}{}
\DeclareFontShape{U}{cbgreek}{m}{n}{
	<-6>    grmn0500
	<6-7>   grmn0600
	<7-8>   grmn0700
	<8-9>   grmn0800
	<9-10>  grmn0900
	<10-12> grmn1000
	<12-17> grmn1200
	<17->   grmn1728
}{}
\DeclareFontShape{U}{cbgreek}{bx}{n}{
	<-6>    grxn0500
	<6-7>   grxn0600
	<7-8>   grxn0700
	<8-9>   grxn0800
	<9-10>  grxn0900
	<10-12> grxn1000
	<12-17> grxn1200
	<17->   grxn1728
}{}

\DeclareRobustCommand{\Qoppa}{%
	\text{\usefont{U}{cbgreek}{\normalorbold}{n}\symbol{21}}%
}
\makeatletter
\newcommand{\normalorbold}{%
	\ifnum\pdf@strcmp{\math@version}{bold}=\z@ bx\else m\fi
}

\DeclareSymbolFont{extraitalic}      {U}{zavm}{m}{it}
\DeclareMathSymbol{\Qopa}{\mathord}{extraitalic}{161} 

\DeclareMathSymbol{\Sampi}{\mathord}{extraitalic}{165} 
\DeclareMathSymbol{\sampi}{\mathord}{extraitalic}{166} 
\DeclareMathSymbol{\Stigma}{\mathord}{extraitalic}{167} 
\DeclareMathSymbol{\stigma}{\mathord}{extraitalic}{168} 

\newcommand{\qopa}[0]{
\raisebox{-0.2ex}{\scaleobj{1.1}{\Qoppa}}
}
   
      \newcommand{\pluma}[0]{\raisebox{-0.35ex}{\scaleobj{0.65}{$\textpmhg{i}$}}} 
      \newcommand{\plumaa}[0]{\raisebox{-0.35ex}{\scaleobj{0.65}{$\textpmhg{ii}$}}}
      \newcommand{\plumaaa}[0]{\raisebox{-0.35ex}{\scaleobj{0.65}{$\textpmhg{iii}$}}}

\newtheorem{thmx}{Theorem}
\newtheorem{corx}[thmx]{Corollary}

    \newtheorem{innercustomthm}{Theorem}
\newenvironment{customthm}[1]
{\renewcommand\theinnercustomthm{#1}\innercustomthm}
{\endinnercustomthm}

\usepackage{amssymb}
\usepackage{mathtools}
\usepackage{thmtools}
\usepackage{xfrac}
\usepackage{thm-restate}
\usepackage{enumerate}
\usepackage{hyperref}
\usepackage{xfrac}
\usepackage{cleveref}
\usepackage{verbatim}
\usepackage{amsmath}
\usepackage{amssymb}
\usepackage{mathabx}
\usepackage{graphicx}
\usepackage{amsfonts}
\usepackage{pb-diagram}
\usepackage{tikz-cd}
\usepackage{graphicx}
\usepackage{xcolor}
\usepackage{adjustbox}
\usepackage{mathtools}
\usepackage{todonotes}
\newtheorem{theorem}{Theorem}
\usepackage{mathrsfs}
\usepackage{bm}
\usepackage{cypriot}
\usepackage{sarabian}
\usepackage{hieroglf}

\newcommand{\transv}{\mathrel{\text{\tpitchfork}}}
\makeatletter
\newcommand{\tpitchfork}{%
  \vbox{
    \baselineskip\z@skip
    \lineskip-.52ex
    \lineskiplimit\maxdimen
    \m@th
    \ialign{##\crcr\hidewidth\smash{$-$}\hidewidth\crcr$\pitchfork$\crcr}
  }%
}

\newtheorem{definition}[theorem]{Definition}

\newtheorem{corollary}[theorem]{Corollary}
\newtheorem{lemma}[theorem]{Lemma}
\newtheorem{remark}[theorem]{Remark}
\newtheorem{observation}[theorem]{Observation}
\newtheorem{fact}[theorem]{Fact}
\newtheorem{notation}[theorem]{Notation}

\newtheorem{proposition}[theorem]{Proposition}

\newtheorem{question}{Question}
\newtheorem{conjecture}{Conjecture}

\newtheorem{claim}[theorem]{Claim}

\newcommand{\jet}[0]{\mathsf{J}^{1}}

\newcommand{\loc}[1]{\big|_{#1}}

\newenvironment{subproof}[1][\proofname]{%
\begin{proof}[#1]%
	}{%
\end{proof}%
}

\newenvironment{subsubproof}[1][\proofname]{%
\begin{proof}[#1]%
}{%
\end{proof}%
}

\newcommand{\U}[0]{\mathbb{U}}

\newcommand{\nin}[0]{\notin}

\newcommand{\nn}[0]{\mathsf{N}}
\newcommand{\N}[0]{\mathbb{N}}

\renewcommand{\U}[0]{\mathcal{U}}
\newcommand{\V}[0]{\mathcal{V}}
\newcommand{\W}[0]{\mathcal{W}}

\title[Minimality of the compact-open topology]{The compact-open topology on the diffeomorphism or homeomorphism group of a smooth manifold without boundary is minimal in almost all dimensions}

\renewcommand{\emph}[1]{\textbf{#1}}
\newcommand{\R}[0]{\mathbb{R}}

\newcommand{\rl}[1]{|#1|} 
   
\newcommand{\hh}[0]{\mathfrak{H}}
\renewcommand{\gg}[0]{\mathfrak{G}}
\newcommand{\PS}[2]{\hh_{\rl{#2}}}   
\renewcommand{\SS}[2]{\hh_{#2}} 
\renewcommand{\emph}[1]{#1} 

\newcommand{\Sp}[0]{\mathsf{Sp}}
\newcommand{\Tsp}[0]{\mathsf{Tsp}}

\author{J. de la Nuez Gonz\'alez }
\email{jnuezgonzalez@gmail.com}
\address{Korea Institute for Advanced Study (KIAS)}
\date{\today}
\thanks{Work supported by Samsung Science and Technology Foundation under Project Number SSTF-BA1301-51 and by KIAS individual grant MG084001.}
\subjclass{57S05,22A05}
\begin{document}
\begin{abstract}
We show that for any connected smooth manifold $M$ of dimension different from $3$ the restriction of the compact-open topology to the diffeomorphism group of $M$ is minimal, i.e., the group does not admit a strictly coarser Hausdorff group topology. This implies the minimality of the compact-open topology on the homeomorphism group of $M$ in all dimensions different from $3$ and $4$. In those cases for which in addition to all of this automatic continuity is known to hold, such as when $M$ is closed, one can conclude that the compact-open topology is the only separable Hausdorff group topology on the homeomorphism group.
\end{abstract}

\maketitle

\renewcommand{\H}[0]{\mathcal{H}}
\newcommand{\G}[0]{\mathcal{G}}
\newcommand{\T}[0]{\mathfrak{t}}
\newcommand{\tc}[0]{\T_{co}}
\newcommand{\nd}[0]{\mathcal{N}_{\T}(1)}
\newcommand{\cs}[0]{\mathcal{CS}}
\numberwithin{fact}{section}
\numberwithin{definition}{section}
\numberwithin{remark}{section}
\newcommand{\ninf}[1]{\lVert #1 \rVert_{\infty}} 
\newcommand{\norm}[1]{\lVert #1 \rVert_{2}} 
\newcommand{\J}[0]{\hat{I}} 
 \newcommand{\w}[0]{\mathcal{W}}

\numberwithin{theorem}{section}
\numberwithin{lemma}{section}
\numberwithin{corollary}{section}
\numberwithin{observation}{section}
\numberwithin{claim}{section}
\numberwithin{notation}{section}
\numberwithin{proposition}{section}

\setcounter{section}{-1}
\section{Introduction}

  We begin by recalling the following definition:
  \begin{definition}
  Let $G$ be a group and $\T$ a Hausdorff group topology on $G$. We say that $\T$ (resp. the topological group $(G,\T)$) is minimal if $G$ admits no Hausdorff group topology strictly coarser than $\T$.
  \end{definition}
  This definition abstracts a key feature of compact Hausdorff topological groups and has received substantial attention in the literature. Minimal topological groups are a rich subject with some surprising connections with other fields, such as logic \cite{ben2016weakly} or number theory \cite{shlossberg2023minimality}. We refer the reader to the comprehensive survey \cite{dikranjan2014minimality} for the broader context and additional questions.
  
  Given a topological space $X$, the compact-open topology on the homeomorphism group of $\H(X)$ of $X$ is the topology generated by all subsets of the form $[K,U]=\{f\in G\,|\,f(K)\subseteq U\}$, for $K\subseteq X$ compact and $U\subseteq X$ open in $X$. The following classical result is due to Arens \cite{arens1946topologies} in the non-compact case. 
  \begin{fact}
  \label{f: arens}Let $X$ be a Hausdorff, locally compact, and locally connected topological space. Then the compact-open topology is a group topology on $\H(X)$.
  \end{fact}
  
  It is a natural problem whether the compact-open topology on the homeomorphism groups of familiar spaces under the assumptions of \ref{f: arens} is minimal. The following is formulated Question 4.28 in \cite{dikranjan2014minimality} in the compact case (minimality of the homeomorphism group is referred to there as $M$-compactness).
     \newcommand{\D}[0]{\mathcal{D}}
 \begin{question}
 	\label{q: basic}For which manifolds $M$ is the compact-open topology on $\H(M)$ minimal?
 \end{question}
  
  It is a result of Gartside and Glynn \cite{gartside2003autohomeomorphism} that the compact-open topology is minimal in the case of $1$-manifolds (in fact the minimum Hausdorff group topology on the group). 
  If $dim(M)>1$ and  $\partial M\neq\emptyset$, then minimality does not hold for the compact-open topology. A witness of this fact is given in \cite{chang2017minimum} as a particular example of a more general phenomenon, although in this very specific context the behavior can be traced back to \cite{gamarnik1991minimality}. For completeness, we indicate the proof. Consider the restriction homomorphism $\rho:\H(M)\to \H(int(M))$, where $int(M)=M\setminus\partial M$, and let $\T_{co}^{\partial}$ denote the preimage by $\rho$ of the compact-open topology on $\H(int(M))$. Then $\T_{co}^{\partial}$ is a Hausdorff group topology strictly coarser than $\tc$. For a smooth manifold $M$ it is easy to see that the inclusion remains proper if we restrict to the group of diffeomorphisms.

  In this work we settle Question \ref{q: basic} for a large family of manifolds without boundary. In fact, we show something somewhat more general. Given a manifold $M$ we denote by $\H_{c0}(M)$ is the subgroup of all elements of the homeomorphism group $\H(M)$ that are isotopic to the identity through a compactly supported topological ambient isotopy (see Section \ref{s: preliminaries}).
  
  \newcommand{\bodythmhomeos}[0]{
   Assume we are given a connected manifold $M$ and $\gg\leq \H(M)$ as follows:
   \begin{itemize}
   	\item $M$ admits a smooth structure,
   	\item $\partial M=\emptyset$, 
   	\item $dim(M)\nin\{3,4\}$,
   	\item $\H_{c0}(M)\leq\gg$.
   \end{itemize}
   Then the restriction of the compact-open topology to $\gg$ is minimal. 
   } 
 
  \begin{thmx}
  	\label{c: main} \bodythmhomeos  
  \end{thmx}
  We deduce Theorem \ref{c: main} from a related result for diffeomorphism groups, Theorem \ref{t: main} below, valid also in dimension $4$ and whose proof contains the vast majority of the mathematical content of this article. Before stating it, let us indicate how Theorem \ref{c: main} complements other results in the literature.
    
  It was shown in \cite{mann2016automatic}, building upon the dimension $1$ and $2$ cases in \cite{rosendal2006automatic} and \cite{rosendal2008automatic} respectively, that for any compact manifold $M$ the group $\H_{0}(M)$, equipped with the compact-open topology, satisfies the automatic continuity property, 
  i.e., any algebraic homomorphism from said group to a separable topological group is continuous (the property then passes to any supergroup of $\H_{0}(M)$ inside $\H(M)$). This was later generalized by the same author in \cite{mann2020automatic} to certain groups of homeomorphisms of non-compact manifolds. By combining these results with Theorem \ref{t: main}, we obtain the following:
  
  \begin{corx}
  \label{c: uniqueness}Let $M'$ be the interior of a smooth connected compact manifold of dimension $2$ or at least $5$ and let $M$ satisfy one of the following conditions:
  \begin{itemize}
  	\item if $dim(M')\geq 5$, then $M:=M'$,
  	\item if $dim(M')=2$, then $M:=M'\setminus F$, where $F\subset M$ is a finite set or the union of a finite set and a Cantor set.
  \end{itemize}
  Then the compact-open topology is the only separable Hausdorff group topology on the group $\H(M)$.
 
  If $M$ is closed, then the restriction of the compact-open topology is the only separable Hausdorff group topology on any group $\gg$ with $\H_{0}(M)\leq\gg\leq\H(M)$. In particular, any algebraic homomorphism from $\H_{0}(M)$ to a separable Hausdorff topological group has either a trivial image or is a homeomorphic embedding. 
  \end{corx}
  The last clause relies on the simplicity of $\H_{0}(M)$ for closed $M$, which follows, for instance, from Corollary 1 and Lemma 4 in \cite{whittaker1963isomorphic} and Corollary 1.3 in \cite{edwards1971deformations}.
   
  Note that the compact-open topology was already known to be the unique complete separable group topology on the homeomorphism group of a compact manifold \cite{kallman1986uniqueness}, a fact that can also be recovered as a corollary of automatic continuity \cite{mann2016automatic}. The result in \cite{mann2020automatic} is stated in terms of the group of homeomorphisms of a manifold $M$ preserving a set $F$ as in the statement. However, as discussed there, in dimension $2$ the restriction map is an isomorphism of topological groups between the former and $\H(M\setminus F)$.
    
  Let us conclude by stating our main technical result, whose proof will occupy most of the paper. Given a smooth manifold $M$, we denote by $\D_{c0}(M)$ the subgroup of its diffeomorphism group $\D(M)$ consisting of all those $h\in\D(M)$ diffeotopic to the identity through a compactly supported diffeotopy (see Section \ref{s: preliminaries} for detailed definitions).  

  \newcommand{\bodythmdiffeos}[0]{
   Assume a we are given smooth connected manifold $M$ and $\gg\leq\D(M)$ satisfying:
   \begin{enumerate}[(a)]
   	\item $\partial M=\emptyset$, 
   	\item \label{dc0inside}$dim(M)=2$ or $dim(M)\geq 4$,
   	\item \label{cdc0inside}$\D_{c0}(M)\leq\gg$.
   \end{enumerate}
    Then the restriction of the compact-open topology to $\gg$ is minimal. 
  }
 
  \begin{thmx}
  \label{t: main} \bodythmdiffeos
  \end{thmx}

  \paragraph{\bf{Outline of the paper}}
        
        Section \ref{s: preliminaries} establishes some notation and recalls some basic facts, with which many readers will be familiar.
        Section \ref{s: Morse} takes place in a general setting in which we are given an isotopy $H: N\times I\to M$ of an $(m-1)$-dimensional manifold $N$ into an $m$-dimensional manifold $M$ and some fixed embedded $(m-1)$-dimensional sub-manifold $L\subseteq M$. We show that under some conditions on $H$ up an arbitrarily small perturbation one can always assume that outside a discrete set of points $\mathcal{P}\subseteq M\times I$ the embedding $H(-,t):N\to M$ is transverse to $L$, whereas at each point in $\mathcal{P}$ some Morse-like transition admitting a simple description occurs (Lemma \ref{l: morse for isotopies}). If $H=(G_{t}\circ\iota)_{t\in I}$ for some fixed embedding of $N$ in $M$ and some compactly supported diffeotopy $G$ of $Id_{M}$, then this can be done by perturbing $G$. 
        
        We develop this further on this in Section \ref{s: isotopy operations}. In particular, Corollary \ref{c: non-destructive isotopies}, establishes a connection between Section \ref{s: Morse} and the problem of decomposing elements of $\D(M)$ as products of elements with certain properties. We also show in Lemma \ref{l: commutation move} that certain condition on the sequence of transitions of the isotopy $H$ in question can always be assumed to hold. It will be again crucial to show this can be done by modifying the diffeotopy $G$.
         
        Sections \ref{s: neighbourhoods are rich} to \ref{conclusion} are devoted to the proof of Theorem \ref{t: main}, which is roughly structured as follows. We fix $M$ of dimension $m=2$ or $m\geq 4$ and $\gg\leq\D(M)$ as in the statement of the theorem, as well as some Hausdorff group topology $\T$ on $\gg$ strictly coarser than the restriction of the compact-open topology.
        We will write $\nd$ for the collection of neighbourhoods of the identity in $\tau$.  
        The main result of Section \ref{s: neighbourhoods are rich} is Corollary \ref{c: control on vertices}, which states that for any neighbourhood $\V\in\nd$ there is some sufficiently rich family of embedded simplicial complexes $\Delta$ of codimension $1$ such that for each of those $\Delta$ the set $\V$ contains some suitably defined subgroup $\PS{\hh}{\Delta}$ of the pointwise stabilizer of $\Delta$ in $M$ (Definition \ref{d: stabilizers}). 
        In Section \ref{s: compression moves} we use this to conclude that given any $\V\in\nd$ and any embedded arc $\alpha$ in $M$ there is some neighbourhood $U$ of $\alpha$ in $M$ such that any element diffeotopic to the identity through a diffeotopy supported in $U$ is in $\mathcal{V}$ (Corollary \ref{c: 0-slices}). It is here that the assumption $m\neq 3$ becomes relevant. 
        
        In Section \ref{conclusion} we prove Theorem \ref{t: main} by showing that $\T$, which was initially assumed above to be Hausdorff, in fact cannot be. More precisely, in Proposition \ref{p: slice induction} we show that any element $g\in\gg$ which is diffeotopic to the identity in some embedded ball $D$ must belong to the intersection of all $\V\in\nd$. It suffices to show the existence of some universal constant $N>0$ such that for any symmetric $\U\in\nd$ we always have $g\in\mathcal{U}^{N}$.
         We know by then any such $\U$ contains $\PS{\hh}{\Delta}$ for some embedded simplicial complex $\Delta\subseteq M$. We start by showing one can always assume $\Delta^{(m-2)}\cap D=\emptyset$ (Lemma \ref{l: cleaning}), and that it is in fact possible to assume $\U$ contains a larger ``flexible" stabilizer $\SS{\hh}{\Delta}$ (Lemma \ref{l: different stabilizers}). This will allow us to apply the results of Section \ref{s: Morse} and \ref{s: isotopy operations}.                  
         
         An inductive argument that uses Corollary \ref{c: 0-slices} for its base case allows us to assume that for any embedded ball $F\subseteq M$ of proper codimension there is some neighbourhood $V$ of $F$ such that any element diffeotopic to the identity by a diffeotopy supported in $V$ is in $\U$. The tools from Sections \ref{s: Morse} and \ref{s: isotopy operations}, then allow us to generate $g$ as a suitable product of elements in $\SS{\hh}{\Delta}$ and elements of the form above. The inductive hypotheses on the topology $\T$ are somewhat delicate here and force us to actually prove the results of Sections \ref{s: neighbourhoods are rich} under weaker assumptions than described in the previous paragraph.
        
        In Section \ref{s: proof main results} we deduce Theorem \ref{c: main} from Theorem \ref{t: main}. This involves a brief callback to Section \ref{s: neighbourhoods are rich} and simple applications of results available in the literature. We close with some comments and remarks in Section \ref{s: questions}.
             
        We bring to the reader's attention that the proofs in the text sometimes contain auxiliary sublemmas with their proofs. We use the shading of the closing square to indicate the degree of nesting of a given proof.  
        
\section{Preliminaries}
  \label{s: preliminaries}
  
  \subsection*{Generalities}
  
  Let $M$ be a connected topological manifold. If we fix a metric $d$ on $M$ compatible with its topology, then a base of neighbourhoods of the identity for the compact-open topology $\tau_{co}$ on $\H(M)$ is given by the collection of sets
  $$
  \V_{K,\epsilon} :=\{g\in G\,|\,\forall p\in K\,\,d(p,g\cdot p),d(p,g^{-1}\cdot p)<\epsilon\}
  $$
  where $K$ ranges over all compact subsets of $M$ and $\epsilon$ over all positive reals. 
   
  Sometimes we might only be interested in the set $\V_{\epsilon}:=\V_{M,\epsilon}$. For any subset $A\subseteq M$ and $\epsilon>0$ we let $\nn_{\epsilon}(A):=\{p\in M\,|\,d(p,A)<\epsilon\}$.
  When working with $\gg\leq\H(M)$ we will use the same notation to denote the restriction of the topology $\tau_{co}$ and the neighbourhoods above to $\gg$.
  
  From now on, and unless we explicitly use the term ``topological manifold", all the manifolds considered will be assumed to be smooth and equipped with a distance function $d$ induced by some complete Riemannian metric on the manifold.  
  
  Given $g\in\H(M)$ we write $supp(g):=\{x\in M\,|\,g(x)\neq x\}$. 
  We will often use the fact that any smooth manifold $M$ is the union of an ascending chain of compact submanifolds $M_{0}\subseteq\mathring{M}_{1}\subseteq M_{1} \dots$ (see, for instance, the remark after 4.3.2 in \cite{wall2016differential}).
  We will use the symbol $\circ$ to denote composition between heterogeneous smooth maps, omitting it when composition takes place between elements living in the same diffeomorphism (or homeomorphism) group. Working in a group we use the notation $g^{h}:=h^{-1}gh$ wherever it is practical to do so.

 \subsection*{Transversality} Given a smooth map $f\in C^{\infty}(N,M)$ and a submanifold $L\subseteq N$, we say that $f$ is transverse to $L$ if and only if we have $im(Df\loc{p})+T_{q}L=T_{q}M$ for any point $p\in M$ such that $q:=f(p)\in L$, written as $f\transv L$. Recall that $f\transv L$ implies that $f^{-1}(L)$ is a submanifold of $M$ with the same codimension as that of $L$ in $N$.  
 If $N_{1},N_{2}\subseteq M$ are submanifolds, we say that they are in general position if $\iota\transv N_{2}$ where $\iota$ is the inclusion of $N_{1}$ in $M$.\footnote{Note that this relation is actually symmetric.}

  \subsection*{The Whitney topology and the $\mathcal{D}$-topology}
  \newcommand{\wt}[0]{\mathsf{W}}
  \newcommand{\swt}[0]{\mathsf{W}^{s}}
  Given smooth manifolds $N$ and $M$, we denote the strong Whitney topology (or simply the Whitney topology) 
  on $C^{\infty}(N,M)$ by $\wt(N,M)$. Recall that this is the union, for $r\geq 0$ of the restrictions to 
  $C^{\infty}(N,M)$ of the topologies $\wt^{r}(N,M)$ on $C^{r}(N,N)$ of locally uniform convergence in all derivatives of order at most $r$. A useful refinement of the Whitney topology on $C^{\infty}(N,M)$ is the so-called $\mathcal{D}$-topology, or $\D(N,M)$ (see \cite{michor2011manifolds} p.37). 
  \begin{definition}
  	\label{d: D-topology}Assume $\partial N=\emptyset$ and fix any sequence $(K_{n})_{n\in M}$ of compact submanifolds of $N$ with $K_{0}=\emptyset$, $K_{n}\subseteq\mathring{K}_{n+1}$ for all $n\geq 0$ and $\bigcup_{n\geq 0}K_{n}=N$. 
  	For any $n\geq 0$ let $\rho_{n}:C^{\infty}(N,M)\to C^{\infty}(N\setminus\mathring{K}_{n},M)$ be the obvious restriction map
  	 
  	Then a basis of neighbourhoods for $\mathcal{D}(N,M)$ is given by all the intersections of the form $\bigcap_{n\geq 0}\rho_{n}^{-1}(U_{n})$, where $U_{n}\in\wt(N\setminus \mathring{K}_{n},M)$.  
  \end{definition}
  
  We collect some basic facts below, where (\ref{whitney infinity}) follows easily from the definition and (\ref{d infinity}) from (\ref{whitney infinity}). We refer the reader to Chapter $2$ in \cite{hirsch2012differential}, and Chapter $4$ in \cite{michor2011manifolds} for more details. 
  \begin{fact}
  	\label{f: properties Whitney}
  	The following holds: 
  	\begin{enumerate}
  		\item \label{whitney baire}$C^{\infty}(N,M)$ endowed with either $\wt(N,M)$ or $\mathcal{D}(N,N)$ is a Baire space 
  		(\cite{michor2011manifolds}, pp. 34, 38).
  		\item \label{whitney rank}The following collections of maps are open in $\wt(N,M)$:
  		\begin{itemize}
  			\item the collection of maps whose differential has rank $\geq k$,
  			\item the collection of embeddings of $N$ in $M$ (\cite{hirsch2012differential}, Ch.2, Thm. 1.4),
  			\item the collection of proper maps from $N$ to $M$ (\cite{hirsch2012differential}, Ch.2, Thm. 1.5).
  		\end{itemize}
  		\item \label{whitney infinity} Let $L\subseteq\mathring{N}$ be a codimension $0$ submanifold of $N$, $f\in C^{r}(N,M)$ and $\W$ some neighbourhood of $f$ in $\wt^{r}(N,M)$, $r\geq 0$. Then there exists some neighbourhood $\V$ of $f_{\restriction \mathring{L}}$ in $\wt(\mathring{L},N)$ such that for any $h\in\V$ we have $f_{\restriction N\setminus\mathring{L}}\cup h\in \V \subseteq C^{r}(N,M)$.
  		\item \label{d infinity} Let $L\subseteq N\setminus\partial N$ be a codimension $0$ submanifold of $N$, $f\in C^{\infty}(N,M)$ and $\W$ some neighbourhood of $f$ in $\wt(N,M)$. Then there exists some neighbourhood $\V$ of $f_{\restriction \mathring{L}}$ in $\mathcal{D}(\mathring{L},N)$ such that for any $h\in\V$ we have $$f_{\restriction N\setminus\mathring{L}}\cup h\in\W\subseteq C^{\infty}(N,M).$$
  	\end{enumerate}  	 
  \end{fact}

  \subsection*{Isotopies}

  We denote by $I$ the interval $I=[-1,1]$. 
  Given manifolds $N,M$ a homotopy of $N$ in $M$ is a continuous map $H: N\times I\to M$. 
  If $M$ and $N$ are smooth manifolds, an isotopy of $M$ in $N$ is a smooth homotopy
  $H: N\times I\to M$ such that for any $t\in I$ the map $H_{t}:=H(-,t):N\to M$ is an embedding. We will refer to $I$ as the time component of $N\times I$ and denote the associated direction in the tangent space by $\partial_{t}$.

  Given elements $g,g'$ of $\D(M)$ ($\H(M)$), a diffeotopy (topological ambient isotopy) from $g$ to $g'$ is an isotopy (resp. homotopy) of $M$ in itself such that $G_{t}\in\D(M)$ (resp. $G_{t}\in\H(M)$) for all $t\in I$. We will also say that $G$ is a diffeotopy (topological ambient isotopy) of $g$. If instead of $I$ we have any other interval, we speak of a generalized isotopy. 
  
  For convenience, we will sometimes use the notation $(H_{t})_{t\in I}$ to present an isotopy. 
  Given an embedding $\iota:N\to M$ and a diffeotopy $G$ of $M$, we also write 
  $\iota^{*}(G)$ for the isotopy of $N$ in $M$ given by $(G_{t}\circ\iota)_{t\in I}$.  
 
  By the support of a homotopy $H$ of $N$ in $M$ we mean the set $$supp(H):=\{p\in N\,|\,\exists t\in I\,\,H_{t}(p)\neq H_{-1}(p)\}.$$ 
  One can check that $supp(H)$ is an open set and that
  if $H$ is a diffeotopy of $Id_{M}$, then $H(supp(H)\times I)=supp(H)$. 
  We say that $H$ is compactly supported if there is some compact set $K\subseteq M$ such that 
  $supp(H)\subseteq K$. We will use the term proper to refer to isotopies that are proper as continuous maps, i.e., which are such that the preimage of any compact set is compact, rather any other technical meaning.
  
  Given a diffeotopy $G: M\times I\to M$ we write
  $$\ninf{G}:=\sup\{\,d(G(p,t),p)\,|\,p\in M, t\in I\,\}.$$
  
  We let $\D_{c0}(M)$ the collection of elements of $\D(M)$ which are connected to $Id_{M}$ by a compactly supported diffeotopy and $\H_{c0}(M)$ the collection of elements of $\H(M)$ which are connected to $Id_{M}$ by a compactly supported topological ambient isotopy. It is easy to see that $\H_{c0}(M)\leq\H(M)$, and also that $\D_{c0}(M)\leq\D(M)$, by Facts \ref{c: horizontal composition isotopies} and \ref{f: concatenation} below.

   \begin{definition}
   	\label{d: small deformations are dense}We write $\w_{\epsilon}$ for the collection of elements of the form $G_{1}$, where $G:M \times I\to M$ is a diffeotopy of $Id_{M}$ with compact support such that $\ninf{G}<\epsilon$. Given a group topology $\T$ on 
    $\gg\leq\D(M)$ we say that small deformations are close to the identity in $\T$ if 
   	for any $\V\in\nd$ there is $\epsilon>0$ such that $\w_{\epsilon}\subseteq\V$.   
   \end{definition}
      
   \begin{remark}
   	\label{r: small deformations}Clearly, $\w_{\epsilon}\subseteq\V_{\epsilon}$, so that
   	if $\T\subseteq\tc$, then small deformations are close to the identity in $\T$. It follows from \ref{c: horizontal composition isotopies} below that $\w_{\epsilon}=\w_{\epsilon}^{-1}$. 
   \end{remark}
   
  \begin{fact}
  	\label{c: horizontal composition isotopies} If $G$ is a compactly supported diffeotopy, so is $(G_{t}^{-1})_{t\in I}$.
  	If $G^{1}$ and $G^{2}$ are compactly supported isotopies of $M_{1}$ in $M_{2}$ and of $M_{2}$ in $M_{3}$ respectively, then so is $(G^{2}_{t}\circ G^{1}_{t})_{t\in I}$. 
  \end{fact}
  Given a generalized isotopy $H: N\times J\to M$, by a time reparametrization of $H$, we mean an isotopy of the form $H\circ (Id\times\lambda)$, where $\lambda:J\to I$ is an order-preserving smooth bijection. We will say that a property holds for $H$ up to time reparametrization if it holds for some time reparametrization of $H$. If $\lambda$ is a diffeomorphism on the interior of $J$, then we will say that the reparametrization is strict. 
  \begin{fact} [see 2.4 in \cite{wall2016differential}]
	  \label{f: concatenation}Given isotopies $H^{1}$ and $H^{2}$
	  of $N$ in $M$ with $H^{1}_{1}=H^{2}_{1}$, it is possible to replace $H^{i}$ by a strict time reparametrization, so that as a result the map
	  $H:N\times I\to M$ given by
	  	\begin{equation*}
   		H(p,t):=  
   		\begin{cases*}
   		H^{1}(p,2t+1) & if $t\leq 0$ \\
   		 H^{1}(p,2t-1)       & if $t\geq 0$
   	  \end{cases*}
   	\end{equation*}
	  is an isotopy. 
  \end{fact}
  We will refer to $H$ as a concatenation of $H^{1}$ and $H^{2}$, and indicate it by $H^{1}*H^{2}$ (ignoring non-uniqueness). 
 
  We say that an isotopy $F$ of $N$ in $M$ is covered by a diffeotopy $G$ of $Id_{M}$ if $F_{t}=G_{t}\circ F_{-1}$ for all $t\in I$. The following refinement of Thm. 2.4.2 in \cite{wall2016differential} is probably well-known, but we indicate a proof of the sake of completeness.
  \begin{lemma}
  	\label{l: isotopy extension}Let $N,M$ be manifolds of dimensions $m$ and $n$, respectively, and with $\partial M=\emptyset$, and let $H$ be a compactly supported isotopy of $N$ in $M$. Then $H$ is covered by some compactly supported diffeotopy $G$ of $Id_{M}$.
  	Moreover, if $$\epsilon=\sup\{\,\norm{DH\loc{(p,t)}(\partial_{t})}\,\,|\,\,(p,t)\in N\times I\,\},$$ 
  	then for any $\delta>0$ we can assume:
  	\begin{enumerate}
  		 \item \label{support}$supp(G)\subseteq \nn_{\delta}(H(supp(H)\times I))$,
  		 \item \label{slow speed} $G_{t}\in\V_{(1+t)(\epsilon+\delta)}$ for all $t\in I$,
  		 \item \label{general position} if $\partial N=\emptyset$ and 
  		 $N=N_{1}\cap N_{2}=\textstyle\bigcap_{i=1}^{2}(N_{i}\setminus \partial N_{i})$
  		 for two sub-manifolds $N_{1},N_{2}\subseteq M$ in general position, and $H_{t}(N)=N$ for all $t\in I$, then we can assume that $G_{t}(N_{i})=N_{i}$ for all $t\in I$ and $i=1,2$.
  	\end{enumerate}
  \end{lemma} 
  \begin{proof}
    For item (\ref{support}) it suffices to restrict the ambient manifold. Using Whitney's extension theorem \cite{whitney1992analytic} we might assume that $H$ is defined on $M\times\R$, as in the proof of 2.4.2 in \cite{wall2016differential}. We consider the map:  
     $$
  	    \Theta:N\times \R\to M, \quad\quad \Theta(p,t):=(H(p,t),t)
  	 $$
  	 and then define $\chi$ on the submanifold $N^{*}:=\Theta(N\times \R)\subseteq M\times\R$ by the formula:
  	$$\chi(\Theta(p,t))=D\Theta\loc{(p,t)}(\partial_{t}).$$
  	Given an extension $\xi$ of $\chi$ to $M\times\R$ which has constant time component and agrees with $\partial_{t}$ outside $K\times\R$ for some compact set $K$, the diffeotopy $G$ is obtained by integrating $\chi$ and then taking the $M$-component of the result.
    
    For (\ref{slow speed}) it suffices to show that we can take the norm of the $TM$-component $\xi^{M}$ of $\xi$ to be uniformly bounded by $(\epsilon+\delta)$ on $M\times\R$. This can be achieved by multiplying $\xi^{M}$ by a suitable smooth scalar function $\rho:M \times I\to [0,1]$ with support contained in a suitable neighbourhood of $N^{*}$. 
    
    For (\ref{general position}) notice that in this case $N^{*}=N\times\R$ and it suffices to ensure $\xi^{M}(p,t)\in T_{p}N_{i}\subseteq T_{(p,t)}M\times I$ for all $(p,t)\in M\times\R$. 
    Since $N_{1}$ and $N_{2}$ are in general position, we know that there are sets 
    $\mathcal{F}_{1},\mathcal{F}_{2}\subseteq\{1,\dots m\}$, some locally finite covering $\{U_{\lambda}\}_{\lambda\in\Lambda}$ of $M$ 
    and charts $\phi_{\lambda}:U_{\lambda}\cong (-1,1)^{m}$ for $\lambda\in\Lambda$ 
    such that either $U_{\lambda}\cap N=\emptyset$ or $\phi_{\lambda}(N_{l}\cap U_{\lambda})=(-1,1)\cap\{x_{i}=0\,|\,i\in\mathcal{F}_{l}\}$ for $l=1,2$ (Lem. III, 3.1 in \cite{golubitsky2012stable}).
    For every $\lambda$ for which the latter holds a local extension $\xi^{M}_{\lambda}$ of $\xi^{M}$ on $U_{\lambda}$ with the desired properties can be obtained by taking the pull-back in $(-1,1)$ with respect to the linear projection onto $\phi_{\lambda}(N\cap U_{\lambda})$ given by the standard base. The global field $\xi^{M}$ can then be recovered by combining the local extensions using a partition of unity in a standard way.
%
%
%
%
  \end{proof}

    \begin{lemma}
     	 \label{l: supported away} Let $N\subseteq M$ be a submanifold without boundary and suppose that we are given a compactly supported diffeotopy $G$ of $Id_{M}$ such that $(G_{t})_{\restriction N\cup U}=Id_{N\cup U}$ for all $t\in I$, where $U\subseteq M$ is an open set such that $\overline{N}\setminus U$ is compact and contained in $N$. Then for any $\epsilon>0$ there exists some neighbourhood $V$ of $\overline{N}$ in $M$ and some compactly supported diffeotopy $\tilde{G}$ of $Id_{M}$ such that: 
     	 \begin{itemize}
     	 	\item $\tilde{G}_{\restriction V\times I}=G_{\restriction V\times I}$,
     	 	\item $\ninf{\tilde{G}}\leq\epsilon$.
     	 \end{itemize}   
       Therefore, if we write 
     	$g :=G_{1}$ and $h:=\tilde{G}_{1}^{-1}$, then $h\in\w_{\epsilon}$ and $hg$ is isotopic to the identity through a diffeotopy $\hat{G}:=(\tilde{G}_{t}^{-1}G_{t})_{t\in I}$ satisfying $supp(\hat{G})\cap V=\emptyset$.
     \end{lemma}
     \begin{proof}
     	  As in the proof of \ref{l: isotopy extension}, we can extend $G$ to $M\times\R$ and consider the vector field $\chi$ on $M\times\R$ given by
     	  $$\chi(\Theta(p,t))=D\Theta\loc{(p,t)}(\partial_{t}),\quad \Theta:=(G,\pi_{\R}).$$
        Note that we may assume $\chi=\partial_{t}$ outside some compact set. We have 
        \begin{equation*}
        	\label{eq:stability} \tag{$\natural$}\Theta_{\restriction N\times\R}=Id_{N\times\R}\quad\Rightarrow\quad \chi_{\restriction (N\cup U)\times\R}=(\partial_{t})_{\restriction (N\cup U)\times\R}.
        \end{equation*}
         If we let $\chi^{M}$ be the $TM$-component of $\chi$, then (\ref{eq:stability}) implies the existence of some neighbourhood $V'$ of $N$ such that the norm of $\chi^{M}$ (with respect to the Riemannian metric on $M$) is bounded by $\frac{\epsilon}{2}$ on $V'$. Choose some smooth 
        $\rho:M\times\R\to[0,1]$ with the property that $\rho\circ \Theta_{\restriction V''\times\R}=1$
        on some neighbourhood $V''$ of $N$ and $\rho\circ \Theta_{\restriction (M\setminus V')\times\R}=0$ and consider the vector field 
        $\xi:= \rho\chi+(1-\rho)\partial_{t}$ on $M\times\R$.
        Clearly, $D\pi_{\R}(\xi)=\partial_{t}$, and from (\ref{eq:stability}) we can deduce
        \begin{equation*}
        	\label{eq:stability2} \tag{$\natural\natural$}
        	\xi_{\restriction \Theta((V''\cup U)\times\R)}=\chi_{\restriction \Theta((V''\cup U)\times\R)},\quad\xi_{\restriction U\times\R}=(\partial_{t})_{\restriction U\times\R}.
        \end{equation*}
        Since the norm of $\rho\chi^{M}$ is uniformly bounded by $\frac{\epsilon}{2}$, the diffeotopy $\tilde{G}$ of $Id_{M}$ obtained from integrating $\xi$ and restricting to $M\times I$ satisfies 
        $\ninf{\tilde{G}}\leq\epsilon$. It follows from (\ref{eq:stability2}) that $\tilde{G}_{\restriction V\times I}=G_{\restriction V\times I}$ for some neighbourhood $V$ of $\overline{N}$. 
     \end{proof}

  The following result can be extracted from the proof of Prop. 4.4.4 in \cite{wall2016differential},
  where a non-parametric version is formulated under the assumption that $N$ is compact.
  
  \begin{fact}
  	\label{f: close implies isotopic} Let $N,M$ be manifolds with $\partial M=\emptyset$. Assume that $H:N\times I^{k}\to M$, $k\geq0$, is a smooth map with the following properties: 
  	\begin{itemize}
  		\item the map $H_{\underline{s}}:N\to M$ sending $p$ to $H(p,\underline{s})$ is an embedding for all $\underline{s}\in I^{k}$,
  		\item there is a compact set $K\subseteq N$ such that $H(p,\underline{s})$ is independent of $\underline{s}\in I^{k}$ for all $p\in K$.  
  	\end{itemize}
    Then there is some neighbourhood $\mathcal{W}$ of $H$ in $\wt(N\times I^{k},M)$ such that for all 
    $H'\in\mathcal{W}$ such that $H'(p,\underline{s})=H(p,\underline{s})$ for all $p\in N\setminus K$ and $\underline{s}\in I^{k}$ there exists some $\tilde{H}: N\times I^{k+1}\to M$ with the property that for all $(p,\underline{s})\in N\times I^{k}$ we have
    \begin{itemize}
    	\item $\tilde{H}(p,\underline{s},-1)=H(p,\underline{s}),\quad \tilde{H}(p,\underline{s},1)=H'(p,\underline{s})$,
    	\item the curve $t\mapsto H(p,\underline{s},t)$ is a minimizing geodesic segment parametrized with constant speed. 
    \end{itemize} 
  \end{fact}
  
  \begin{corollary}
  	\label{c: main corollary} Let $N,M$, $H: N\times I^{k}\to M$ and $K\subseteq N$ be as in 
  	Fact \ref{f: close implies isotopic}. Then for any $\epsilon>0$ there is some neighbourhood $\W$ of $H$ in 
  	$\wt(N\times I^{k},M)$ such that for any $H'\in\mathcal{W}$
  	with $H'_{(N\setminus K)\times I^{k}}=H_{\restriction (N\setminus K)\times I^{k}}$ there is 
  	$G:M\times I^{k+1}\to M$ such that for all $\underline{s}\in I^{k}$ the map $G[\underline{s}]:M\times I\to M$ given by $G[\underline{s}](p,t):=G(p,\underline{s},t)$ is a 
  	diffeotopy of $Id_{M}$ such that
  	\begin{itemize}
  		\item $G[\underline{s}]$ is supported on  $\nn_{\epsilon}(K)$,
  		\item  $G[\underline{s}]_{1}(H(p,\underline{s}))=H'(p,\underline{s})$,
  		\item $\ninf{G[\underline{s}]}\leq\epsilon$.
  	\end{itemize}
  \end{corollary}
  \begin{proof}
  	For $\W$ small enough we can use \ref{f: close implies isotopic} to find $\hat{H}: (N\times I^{k})\times I\to M$ such that $\hat{H}_{-1}=H$, $\hat{H}_{1}=H'$ and 
  	$$\norm{D\hat{H}\loc{(p,\underline{s},t)}(\partial_{t})}=\frac{1}{2}dist(H(p,\underline{s}),H'(p,\underline{s}))\leq\frac{\epsilon}{2}$$ 
  	for all $(p,\underline{s},t)\in N\times I^{k+1}$.	After suitably extending $\hat{H}$ to $N\times\R^{k}\times I$, we may apply Lemma \ref{l: isotopy extension} to the map $(\hat{H},\pi_{\R^{k}}):N \times\R^{k}\times I\to M\times\R^{k}$ to obtain the desired conclusion.  
  \end{proof}

  \subsection*{Embedded balls}
        
        \newcommand{\cmp}[2]{\mathcal{C}_{#1}(#2)} 
        For $R>0$ and some integer $k>0$ we write: 
        $$
        B^{k}(R)=\{\underline{x}\in\R^{k}\,|\,\norm{x}<R\},\quad \quad \overline{B}^{k}=\{\underline{x}\in\R^{k}\,|\,\norm{x}\leq R\}.
        $$
        We abbreviate $B^{k}(1)$ by $B^{k}$ and similarly for $\overline{B}^{k}$.
        By an embedded $k$-ball (or, simply, a $k$-ball) $D$ in a manifold without boundary $M$ we intend the image of a smooth embedding of $\overline{B}^{k}$ in $M$. By the interior $\mathring{D}$ of $D$ we mean, by an abuse of notation, the image of $B^{k}$ by such an embedding, and by $\partial D$ the image of the standard sphere $S^{m-1}$. 
        
        By applying \cite{wall2016differential}, Thm. 2.5.6 iteratively we get the following.
        \begin{fact}
        \label{f: transitive on disks} Suppose we are given two families of disjoint $k$-balls  $\{D_{i}\}_{i=1}^{k},\{D'_{i}\}_{i=1}^{k}$ in some connected smooth manifold $M$ and a collection of diffeomorphisms $h_{i}:D_{i}\cong D'_{i}$, which we assume to be orientation-preserving if $M$ is orientable. Then there exists $h\in \D_{c0}(M)$ extending all $h_{i}$ simultaneously.
        \end{fact}
        
        An analogous statement holds for the homeomorphism group of a topological manifold as a result of the annulus theorem \cite{moise2013geometric},\cite{Kirby1969},\cite{Quinn1982}. We say that a topological embedding of the standard closed $m$-ball $\overline{B}^{m}(1)$ in an $m$-dimensional topological manifold is collared if it extends to an embedding of $\overline{B}^{m}(2)$. We will refer to the image of such an embedding as a collared ball.  
       
        \begin{fact}
        \label{f: transitive on disks in homeo} Suppose we are given two families of disjoint collared balls  $\{D_{i}\}_{i=1}^{k},\{D'_{i}\}_{i=1}^{k}$ in some connected topological manifold $M$, as well as a collection of homeomorphisms $h_{i}:D_{i}\cong D'_{i}$, which we assume to be orientation-preserving in case $M$ is orientable. Then there exists $h\in \H_{c0}(M)$ extending all $h_{i}$ simultaneously.
        \end{fact}
             
      \subsection*{Embedded simplicial complexes} 
      By an embedded simplicial complex $\Delta$ in $M$ we mean an embedding of the geometric realization of a simplicial complex into $M$ which is smooth on every simplex. For the most part, we will think of this merely as a family $\{\Delta^{k}\}_{k\geq 0}$ of subsets of $M$, where $\Delta^{k}$ stands for the collection of all $k$-simplices of $\Delta$. We also write $\Delta^{(k)}$, or the $k$-skeleton of $\Delta$, for the union $\bigcup_{0\leq l\leq k}\Delta^{l}$, and $\rl{\Delta}$ for the collection of points in the image of $\Delta$.
      By Whitney's extension theorem \cite{whitney1992analytic} and the inverse function theorem, 
      each simplex $\sigma\in\Delta^{k}$ is the image by some smooth embedding $\phi:\overline{B}^{k}\to M$ 
      of some affine simplex $\tau\subseteq B^{k}$.
      
      Given a submanifold $N\supseteq M$ and an embedded simplicial complex $\Delta$ in $M$, we write $N\transv\Delta$ if each of the simplices of $\Delta$ can be extended to an embedded manifold of the same dimension in general position with respect to $N$. 
 
     \begin{definition}
     	  \label{d: stabilizers}If $\Delta$ is a $(m-1)$-dimensional embedded finite simplicial complex in an $m$-dimensional manifold $M$ we denote by $\PS{\hh}{\Delta}$ the subgroup of $\D_{c0}(M)$ consisting of all the elements $g$ of the form $g=G_{1}$ for some compactly supported diffeotopy $G$ of $Id_{M}$ for which there is some neighbourhood $V$ of $\rl{\Delta^{(m-2)}}$ such that for all $t\in I$ we have $(G_{t})_{\restriction V\cup\rl{\Delta}}=Id_{\restriction V\cup\rl{\Delta}}$.
            
       We also denote by $\SS{\hh}{\Delta}$ the subgroup consisting of all elements of the form $G_{1}$ for 
       some compactly supported diffeotopy $G: M\times I\to M$ of $Id_{M}$ such that for some neighbourhood $V$ of $\rl{\Delta^{(m-2)}}$ and every $t\in I$: 
        \begin{itemize}
        \item $G_{t}$ preserves each simplex of $\Delta$ setwise,
        \item $(G_{t})_{\restriction V}=Id_{V}$.
        \end{itemize}
     \end{definition}
        
     \subsection*{Triangulations and spines}
      We will refer to an embedded simplicial complex as a triangulation of $M$ when the embedding is a homeomorphism. An account of the fundamentals of the theory of triangulations of smooth manifolds can be found in \cite{munkres2016elementary} (see also \cite{cairns1961simple} for a short proof of existence). 
      \begin{fact}
      	\label{f: existence of triangulations}Every smooth manifold admits a triangulation, which can be taken to be finite in case the manifold is compact. Moreover, for any given $\epsilon>0$ its simplices can be assumed to have diameter uniformly bounded by $\epsilon$. 
      \end{fact}   
     
     \newcommand{\tr}[0]{\Upsilon}
     
     Given a triangulation $\tr$ of an $m$-dimensional manifold with boundary, we let $\mathcal{G}(\tr)$ be the graph which has $\tr^{m}$ as a set of vertices and an edge between $\sigma$ and $\sigma'$ whenever $\sigma$ and $\sigma'$ share an $(m-1)$-dimensional face.  
     
     \begin{definition}
	     \label{d: retracting forest} By a retracting forest for $\tr$ we mean a finite collection $\mathcal{F}=\{(\Gamma_{i},\tau_{i})\}_{i=1}^{r}$,
	     where $\{\Gamma_{i}\}_{i=1}^{r}$ is a collection of disjoint subtrees of $\mathcal{G}(\tr)$ and $\tau_{i}$ some $(m-1)$-dimensional simplex which lies on $\partial M$ and is a face of one of the simplices in $V(\Gamma_{i})$. 
	     We recycle notation by letting 
	     $$\rl{\Gamma_{i}}:=\bigcup_{\sigma\in V(\Gamma_{i})}\sigma, \quad\quad\rl{\mathcal{F}}:=\displaystyle\bigcup_{i=1}^{r}\rl{\Gamma_{i}}.$$
	     
	     By the $\mathcal{F}$-spine of $\mathcal{F}$ or $\Sp(\mathcal{F})$ we mean the $(m-1)$-dimensional simplicial subcomplex of $\tr$ consisting the $(m-2)$-skeleton of $|\mathcal{F}|$, together with all those $(m-1)$-simplices that are 
	     \begin{itemize}
	     	\item contained in some simplex in $\mathcal{F}$,
	     	\item not among $\{\tau_{i}\}_{i=1}^{r}$ nor dual to an edge in one of the $\Gamma_{i}$. 
	     \end{itemize}
	     We define the thick spine of $\mathcal{F}$ or $\Tsp(\mathcal{F})$ as the union of $\Sp(\mathcal{F})$ and the subcomplex spanned by all the $m$-dimensional simplices of $\Upsilon$ not in $\mathcal{F}$.
     \end{definition}
     
     From the proof of Theorem 2.1 in \cite{whitehead1961immersion} one can extract the following.
     \begin{fact}
     	\label{f: spine} Assume $N$ is a manifold without boundary, $\tr$ some triangulation of $N$,
     	$\mathcal{F}=\{(\Gamma_{i},\tau_{i})\}_{i=1}^{r}$ a retracting forest for $\tr$ and $W$ a neighbourhood of 
     	$\rl{\Tsp(\mathcal{F})}$. 
     	Then there exists some open set $V$, $\rl{\Tsp(\mathcal{F})}\subseteq V\subseteq W$ and some isotopy $H$ of $N$ in itself supported on $N\setminus V$ such that $H_{1}(N)\subseteq W$. 
     \end{fact}
     
    \begin{figure}[t]
   	\includegraphics[width=0.7\textwidth]{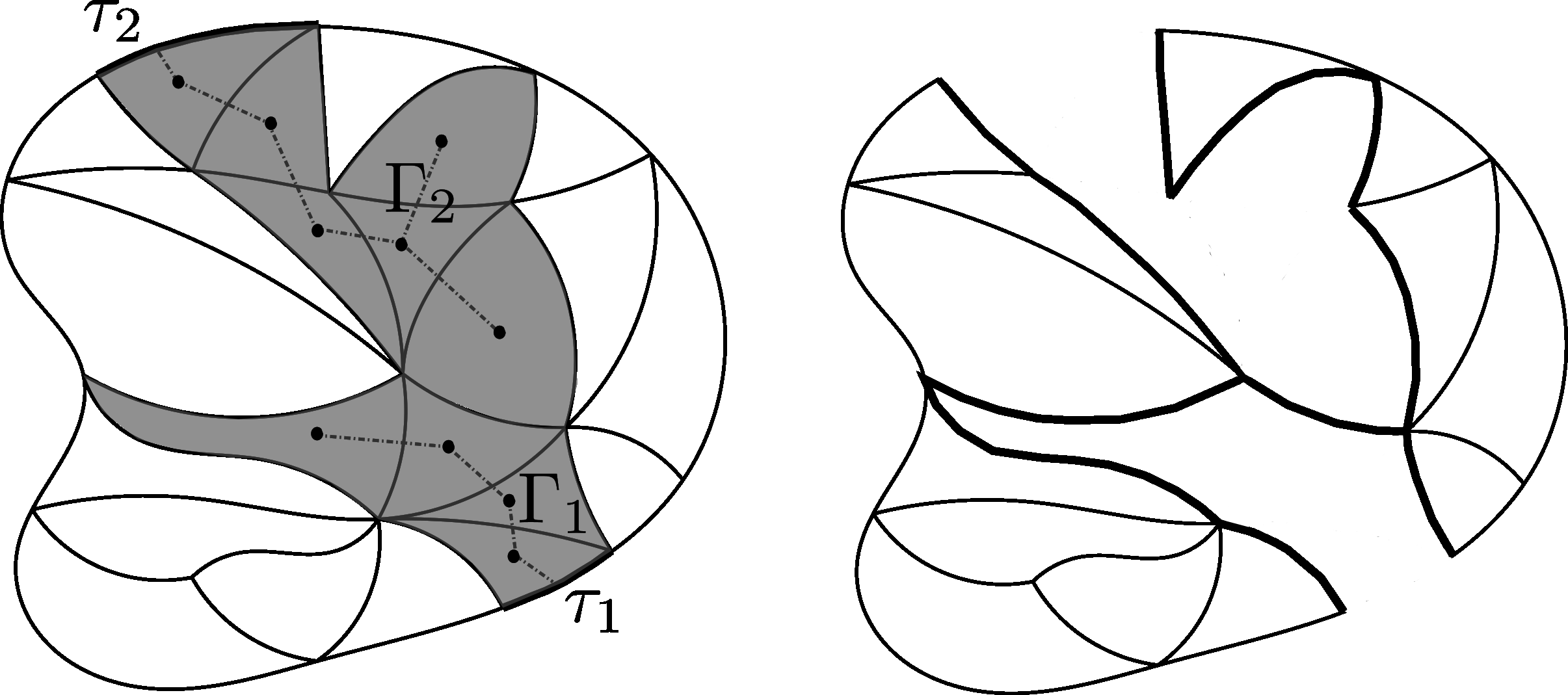}
   	\caption{	\label{fig:spine}On the left a retracting forest $\mathcal{F}=\{(\Gamma_{1},\tau_{1}),(\Gamma_{2},\tau_{2})\}$ for a triangulation of $\Upsilon$ of a $2$-dimensional disk. On the right the $2$-dimensional complex $\Tsp(\mathcal{F})$ which is the image of the retraction, with the $1$-dimensional subcomplex 
   	 $\Sp(\mathcal{F})$ highlighted using a thicker trace.} 
   \end{figure} 
     
     \subsection*{Tubular neighbourhoods} 
     
     By an (open) tubular neighbourhood of a submanifold $N\subseteq M$, $dim(M)=m$, $dim(N)=n$ we mean a bundle
     $\pi:E\to N$ with fiber $\overline{B}^{n-m}$ ($B^{n-m}$), together with an embedding $\xi: E\to M$ whose restriction to the zero section $E_{0}$ is a diffeomorphism with $N$. We will denote this simply by $(E,\xi)$. We will write $E_{p}=\pi^{-1}(p)$ for the fiber over $p$ and refer to $\xi(E_{p})$ as a fiber image. We say that it is trivial if (up to diffeomorphism) $E=N\times B^{d}$, for some $d>0$ and $\pi$ equals the projection on the first factor. 
     
     \begin{fact}
     	Any submanifold $N\subseteq M$ admits a tubular neighbourhood, which is trivial in the connected components of $N$ are simply connected.
     \end{fact}
      
                    
     The following refinement of a particular case of Lem. 2.5.2 and Cor. 2.5.3 in \cite{wall2016differential} is probably well known, but we sketch the proof for the sake of completeness.
     \begin{lemma}
     	\label{l: tubular neighbourhoods} Let $N\subseteq M$ be codimension $1$ manifold without boundary and $(E,\xi)$,$(E,\xi')$ two open tubular neighbourhoods of $N$ in $M$ such that $\xi$ and $\xi'$ agree on $E\setminus\pi^{-1}(K)$ for some compact $K\subseteq N$.
     	Then there exists some isotopy between $\xi$ and $\xi'$ constant on $\pi^{-1}(N\setminus K)\cup E_{0}$. 
     \end{lemma}
     \begin{proof}   	
     A standard partition of unity argument yields a fiber-preserving map $\lambda:E\to E$ which satisfies $\xi'\circ\lambda\subseteq im(\xi)$, restricts to multiplication by a scalar on each fiber and is the identity $\pi^{-1}(N\setminus K)$. The map $\iota:=\xi^{-1}\circ\xi'\circ\lambda$ is also the identity on $\pi^{-1}(N\setminus K)$. As in the proof of 2.5.2 in \cite{wall2016differential}, 
     	if we let $H_{t}(p)=\frac{2}{t+1}\cdot\iota(\frac{t+1}{2}\cdot p)$ for $p\in E$ and $t\in I$, then this extends to an isotopy
     	$(H_{t})_{t\in I}$ of $E$ in itself so that
     	$H_{-1}:E\to E$ is multiplication by some positive scalar.
     	There are obvious fiber-preserving isotopies $H'$ from      	
     	$Id_{E}$ to $H_{0}$ and $H''$ from $\lambda$ to $Id_{e}$ constant on $\pi^{-1}(N\setminus K)$, and the isotopy $(\xi\circ (H'*H))*(\xi'\circ H'')$ satisfies the required conditions.  	
     \end{proof}

    The following Lemma allows us to obtain covering diffeotopies with no control on the derivatives of the given isotopy, provided some other conditions are satisfied. 
   \begin{observation}
   	\label{o: lifting small diffeotopies} Let $N\subseteq M$ a closed submanifold of codimension $1$ which admits a trivial tubular neighbourhood $\xi: N\times I\to M$.   	
   	Then there exists some function $\omega:\R_{>0}\to\R_{>0}$ which satisfies $\lim_{\epsilon\to 0}\omega(\epsilon)=0$ 
   	and has the following property. 
   	Given a diffeotopy $F:N \times I\to N$ of $Id_{N}$ such that $\ninf{F}<\epsilon$, there exists some  diffeotopy $G$ of $Id_{M}$ such  
   	that $G_{\restriction N\times I}=F$ and $\ninf{G}<\omega(\epsilon)$. 
   \end{observation}
   \begin{proof}
    Take any smooth function $\rho:I\to I$ such that 
    \begin{itemize}
    	\item $\rho$ equals $0$ in some neighbourhood of $\partial I=\{-1,1\}$,
    	\item $\rho(0)=1$, 
    \end{itemize}
    and define $G$ as follows: 
    \begin{equation*}
    	G(q,t):=
    	\begin{cases*}
    	 \xi(F(p,(t+1)\rho(s)-1),s) & if $q=\xi(p,s)$ for some $(p,s)\in N\times I$ \\
    	 q                   & if $q\notin im(\xi)$
    	\end{cases*}
    \end{equation*}
    It is easy to check that this is a well-defined compactly-supported diffeotopy of $Id_{M}$ extending 
    $F$. The existence of $\omega$ as in the statement follows easily by continuity. 
   \end{proof}

     \subsection*{Jet spaces and Thom's transversality} We establish some notation and recall some basic facts. More details can be found in Chapter 4 of \cite{wall2016differential}. Given manifolds $N$ and $M$ of dimension $n$ and $m$, respectively, the space of $1$-jets from $N$ to $M$, or $\jet(N,M)$, is the set of triples $(p,q,\lambda)$, where $p\in N$, $q\in M$ and $\lambda$ belongs to the set of linear maps $\mathcal{L}(T_{p}N,T_{q}M)$. We denote by $\pi_{\jet}:\jet(N,M)\to N\times M$ the map sending $(p,q,\lambda)$ to $(p,q)$.
     
     \begin{fact}
	     \label{f: jet coordinates} One can endow $\jet(N,M)$ with the structure of a smooth manifold and in fact a vector bundle over $N\times M$ (with the obvious projection). One does this by constructing for each pair of local charts $\phi=\pmb{x}$ on $U\subseteq N$ and $\psi=\pmb{y}$ on $V\subseteq M$ 
	     a chart $\Psi(\phi,\psi)$ on $\pi^{-1}_{\jet}(U\times V)$ with coordinates: 
	     $$\{x_{i},y_{j},\zeta_{i,j}\,|\,1\leq i\leq m, 1\leq j\leq n\},$$ where $\zeta_{i,j}(p,q,\lambda)$
	     is the $\partial_{y_{j}}$-component of the vector $\lambda(\partial_{x_{i}})$ with respect to the base
	     $\{\partial_{y_{i}}\}_{i=1}^{m}$.
     \end{fact}

     Any $f\in C^{\infty}(N,M)$ induces a map 
     $$j^{1}f:N\to J^{1}(N,M), \quad \quad p\mapsto j^{1}f(p):=(p,f(p),Df\loc{p}).$$  
     \begin{fact}
     	\label{f: local jets} The map $j^{1}f$ is smooth. Given local coordinates $\phi=\pmb{x}$ around $p\in N$ and $\psi=\pmb{y}$ around $q=f(p)\in M$, with respect to $\phi$ and $\Psi(\phi,\psi)$ we have the local expression:
     	$$
     	Dj^{1}f\loc{\underline{x}}(\partial_{x_{i}})=\sum_{j=1}^{m}\frac{\partial f_{j}}{\partial x_{i}}(\underline{x})\partial_{y_{j}}+\sum_{j=1}^{m}\sum_{l=1}^{n}\frac{\partial^{2}f_{j}}{\partial x_{i}\partial x_{l}}(\underline{x})\partial_{\zeta_{l,j}},
     	$$ 
     	where $f_{j}:=y_{j}\circ f\circ\phi^{-1}$ is the $j$-th component of the expression of $f$ with respect to $\phi$ and $\psi$.
     \end{fact}
     
     We will use the following refinement of the first order approximation to Thom's transversality theorem (see Thm. 4.5.6 in \cite{wall2016differential} or Ch.3, Thm. 2.5 in \cite{hirsch2012differential}).
      \begin{fact}[\cite{michor2011manifolds}, Thm. 6.8, p.55]
      	\label{f: Thom}Let $W\subseteq \jet(N,M)$ be a submanifold. Then the set of $f\in C^{\infty}(N,M)$ such that $j^{1}f\transv W$ is comeager and therefore dense in the $\mathcal{D}$-topology. In particular, the set of $f\in C^{\infty}(N,M)$ such that $f\transv W$ for a submanifold $W\subseteq N$ is comeager.
      \end{fact}
      
      From this and items (\ref{whitney baire}) and (\ref{d infinity}) in Fact \ref{f: properties Whitney} one easily obtains the following:    
      \begin{corollary}
      	\label{c: refined transversality} Let $N,M$ be manifolds, $N'\subseteq N$ a codimension $0$ submanifold containing $\partial N$  and
      	$W_{i}\subseteq\jet(N,M)$ for $1\leq i\leq r$ a submanifold. Let also $f\in C^{\infty}(N,M)$ be such that
      	 $(j^{1}f)(N)\cap W_{i}=\emptyset$ for $1\leq i\leq r$. Then for any neighbourhood $\W$ of $f$ in $\wt(N,M)$ there exists $f'\in\W$ such that $j^{1}f'\transv W_{i}$ for all $1\leq i\leq r$ and $f'_{\restriction N'}=f$. 
      \end{corollary}
     
      We record the following useful corollary of \ref{c: main corollary} and \ref{c: refined transversality}.
      \begin{lemma}
      \label{l: wlog transverse} Let $M$ be a manifold without boundary and $N,N_{1},\dots N_{r}\subseteq M$ closed submanifolds. Assume that we are given $\epsilon>0$ such that
      $K:=N\cap\bigcup_{i=1}^{r}N_{i}$ is compact and $\nn_{\epsilon}(K)\cap\partial N=\emptyset$. 
      Then there is $g\in\D_{c0}(M)$ such that
      \begin{itemize}
      	\item $g=G_{1}$ for a diffeotopy $G$ of $Id_{M}$ supported on $\nn_{\epsilon}(K)$ and with $\ninf{G}\leq\epsilon$,
      	\item $N_{i}\transv g\cdot N$ (equivalently, $g^{-1}\cdot N_{i}\transv N$) for all $1\leq i\leq r$.
      \end{itemize}
     \end{lemma}
     \begin{proof}
      Choose a codimension $0$ submanifold $N'\subseteq N$ such that 
      $$K\subseteq N\setminus N'\subset \overline{N\setminus N'}\subseteq \nn_{\frac{\epsilon}{2}}(K)$$  
       and let $\iota:N\to M$ be the inclusion. 
      By \ref{c: refined transversality}, given any neighbourhood 
      $\W$ of $\iota$ in $C^{\infty}(N,M)$ there is $f\in\W$ such that $f_{\restriction M\setminus N'}=\iota_{\restriction M\setminus N'}$ and $f\transv N_{i}$ for all $1\leq i\leq r$. The conclusion follows from Corollary \ref{c: main corollary} by taking $\W$ small enough.
     \end{proof}
     \subsection*{Mather's stability}    
       
       A basic account of the different notions of stability for smooth manifolds and the equivalences between them can be found in Chapter $\text{V}$ of \cite{golubitsky2012stable}. 
       In particular, we recall the concept of stability under deformations, discussed in V, 2.1 of said reference using somewhat different language. 
       
       \begin{definition}
	       \label{d: stable under deformations}Let $F\in C^{\infty}(N\times J,M)$, where $J\subseteq\R$ is an open interval. We say that $F$ is locally trivial around $s\in J$ if there are $\epsilon>0$ and generalized diffeotopies 
	       $G^{N}: N\times[-\epsilon,\epsilon]\to N$, $G^{N}_{0}=Id_{N}$ and $G^{M}: M\times[-\epsilon,\epsilon]\to M$, $G^{M}_{0}=Id_{M}$ 
	       such that $F_{s+t}=G^{M}_{t}\circ F_{s}\circ G^{N}_{t}$ for all $t\in[-\epsilon,\epsilon]$.  
         We say that $f\in C^{\infty}(N,M)$ is stable under deformations if every $F\in C^{\infty}(N\times J,M)$ as above with $f=F_{s}:=F(-,s)$ is locally trivial around $s$.
       \end{definition}
       
       \begin{remark}
       	\label{r: different basepoint}Note that in the definition of local triviality one can write $F_{t}=G^{M}_{t-s}(G^{M}_{-\epsilon})^{-1}\circ F_{s-\epsilon}\circ(G^{N}_{-\epsilon})^{-1}\circ G^{N}_{t-s}$ for $t\in[s-\epsilon,s+\epsilon]$.
       \end{remark}
                    
       \begin{lemma}
       	\label{l: deformation stability} Let $H$ be a compactly supported smooth homotopy of $N$ in $M$ such that $H_{t}$ is stable under deformations for all $t\in I$. Then, up to time reparametrization of $H$, we can assume there exist compactly supported diffeotopies $G^{M}$ of $Id_{M}$ and $G^{N}$ of $Id_{N}$ such that $H_{t}=G^{M}_{t}\circ H_{-1}\circ G^{N}_{t}$ for all $t\in I$.       	 
        \end{lemma}
       \begin{proof}
       	By inspecting the proof of Prop. V, 4.3 in \cite{golubitsky2012stable} one can see that it also applies when $M$ is non-compact case provided $H$ is proper and compactly supported and that in that case the diffeotopies in Definition \ref{d: stable under deformations} can be assumed to be compactly supported.\footnote{The vector field which solves the local problem can be taken to be zero on charts on which the datum is also zero.} After reparametrizing, we might assume that $H$ is defined on (extends to) $N\times J$ for some open interval $J$ containing $I$. 
       	A compactness argument and Remark \ref{r: different basepoint} yield $-1=t_{0}<t_{1}\dots <t_{k}=1$ 
       	and compactly supported generalized diffeotopies
       	\begin{itemize}
       	  \item $(G^{M,i})_{t\in[0,t_{i+1}-t_{i}]}$ with $G^{M,i}_{0}=Id_{M}$ and $g_{i}:=G^{M,i}_{1}\in\D_{c0}(M)$,
       	  \item and $(G^{N,i})_{t\in[0,t_{i+1}-t_{i}]}$ with $G^{N,i}_{0}=Id_{N}$ and $h_{i}:=G^{N,i}_{1}\in\D_{c0}(N)$, 
       	\end{itemize}
        for $1\leq i\leq k-1$, with the property that $H_{t_{i}+s}=G^{M,i}_{s-t_{i}}\circ H_{t_{i}}\circ G^{N,i}_{s-t_{i}}$ for $s\in [0,t_{i+1}-t_{i}]$. 
       	For $s\in [t_{i},t_{i+1}]$ consider the diffeomorphisms 
       	$$G^{M}_{s}:=G^{M,i}_{s-t_{i}}g_{i-1}\cdots g_{0},\quad\quad G^{N}_{s}:=h_{0}\cdots h_{i-1}G^{N,i}_{s-t_{i}}.$$
        It follows from Fact \ref{f: concatenation} that the parametrized families $(G^{M}_{s})_{s\in I}$ and $(G^{N}_{s})_{s\in I}$ become diffeotopies after a suitable simultaneous time reparametrization. 
       \end{proof}
        
        The most basic consequence of Lemma \ref{l: deformation stability} is Lemma \ref{l: stability application} below. 

       \begin{definition}
   	\label{d: manifoldstabilizer} Given a manifold $M$ and a closed set $F\subseteq M$, we write 
  $\D_{c0}^{F}(M)$ for the subgroup of all those diffeomorphisms of $M$ diffeotopic to the identity by a compactly-supported isotopy fixing $F$ at any point in time and we write $\D_{c0}^{\{F\}}(M)$ for the collection of all $g\in\D(M)$ that are diffeotopic to the identity through a compactly supported diffeotopy that preserves $F$ setwise at every point in time. We extend this notation further, writing $$\D_{c0}^{F,\{L_{1},L_{2}\}}(M):=\D_{c0}^{F}(M)\cap\D_{c0}^{\{L_{1}\}}(M)\cap\D_{c0}^{\{L_{2}\}}(M)$$
  and so on.
   \end{definition} 
    
     \begin{lemma}
   		\label{l: stability application}Let $M$ an $m$-manifold without boundary, $L$ a closed codimension $1$ submanifold of $M$. Let also $N$ be an $(m-1)$-dimensional manifold, $\iota$ a proper embedding of 
   		$N$ in $M$ and $G$ a compactly supported diffeotopy of $Id_{M}$ such that $H_{t}:=G_{t}\circ\iota\transv L$ for all $t\in I$.	Then, up to time reparametrization of $G$, there are compactly supported diffeotopies $G^{L}$, $G^{N}$ of $Id_{M}$ such that
   		\begin{itemize}
   			\item $G^{L}_{t}(L)=L$, $G_{t}^{N}(\iota(N))=\iota(N)$ for all $t\in I$,
   			\item $G_{t}=G^{L}_{t}G^{N}_{t}$ for all $t\in I$. 
   		\end{itemize}
   	  In particular, $G_{1}\in\D^{\{L\}}_{c0}(M)\D^{\{\iota(N)\}}(M)$.
   	\end{lemma}
   	\begin{proof}
   		Let $\kappa$ denote the inclusion of $L$ in $M$.  
      Consider the manifold $\hat{N}:=N \coprod L$ and the smooth homotopy $\hat{H}: \hat{N}\times I\to M$ given by $(\hat{H}_{t})_{\restriction L}=\kappa$ and $(\hat{H}_{t})_{\restriction N}=H_{t}$.     
      It follows from III, Thm. 3.11 and V, Thm. 4.3 in \cite{golubitsky2012stable} that proper immersions with normal crossings are stable under deformations. This applies to $\hat{H}_{t}$ for all $t\in I$ (compare Def. III, 3.1 in \cite{golubitsky2012stable} with Lem. 4.6.5 in \cite{wall2016differential}).
      
      By Lemma \ref{l: deformation stability}, we may thus assume the existence of a compactly supported diffeotopy $G^{L}$ of $Id_{M}$ and a compactly supported diffeotopy $G^{\hat{N}}$ of $Id_{\hat{N}}$ such that 
     $\hat{H}_{t}=G^{L}_{t}\circ\hat{H}_{-1}\circ G^{\hat{N}}_{t}$ for all $t\in I$. This last equality restricts to $\kappa=G^{L}_{t}\circ(\kappa\circ G^{\hat{N}}_{t})_{\restriction L}$, which implies that $G^{L}_{t}$ preserves $L=im(\kappa)$ setwise. It also restricts to 
     $G_{t}\circ\iota=G^{L}_{t}\circ\iota\circ(G^{\hat{N}}_{t})_{\restriction N}$, which implies that 
      the diffeotopy $G^{N}:=((G^{L}_{t})^{-1}G_{t})_{t\in I}$ preserves $N$ setwise at every point in time. Then $G_{1}=G^{L}_{1}G^{N}_{1}\in\D_{c0}^{\{L\}}(M)\D_{c0}^{\{\iota(N)\}}(M)$.
   	\end{proof}
    
	     \subsection*{The Thom-Boardman stratification}
	     
	     We will recall the bare minimum of the theory needed here. A short introduction can be found in chapter VI of \cite{golubitsky2012stable} and a more detailed one in 
	     \cite{arnold2014singularities}.
	     Let $N,M$ be manifolds without boundary and $f\in C^{\infty}(N,M)$. 
	     Let $k=\min(dim(M),dim(N))$. For any integer $r\geq 0$ one can consider the collection 
	     $$S_{r}(f)=\{p\in M\,|\,rk(Df\loc{p})=k-r\}\subseteq M.$$
	     If $S_{r}(f)$ is a submanifold of $N$, one may consider for any $s\geq 0$ the collection 
	     $S_{r,s}(f)=S_{s}(f_{\restriction S_{r}(f)})$ and, if $S_{r.s}(f)$ happens to be a submanifold of $S_{r}(f)$, continue the process further. A result of Boardman \cite{boardman1967singularities} states that generically this process can be continued until the dimension of the ambient manifold is exhausted and the submanifold in question is forced to be empty.
	    \begin{fact}
	    	\label{f: boardman}Let $S_{\varnothing}(f)=N$ and $\mathcal{I}$ the collection of all sequences of positive integers including the empty sequence $\varnothing$. There exists some comeager (and hence dense) set $\mathcal{B}\subseteq C^{\infty}(N,M)$ such that 
	    	for any $f\in\mathcal{B}$, any $\underline{i}\in\mathcal{I}$ and any $j>0$ the set $S_{\underline{i}, j}(f)$ is a proper submanifold of $S_{\underline{i}}(f)$, so that $S_{\underline{i}, 0}(f)=S_{\underline{i}}(f)\setminus\bigcup_{j>0}S_{\underline{i}, j}(f)$ is a non-empty open subset of $S_{\underline{i}}(f)$.
	    \end{fact}
      \begin{remark}
      	\label{r: boardman closure} Although for $f\in\mathcal{B}$ the set $S_{\underline{i}}(f)$ is not necessarily closed, 
      	lower semicontinuity of the rank of the differential ensures that 
      	$\bigcup_{l\geq k}S_{\underline{i},l}(f)$ is always closed in $S_{\underline{i}}(f)$.  
      \end{remark}
    
	     As a matter of fact, we are only interested in a very specific case.
	     \begin{corollary}
	     	\label{c: projecting simplicial complexes}Let $f$ be an embedding of a compact manifold $L'$ into a manifold of the form $M\times (0.1)$, let $\pi_{M},\pi_{(0,1)}$ be the factor projections and let $L:=L'\setminus\partial L$. Then for any $\epsilon>0$ there is some $g\in\w_{\epsilon}\cap\D_{co}(M\times(0,1))$ such that if we let $h=\pi_{1}\circ g\circ f$, then there exists some sequence $L=L_{0}\supsetneq L_{1}\supsetneq\dots L_{n}\supsetneq L_{n+1}=\emptyset$ where
	     	$L_{i+1}$ is a closed sub-manifold of $L_{i}$ of positive codimension and $h_{\restriction L_{i}\setminus L_{i+1}}$ an immersion.     
	     \end{corollary}
	     \begin{proof}
	        By Fact \ref{f: properties Whitney} there is a neighbourhood $\mathcal{W}$ of $\pi\circ f$ in $\wt^{\infty}(L,M)$ such that for all $h\in\mathcal{W}$ the differential of $h$ has rank at least $dim(L)-1$ at all points and the map $(h,\pi_{2}\circ f):L\to M$ is an embedding. By Corollary \ref{c: main corollary} we may assume that for any $h\in\mathcal{W}$ the map $(h,\pi_{2}\circ f)$ is of the form 
	        $g\circ f$ for some $g\in\D_{c0}(M)$. Moreover, given any $\epsilon>0$ we can find $g$ in $\w_{\epsilon}$, provided that $\mathcal{W}$ is small enough. If $\mathcal{B}$ is as in \ref{f: boardman}, then for $h\in\mathcal{W}\cap\mathcal{B}$ we have $S_{\underline{i}}(h)\neq\emptyset$ only if no entry of $\underline{i}$ is greater than one. If we let $L_{k}:=S_{\underline{1}_{k}}$, where $\underline{1}_{k}$ stands for the sequence consisting of $k$ ones, then this implies that the restriction of $h$ to $L_{k}\setminus L_{k+1}$ is an immersion.
	        Finally, $L_{k+1}$ is closed in $L_{k}$ by Remark \ref{r: boardman closure}.
	     \end{proof}
	      
       \section{A Morse-like lemma for isotopies and embedded hypersurfaces}
  \label{s: Morse} The following two sections present general results in which $\T$ and $\gg$ play no role.   \begin{definition}
  	By a nice $m$-triple $(N,M,L)$ we mean a triple of smooth manifolds of dimensions $m-1$, $m$ and $m-1$ respectively, where $L$ is a closed submanifold of $M$. 
  \end{definition}
  
  \begin{definition}
  	\label{d: compatible isotopy}Let $(N,M,L)$ be a nice $m$-triple and $H: N\times I\to M$ a (generalized) isotopy of $N$ in $M$. We say that $H$ is $L$-compatible if $H$ is proper, it has compact support there, is a (possibly empty sequence) $-1<t_{1}<t_{2}<\dots t_{r}<1$ of times, as well as points $p_{1},\dots p_{r}\in N$ (the latter not necessarily distinct) such that:
  	\begin{enumerate}[(a)]
  		\item \label{kissing points} for $(p,t)\notin\{(p_{i},t_{i})\}_{i=1}^{r}$
  		we have $H(p,t)\notin L$ or  $DH_{t}(T_{p}N)\neq T_{H(p,t)}L$, 
  		\item \label{local isotopy model} for any $1\leq i\leq r$ there is some $\delta_{i}>0$ and a system of local coordinates $\phi_{i}=\pmb{x}:U_{i}\to\R^{m-1}$ around $p_{i}$ and 
  		$\psi_{i}:=\pmb{y}:V_{i}\to\R^{m}$ around $H(p_{i},t_{i})\in M$ with respect to which $H$ can be written as: 
  		$$
  		(\underline{x},t)\mapsto (\underline{x},t-t_{i}+\sum_{j=1}^{m-1}\epsilon_{i}x_{j}^{2})
  		$$ 
  		for $t\in (t-\delta_{i},t+\delta_{i})$ and some constants $\epsilon_{i}\in\{\pm 1\}$. 
  		Notice that, in particular, $T_{H(p_{i},t_{i})}L=im(DH\loc{(p_{i},t_{i})})$.
  	\end{enumerate}
    In the situation above we refer to each $t_{i}$ as a \emph{tangency time} and to  $(p_{i},t_{i})$ as a \emph{tangency point} of $H$ (the connection with $L$ will be always clear from the context). 
    If $H$ is $L$-compatible and the set of tangency points is empty, we say that $H$ is $L$-clean.
    We also let 
    the \emph{dimension} of $(p_{i},t_{i})$ be the number $d_{i}$ of values of $i$ such that $\epsilon_{i}=1$. It is easy to see that this does not depend on the chart.     
    By the \emph{type} of $H$ we mean the sequence of dimensions $(d_{1},\dots d_{r})$. 
    If all the conditions above are satisfied, but we allow the possibility that $t_{i}=t_{j}$ for distinct $i,j$, then we say that $H$ is weakly $L$-compatible.
  \end{definition}
  
   \begin{observation}
  	\label{o: tangency not perturbed} By using the fact that each $H_{t}$ is a proper embedding for all $t$, a standard argument involving the extraction of convergent subsequences also shows one can assume  
  	$H^{-1}(V_{i})=U_{i}\times(t_{i}-\delta_{i},t_{i}+\delta_{i})$ in the definition above.   
  \end{observation}
  \begin{observation}
  	\label{o: compatibility and strictness} The property of being $L$-compatible, as well as the type, are preserved by strict time reparametrizations.
  \end{observation}
 
  \begin{lemma}
  	\label{l: compatible redefinition} Let $(N,M,L)$ be a nice $m$-triple and $H$ a proper compactly supported isotopy of $N$ in $L$. Assume that there is a finite family $\mathcal{F}\subseteq N\times (-1,1)$ of points with different time coordinates such that
  	 $DH_{t}(T_{p}N)\subseteq T_{H(p,t)}L$ if and only if $(p,t)\in\mathcal{F}$.
  	 Assume, moreover, that for all $(p_{0},t_{0})\in\mathcal{F}$ there are charts $\phi=\pmb{x}:U\to\R^{m-1}$ around $p_{0}$ and $\psi=\pmb{y}:V\to\R^{m}$ around $q_{0}:=H(p_{0},t_{0})$ such that 
  	\begin{itemize}
  		\item $\psi(L\cap V)=\psi(V)\cap\{\,y_{m}=0\,\}$,
  		\item $\phi(p_{0})=\underline{0}$, $\psi(q_{0}):=\underline{0}$,
  	\end{itemize}
  	 and the local expression $F:=\psi\circ H\circ(\phi^{-1},Id_{I})$ for $H$ satisfies
  	 the following conditions, where we use $F^{m}$ to denote the $m$-th component of $F$:
  	 \begin{enumerate}
     	\item \label{first differentiald}$\frac{\partial F^{m}}{\partial t}\big|_{(\underline{0},t_{0})}\neq 0$, 
     	\item \label{second differentiald}the matrix $\left(\frac{\partial^{2}F^{m}}{\partial x_{i}\partial x_{j}}\big|_{(\underline{0},t_{0})}\right)_{\substack{1\leq i\leq m-1\\ 1\leq j\leq m-1}}$ is non-singular.
     \end{enumerate}
     Then $H$ is weakly $L$-compatible (see end of Definition \ref{d: compatible isotopy}) and $\mathcal{F}$ is the collection of its tangency points. The dimension of $(p_{0},t_{0})$ above coincides with the dimension $d$ of the maximal subspace on which the bilinear form determined by the matrix in (\ref{second differentiald}) is positive definite.
  \end{lemma}
  \begin{proof}
  	 Denote by $\pi,\pi_{m}$ the projection of $\R^{m}$ onto the first $m-1$ and last coordinate, respectively.
  	 Notice that our assumptions on $F$ and $\psi$ also imply
     \begin{enumerate}
     	\setcounter{enumi}{2}
     	\item \label{horizontal} $DF^{m}_{t_{0}}\loc{\underline{0}}=\underline{0}$, and thus $D(\pi\circ F_{t_{0}})\loc{\underline{0}}$ is non-singular, where  $F_{s}:=F(-,s)$.
     \end{enumerate} 
     
     \begin{lemma}
       \label{l:straigthening time}There is some local diffeomorphism $\Omega$ of $\R^{m}$ around $\underline{0}$   fixing the level set $\{y_{m}=0\}$ in a neighbourhood of $\underline{0}$ and such that 
       $\tilde{F}:=\Omega\circ F$ is of the form $\tilde{F}(\underline{x},t)=\tilde{F}(\underline{x},t_{0})+(\underline{0},t-t_{0})$ around $(p,t_{0})$.	  
     \end{lemma}
    \begin{subproof}
    	
     Conditions (\ref{first differentiald}) and (\ref{horizontal}) above imply that 
     for some open interval $J$ containing $t_{0}$ the map
     $F$ restricts to a diffeomorphism between a neighbourhood of $(\underline{0},t_{0})$ in $U\times J$ and a neighbourhood of $\underline{0}$ in $V$.      
     This allows one to define a vector field $\xi$ in some neighbourhood of $(\underline{0},t_{0})\in\R^{m}$
     by the formula  $\xi(F(\underline{x},t)):=DF\loc{(\underline{x},t)}(\partial_{t})$. Integrating $\xi$ results in a smooth flow $\Theta: V'\times (-\epsilon,\epsilon)\to \R^{m}$, with $\Theta(\underline{y},0)=\underline{y}$ for all $\underline{y}\in V'$. Where $V'$ is some neighbourhood of $\underline{0}$ in $\R^{m}$ and $\epsilon>0$ satisfies $(t_{0}-\epsilon,t_{0}+\epsilon)\subseteq J$. Let $V'_{0}=\{\underline{x}\in\R^{m-1}\,|\,(\underline{x},0)\in V'\}$. 
     
     Consider the map
      $\Omega^{*}:V'_{0}\times(-\epsilon,\epsilon)\to\R^{m}$ given by:
     $$
     \Omega^{*}(y_{1},\dots y_{m-1},t):=\Theta(y_{1},\dots y_{m-1},0,t)\in\R^{m}.
     $$ 
     Notice that $D\Omega^{*}\big|_{\underline{0}}$ is non-singular, since
     \begin{align*}
     D\Omega^{*}\big|_{\underline{0}}(\partial_{y_{i}})&=\partial_{y_{i}},\quad 1\leq i\leq m-1, \\
     D(\pi_{m}\circ\Omega^{*})\big|_{\underline{0}}(\partial_{t})&=d\pi_{m}(\xi(\underline{0}))=d\pi_{m}(\frac{\partial F}{\partial t}(\underline{0},t_{0}))=\frac{\partial F^{m}}{\partial t}(\underline{0},t_{0})\neq 0.
     \end{align*}  
     
      For  $(\underline{x},t)$ sufficiently close to $(\underline{0},t_{0})$ there is
      $(\underline{x}',t')$ with $\Omega^{*}(\underline{x}',t')=F(\underline{x},t)$ and the following equality holds:
      \begin{align*}
      	D\Omega^{*}\big|_{(\underline{x}',t')}(\partial_{t})=
	      D\Theta\big|_{(\underline{x}',0,t')}(\partial_{t})=\xi(\Theta(\underline{x}',0,t'))=\xi(\Omega^{*}(\underline{x}',t'))=DF\big|_{(\underline{x},t)}(\partial_{t}).
      \end{align*}
      If we let $\Omega$ be a local inverse for $\Omega^{*}$ in some neighbourhood of $\underline{0}$, then 
      the equality above implies that $D(\Omega\circ F)(\partial_{t})=\partial_{t}$ holds locally around $(\underline{0},t_{0})$, so that $\Omega\circ F$ has the desired form in some neighbourhood of $(\underline{0},t_{0})$.
      Notice also that $\Omega$ fixes the level set $\{y_{m}=0\}$, since $\Omega^{*}$ does (by definition).
    \end{subproof}
     The fact that $\Omega$ fixes $\{y_{m}=0\}$ locally around $\underline{0}$ implies, by an easy calculation, that replacing $F$ by $\tilde{F}$ does not affect the validity of conditions (\ref{first differentiald}) to (\ref{horizontal}),
     nor changes the signature of the bilinear form associated to the matrix in (\ref{second differentiald}).

     Clearly, the property established the sublemma above does not depend on the choice of $\phi$.
     We can, therefore, assume $F$ is of the form $\underline{x}\mapsto(\underline{x},f(\underline{x}))$.
     By Morse's Lemma (\cite{wall2016differential}, Proposition 4.8.2) we may assume, furthermore, that
     the map $f$ is of the form 
     $f(\underline{x})=\sum_{i=1}^{d}x_{i}^{2}-\sum_{i=d+1}^{m-1}x_{i}^{2}$  in some neighbourhood of $\underline{0}$, as needed.
  \end{proof}
  
  We are now ready to state the main result of this section.

  \begin{lemma}
  	\label{l: morse for isotopies}Let $(N,M,L)$ be a nice $m$-triple and $H: N\times I\to M$ a smooth proper isotopy with compact support of $N$ in $M$ 
  	such that for some compact set $C\subseteq N$ we have $H((N\setminus C)\times I)\cap L=\emptyset$ for any $(p,t)\in(N\setminus C)\times I$.
  	Then for any neighbourhood $\mathcal{W}$ of $H$ in $\wt(N\times I,M)$ and any neighbourhood $W$ of $C$ in $N$ there exists some $L$-compatible isotopy $H'\in \mathcal{W}$ such that $H_{\restriction (N\setminus W)\times I}=H'_{\restriction (N\setminus W)\times I}$. 
  \end{lemma}
  \begin{proof}	
  	\newcommand{\pl}[0]{P_{L}}
  	\newcommand{\jj}[0]{J}
  	By Whitney's extension theorem, we may assume $H$ is defined on $N\times\jj$, where $\jj$ is an open interval containing $I$. It is possible that the extension $\overline{H}$ given by the theorem does not satisfy $H^{-1}(L)\subseteq (N\setminus C)\times J$ for some compact set on the nose, but using the fact that $\overline{H}^{-1}(L)$ is closed one can always find a neighbourhood $U$ of $N\times I$ and a diffeomorphism $\theta:U\to N\times J'$, $\theta_{N\times I}=Id_{N\times I}$ such that the property holds for $\overline{H}\circ\theta$. 
  	Consider the subset of $\jet(N\times J,M)$ given by
  	$$
  	\pl:=\{\,((p,t),q,\lambda)\,|\,q\in L,\,\,\lambda(T_{(p,t)}(N\times\{t\}))\subseteq T_{q}L\,\}.
  	$$ 
  	
  	\begin{observation}
  		\label{o: subspace of jets}
  		Let $((p_{0},t_{0}),q_{0})\in(N\times J)\times M$, $\phi=\pmb{x}:U\to\R^{m-1}$ be a chart around $p_{0}$ and $\psi=\pmb{y}:V\to\R^{m}$ be a chart around $q_{0}$ such that 
  		$$\psi(V\cap L)=\psi(V)\cap\{y_{m}=0\}.$$
  		
  		We write $\Psi(\phi\times Id,\psi)$ for the chart on 
  		$\pi_{\jet}^{-1}(U\times J\times V)$ with coordinates:
  		$$(x_{1},\dots x_{m-1},t,y_{1},\dots y_{m},\zeta_{i,j})_{1\leq i,j\leq m},$$
  		as described in \ref{f: jet coordinates}, where $t$ plays the role of $x_{m}$ in the definition of $\zeta_{m,j}$.
  		 
  		Then with respect to this chart the set $\pl\cap\pi_{\jet}^{-1}(U\times J\times V)$ is described by the linear equations $$y_{m}=0, \quad\quad \zeta_{i,m}=0, \quad\quad 1\leq i\leq m-1.$$
  		 In particular, $\pl$ is a submanifold of $\jet(N\times J,M)$ of codimension $m$.
  	\end{observation}

  	By Corollary \ref{c: refined transversality} and the exhaustion of smooth manifolds by codimension $0$ submanifolds, for any neighbourhood $\mathcal{W}$ of $H$ in the Whitney topology there exists some $H'\in\mathcal{V}$ such that $j^{1}H'\transv \pl$ and $H$ and $H'$ agree outside of $W\times J$. The set $\mathcal{F}:=(j^{1}H')^{-1}(\pl)$ is a codimension $m$ submanifold of $N\times J$, i.e., a discrete set of points in $N\times J$.
  	
     Let us now examine the meaning of the transversality condition at a point 
     $(p_{0},t_{0})\in \mathcal{F}$. Write $q_{0}:=H'(p_{0},t_{0})$ and 
     $r_{0}:=j^{1}H'(p_{0},t_{0})$.
     Let
     $\phi=\mathbf{\underline{x}}: U\to\R^{m-1}$
     and $\psi=\pmb{y}:V\to\R^{m}$ be as in Observation \ref{o: subspace of jets},
     and assume that $\psi(q_{0})=\underline{0}$, $\phi(p_{0})=\underline{0}$.\footnote{Notice that the meaning of $\underline{0}$ depends on the ambient space.} Let $J_{0}$ be some interval around $t_{0}$ such that $H'(U\times J_{0})\subseteq V$ and write  
     $$F:=\psi\circ H'\circ(\phi^{-1}\times Id_{J_{0}}):\phi(U)\times J_{0}\to \psi(V).$$
     
     By Observation \ref{o: subspace of jets} and the expression for the differential of $h:=j^{1}H'$ provided in Fact \ref{f: local jets}, the condition $h\transv_{r}\pl$ at $(p_{0},t_{0})$ translates into the following two conditions on the $m$-th component $F^{m}$ of $F$ at $(\underline{0},t_{0})$:
     \begin{enumerate}
     	\item \label{first differential}$\frac{\partial F^{m}}{\partial t}\big|_{(\underline{0},t_{0})}\neq 0$, 
     	\item \label{second differential}the matrix $\left(\frac{\partial^{2}F^{m}}{\partial x_{i}\partial x_{j}}\big|_{(\underline{0},t_{0})}\right)_{\substack{1\leq i\leq m-1\\ 1\leq j\leq m-1}}$ is non-singular, by (\ref{first differential}) and the fact that $\frac{\partial F^{m}}{\partial_{x_{i}}}\loc{(\underline{0},t_{0})}=0$ for $1\leq i\leq m-1$.
     	%
     \end{enumerate}
    
     It follows from Lemma \ref{l: compatible redefinition} that $H'$ (or rather, its restriction to $N\times I$) is weakly $L$-compatible. The next Lemma takes care of the remaining details.
    \begin{lemma}
  		Up to replacing $H'$ by some $H''\in\mathcal{W}$ with which it agrees outside of some compact set contained in $W\times J$, we may assume that $\mathcal{F}\cap N\times\{-1,1\}=\emptyset$ and $\#(N\times\{t\}\cap\mathcal{F})\leq 1$  for any $t\in(-1,1)$. 
  	\end{lemma}
    \begin{subproof}
    	Given $(p,t)\in \mathcal{F}$, choose $m$-balls $D,D'$ in $M\times J$ with $(p,t)\in D\subseteq\mathring{D}'$ and $\{(p,t)\}=\mathcal{F}\cap D'$. By integrating a suitable vector field supported on $D'$ 
    	we can find a diffeotopy $G$ of $Id_{N\times J}$ supported on $D'$ with the property that 
    	$G_{s}$ acts as a translation 
    	$(p,t)\mapsto(p,t+s)$ on $D$. Since $G_{s}$ converges to the identity in the Whitney topology, as $s$ converges to $-1$,
    	it follows that $(j^{1}(H'\circ (G_{s},Id_{J})))^{-1}(\pl)\cap (D'\setminus\mathring{D})=\emptyset$ if $s+1$ is small enough.
    	Replacing $H'$ with $H'':=H'\circ (G_{s},Id_{J})$ preserves the properties already established for $H$, while modifying the time component of one point of $\mathcal{F}$ by a real value that can be chosen arbitrarily in some small interval around $0$. This can be also done in such a way that $H''$ belongs to any open neighbourhood $\W$ of $H$ in the Whitney topology. 
    \end{subproof}
    This concludes the proof of Lemma \ref{l: morse for isotopies}.   
   \end{proof}
   
   \newcommand{\nml}[0]{(N,M,L)} 
   We can combine Lemma \ref{l: morse for isotopies} above and Corollary \ref{c: main corollary} to obtain the following result.   
   \begin{corollary}
   	\label{c: Morse} Under the assumptions of Lemma \ref{l: morse for isotopies}, suppose that 
   	$H$ is of the form $\iota^{*}(G)$, where $\iota=H_{-1}$ is some proper embedding of $N$ in $M$
   	such that $\iota(N)\cap L=\emptyset$ and $G$ is a compactly supported diffeotopy $G$ of $Id_{M}$. Then, for any $\epsilon>0$ there is $h\in\w_{\epsilon}$ and some compactly supported diffeotopy $G'$ from $Id_{M}$ to $hg$ such that 
   	$\iota^{*}(G')$ is $L$-compatible.  
   \end{corollary}
 
   \section{More on compatible isotopies} 
  \label{s: isotopy operations}

  \begin{definition}
  	\label{d:op} Given an isotopy $H$ of $N$ in $M$, we write $H^{op}:=(H_{-t})_{t\in I}$. Given a diffeotopy $G$ of $Id_{M}$, we also write $G^{inv}:=(G_{-t}^{-1}G_{1})_{t\in I}$. 
  \end{definition}
  
  \begin{observation}
  	\label{o: niceness reversal}
  	If $H$ is an isotopy of $N$ in $M$ of the form $(G_{t}\circ H_{-1})_{t\in I}$, then $H^{op}$ is an
  	isotopy of the form $(G^{inv}_{t}\circ H_{1})_{t\in I}$. If $(N,M,L)$ is a nice $m$-triple and $H$ an $L$-compatible isotopy of 
  	$N$ in $M$, then the isotopy $H^{op}$, is itself $L$-compatible, and every tangency point of dimension $d$ of $H$ corresponds to a tangency point of dimension $m-d-1$ of $H^{op}$.
  \end{observation}
   
  \begin{definition}
  	Let $(N,M,L)$ be a nice $m$-triple and $H$ an $L$-compatible isotopy of $N$ in $M$. We say that $H$ is \emph{non-destructive} if all the tangency points of $M$ have dimension at most $m-2$. We say that $H$ is \emph{purely destructive} if all the tangency points of $M$ have dimension $m-1$. 
  \end{definition}

      \begin{figure}[t]
   	\includegraphics[width=0.45\textwidth]{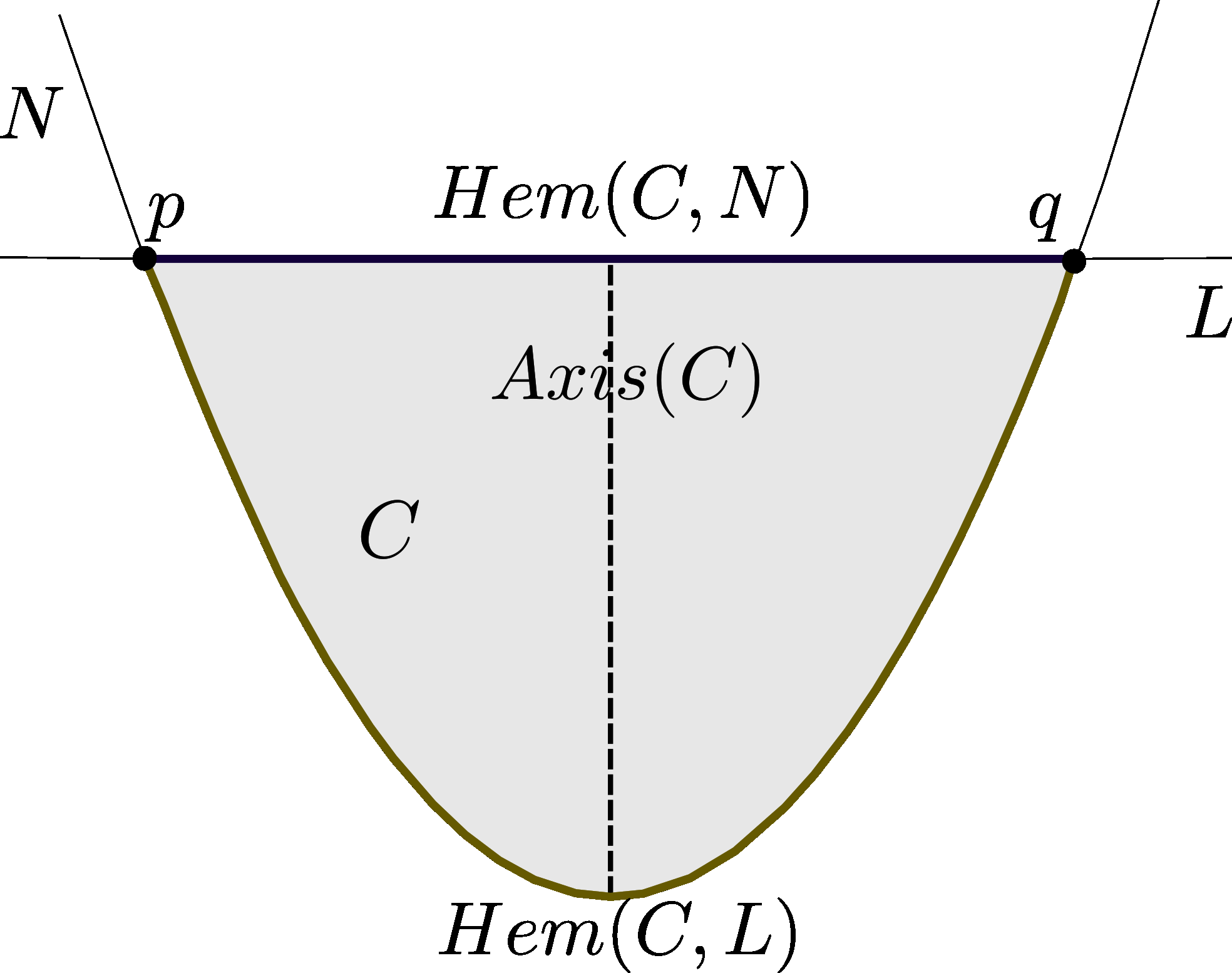}
   	\caption{	\label{fig:core} A $1$-bridge $C$, with the axis represented by a dashed line. In this case the equator consists only of two points $p,q$. The label $Hem(C,Z)$ indicates the $Z$-hemisphere of $C$. }
   \end{figure} 
   \begin{definition}
   	\label{d: core}Let $L,N\subseteq M$ be codimension $1$ submanifolds in general position. 
   	By a $d$-bridge $C$ between $L$ and $N$ we mean a set of the form 
   	$$
   	\iota(\{(\underline{x},y)\in B^{d}\times(-1,1)\,|\,\norm{\underline{x}}-\frac{1}{2}\leq y\leq 0\}),
   	$$
   	where $\iota:\overline{B}^{d}\times (-1,1)\cong C\subseteq M$ is a smooth embedding transverse to $N$ and $L$ 
   	such that 
   	$\iota^{-1}(L)=\overline{B}^{d}\times\{0\}$ and $\iota^{-1}(N)$ is the graph of the function 
   	$\underline{x}\mapsto \norm{\underline{x}}^{2}-\frac{1}{2}$. 
   	
   	In the situation above $C\cap N$ and $C\cap L$ are embedded closed $d$-balls in $N$ and $L$ respectively 
   	which we will refer to as the $N$-hemisphere and the $L$-hemisphere of $C$.
   	We refer to $\mathring{C}:=C\setminus(L\cup N)$ as the interior of $C$ and to 
   	$C\cap\iota(\underline{0}\times(-1,1))$ as the axis of $C$. The intersection of the two hemispheres is a $(d-1)$-embedded sphere in $N\cap L$, which we call the equator of $C$. The situation is depicted in Figure \ref{fig:core}. 	
   \end{definition}
   \begin{remark}
   	Note that a $d$-bridge has dimension $d+1$. A $0$-bridge is just a smooth arc intersecting each of $L$ and $N$ at a different endpoint. 
   \end{remark}

  \begin{lemma}
  	 \label{l: non-destructive isotopies} Let $(L,M,N)$ be a nice $m$-triple, $\iota:N\to M$ an embedding and $G$ a compactly supported diffeotopy such that $\iota^{*}(G)$ is $L$-compatible, with a single tangency point $(p_{0},t_{0})$ of dimension $d$. Then for any $-1\leq t^{\star}<t_{0}$ sufficiently close to $t_{0}$ there is a bridge $D\subseteq M$ between $N_{\star}:=H_{t^{\star}}(N)$ and $L$ such that for any $\eta>0$ we have the following.
  	 \begin{enumerate}[(i)]
  	 	\item \label{tubular}A (trivial) tubular neighbourhood $(E,\xi^{\eta})$ of $\mathring{D}_{\eta}$ with $im(\xi^{\eta})\subseteq\nn_{\eta}(D)$.
  	 	\item \label{diffeotopy}A diffeotopy $\chi^{\eta}$ of $Id_{M}$ supported on 
  	 	\begin{itemize}
  	 		\item $\nn_{\eta}(D)$ in the general case,
  	 		\item $im(\xi^{\eta})\cup\nn_{\eta}(D\cap(N_{\star}\cup L))\subseteq\nn_{\eta}(D)$ if $d=0$, 
  	 	\end{itemize}
  	 	 such that $\xi^{\eta}$ and the isotopy $\hat{H}:=(\chi^{\eta}_{t}\circ H_{t^{\star}})_{t\in I}$ are compatible, in the following sense:
  	 \begin{enumerate}[(a)]
  	 	   \item \label{zero property} $H_{t}(supp(H))\subseteq im(\xi^{\eta})\cup\nn_{\eta}(D\cap L)$ for all $t\in(-1,1)$,
  	 		 \item \label{first property}$\hat{H}$ is $L$-compatible, with a single tangency point of dimension $d$, and the same is true if instead of $L$ we take any of the $\xi^{\eta}$-fiber images through a point in $\mathring{D}$,
         \item \label{second property}for any $t\in I$ and $p\in N$
               the differential $D\hat{H}_{t}\loc{p}$ is tangent to the fiber image through $\hat{H}_{t}(p)$ only if $\hat{H}_{t}(p)$ is in the axis of $D$.
      \end{enumerate} 
 	 		\item \label{connecting isotopy}Some compactly supported diffeotopy $\tilde{G}^{\eta}$ from $\chi^{\eta}_{1} G_{t^{\star}}$ to $G_{1}$ such that $\iota^{*}(\tilde{G}^{\eta})$ is $L$-clean.
    \end{enumerate}  	 
  \end{lemma}
   \begin{proof}
   	Write $H$ for $\iota^{*}(G)$ and $(p_{0},t_{0})$ for the unique tangency point of $H$.
   	Up to time reparametrization of $G$, we can find $\delta=\delta_{0}>0$ and charts $\phi:U\to B^{m-1}$ and $\psi:V\to B^{m-1}\times(-1-\delta,1+\delta)$ so that
   	$$N\times(t_{0}-\delta,t_{0}+\delta)\cap H^{-1}(H(U\times(t_{0}-\delta,t_{0}+\delta)))=U\times(t_{0}-\delta,t_{0}+\delta)$$
   	and the local expression for 
   	$H_{t}$ on $U\times(t_{0}-\delta,t_{0}+\delta)$ with respect to $\phi$ and $\psi$ is of the form \begin{equation}
   		\label{eq:def f}\tag{$\flat$}x\mapsto (\underline{x},f_{t}(\underline{x})):=(\underline{x},t-t_{0}+\sum_{i=1}^{d}x_{i}^{2}-\sum_{i=d+1}^{m-1}x_{i}^{2}).
   	\end{equation}
   	 
   	Given $\underline{x}\in B^{m-1}$, write $x^{+}:=(x_{i})_{i=1}^{d}$, $x^{-}:=(x_{i})_{i=d+1}^{m-1}$.
   	
   	\begin{observation}
	   	\label{o: disjointness} By taking $U$, $V$ and $\delta$ to be sufficiently small, we may further assume the existence of an open set $V'$ with $V\subseteq V'\subseteq M$ and $\epsilon\in (0,\frac{1-t_{0}-\delta}{2})$ such that for all $t',t''\in(t_{0}-\delta,t_{0}+\delta)$ we have: 
	   	\begin{enumerate}
	   		\item  $G_{t'}G_{t''}^{-1}(V)\subseteq V'$, 
	   		\item $G_{t'+2\epsilon}G_{t''}^{-1}(\overline{V'})\cap L=\emptyset$.
	   	\end{enumerate}
   	\end{observation}
   	
    Fix any $\delta'\in(0,\min\{\delta,\frac{1}{4}\})$ and let $t^{\star}:=t_{0}-\delta'$, $\epsilon:=\sqrt{\delta'}$ .  	
   	Let $\rho:\R\mapsto\R$ be a smooth function such that: 
   	\begin{itemize}
   		\item $\rho$ is symmetric,
   		\item $\rho$ is supported on $I$,
   		\item $\rho'(0)<0$ on $(0,1)$,
   		\item and for some $C<0$ the function $\rho$ is of the form $g(s)=1+Cs^{2}$ in some neighbourhood of $0$.  
   	\end{itemize}
    	
   	For $0<\mu\leq 1$ write $\rho_{\mu}(s):=\rho(\frac{2}{\mu} s)$. 
   	Choose also some smooth function $\sigma_{\alpha}:\R_{\geq 0}\to\R_{\geq 0}$ supported on $[0,\epsilon+\frac{\alpha}{2}]$ such that the function $\tau_{\alpha}: (-1,1)\to\R$ given by $\tau_{\alpha}(s):=s^{2}-\delta'+\sigma(s^{2})$
   	satisfies the following:
   	\begin{itemize}
   	\item $\tau_{\alpha}''(s)>0$ for all $s\in(-1,1)$,
   	\item $\tau_{\alpha}(0)=\frac{\epsilon\alpha}{2}$.
   	\end{itemize}
   	 This can be done, for instance, by postcomposing the function $s\mapsto g(s):=s^{2}-\delta'$ with a suitable reparametrization of its codomain, taking the difference with $g$, and finally using that any symmetric smooth real function can be written as $h(s^{2})$ for some smooth function $h$ (see \cite{golubitsky2012stable}, p.108).  
   	
   	Given $0<\alpha<1$, we define $\theta_{\alpha}:B^{m-1}\to\R$ by the formula 
   	$$
   	\theta_{\alpha}(\underline{x}):=\norm{\underline{x}^{+}}^{2}-\norm{\underline{x}^{-}}^{2}+\rho_{\alpha}(\norm{\underline{x}^{-}})\, \sigma_{\alpha}(\norm{\underline{x}^{+}}^{2})-\delta'. 
   	$$
   	\newcommand{\bxx}[0]{B^{d}(\epsilon+\frac{\alpha}{2})\times B^{m-d-1}(\frac{\alpha}{2})}
   	Clearly $\theta_{\alpha}$ and $x\mapsto\norm{\underline{x}^{+}}-\norm{\underline{x}^{-}}-\delta'$ agree on 
   	$B^{m-1}\setminus\bxx$.
   	The observation below follows from the properties of $\rho_{\alpha}$ and $\sigma_{\alpha}$ by a simple calculation.
   
   	\begin{observation}
   	\label{o: derivatives}The derivative $\frac{\partial\theta_{\alpha}}{\partial x_{i}}(\underline{x})$ equals $0$ if and only if $x_{i}=0$ and has sign 
   	$\epsilon_{i}\frac{x_{i}}{|x_{i}|}$ otherwise, where $\epsilon_{i}=1$ if $i\leq d$ and $-1$ if $i>d$.
   	We also have 
   	$$\frac{\partial^{2}\theta_{\alpha}}{\partial x_{i}\partial x_{j}}(\underline{0})=
   	\begin{cases*}
   	 0 & for $i\neq j$\\
   	 1 & for $i=j\leq d$\\
   	-1 & $i=j>d$
   	\end{cases*}
   	$$
   	\end{observation}
   	\newcommand{\bx}[0]{B^{m-1}\times(1-\delta,1+\delta)}
   
   	Consider the sets  
   	\begin{align*}
   		D :=& \psi^{-1}(\{(\,\underline{x},y_{m})\in\bx\,| \norm{\underline{x}}-\delta'\leq y_{m}\leq 0,\,\underline{x}^{-}=\underline{0}\}),\\
   		 E^{\alpha} :=& \psi^{-1}(\{(\,\underline{x},y_{m})\in\bx\,| \norm{\underline{x}}-\delta'\leq y_{m}\leq 0,\,\norm{\underline{x}^{-}}\leq\alpha\}).
   	\end{align*}
   	Clearly, $D$ is a $d$-bridge between $N_{\star}$ and $L$. 
   	Let
   	\begin{align*}
   	 &\tilde{\xi}^{\alpha}:D \times B^{m-d-1}\to E^{\alpha},\quad\quad \tilde{\xi}^{\alpha}(p,v):=\psi^{-1}(\psi(p)^{+},\alpha v),\\   	
   		U_{\alpha}:=\{(\,\underline{x}&,y_{m})\in B^{m-1}\times\R\,\,|\,\, \norm{\underline{x}}-\delta'-\alpha\leq y_{m}\leq \alpha,\,\norm{\underline{x}}^{+}\leq\epsilon+\alpha,\, \norm{\underline{x}^{-}}\leq\alpha\},
   	\end{align*} 
   so that $ E^{\alpha}\subseteq U_{\alpha}$. 
   
   Consider also the isotopies 
    $\tilde{\chi}^{\alpha}, \hat{\Theta}^{\alpha}$ of $B^{m-1}$ in $B^{m-1}\times(-1-\delta,1+\delta)$ given by
   		\begin{align*}
	   		\tilde{\chi}^{\alpha}(\underline{x},t):=& (\underline{x},\frac{1+t}{2}\theta_{\alpha}(\underline{x})+\frac{1-t}{2}f_{t^{\star}}(\underline{x})), \\
	   		\hat{\Theta}^{\alpha}(\underline{x},t):= &(\underline{x},\frac{1+t}{2}f_{t_{0}+\delta'}+\frac{1-t}{2}\theta_{\alpha}(\underline{x})).
	   	\end{align*}	   
	  Let also $L_{\lambda}$ be the submanifold of $B^{m-1}\times\R$ given by the equation $y_{m}=\lambda$. We note the following properties.
   	\begin{lemma}
   		 \label{l: Gamma} The following conditions hold:  		 
   		  \begin{enumerate}[(I)]
   		  	 \item \label{final stage} For all $\underline{x}\in B^{m-1}$ $$\hat{\chi}^{\alpha}_{1}(\underline{x},f_{t^{\star}}(\underline{x}))=(\underline{x},\theta_{\alpha}(\underline{x}))=\hat{\Theta}^{\alpha}_{-1}(\underline{x}).$$
   		  	 \item \label{cuatro}For $\underline{x}\in B^{m-1}\setminus B^{d}(\epsilon+\alpha)\times B^{m-d-1}(\alpha)$ one has 
            $$\hat{\Theta}^{\alpha}(\underline{x},t):=(\underline{x},f_{t^{\star}+(t+1)\delta'}(\underline{x}))$$ 
            \item \label{constant}The isotopy $\tilde{\chi}^{\alpha}_{t}$ is constant outside of $B^{d}(\epsilon+\frac{\alpha}{2})\times B^{m-d-1}(\frac{\alpha}{2})$.
          	\item \label{tangency condition}For  $\lambda\in(-\delta',\frac{\epsilon\alpha}{2})$ the isotopy $\hat{\chi}^{\alpha}$ is $L_{\lambda}$-compatible, with a unique tangency point of dimension $d$ of the form $(\underline{0},t)$.
           \item \label{tres} The isotopy $\hat{\Theta}^{\alpha}$ is $L_{0}$-clean. 
   		  	\item \label{covering}There is a diffeotopy $\hat{\chi}^{\alpha}$ of  $Id_{B^{m-1}\times(-1-\delta,1+\delta)}$ covering $\tilde{\chi}^{\alpha}$ and supported on $U_{\alpha}$. 
   		  	\end{enumerate}    
   	\end{lemma}

   	\begin{subproof} 
   		Items (\ref{final stage})-(\ref{constant}) are clear from the definitions of $\hat{\Theta}^{\alpha}$, $\hat{\chi}^{\alpha}$ and the properties of $\theta_{\alpha}$. 
   		By Observation \ref{o: derivatives}, for any $1\leq i\leq m-1$ and $t\in(t-\delta,t+\delta)$ the partial derivatives
   		$\frac{\partial f_{t}}{\partial x_{i}}$ and $\frac{\partial \theta_{\alpha}}{\partial x_{i}}$ have the same set of zeroes and 
   		the same sign outside of it. Similarly, for $1\leq i, j\leq d-1$ we have that $\frac{\partial^{2}f_{t}}{\partial x_{i}\partial x_{j}}(\underline{0})$ is zero, positive or negative precisely when 
   		$\frac{\partial^{2}\theta_{\alpha}}{\partial x_{i}\partial x_{j}}(\underline{0})$ is. 
   		Clearly, the same applies to any two convex combinations of $f_{t}$ and $\theta_{\alpha}$. In particular, $\underline{0}$ is the only critical point of any such combination. 
   		It follows that $\tilde{\chi}^{\alpha}$ is $L_{\lambda}$-compatible, with a unique tangency point of dimension $d$ for every $\lambda\in(-\delta',\frac{\epsilon\alpha}{2})$, by Lemma \ref{l: compatible redefinition}. That $\hat{\Theta}^{\alpha}$ is $L_{0}$-clean follows from the same argument as above (in this case the critical value remains positive at all points in time). This settles (\ref{tangency condition}) and (\ref{tres}). Lemma \ref{l: isotopy extension}, together with (\ref{constant}) implies that $\tilde{\chi}^{\alpha}$ is covered a diffeotopy $\hat{\chi}^{\alpha}$ of $Id_{B^{m-1}\times(-1-\delta,1+\delta)}$ supported in $U_{\alpha}$, so (\ref{covering}) holds as well.
   	\end{subproof}

   	Clearly, for $\alpha$ small enough we have $U_{\alpha}\subseteq\nn_{\eta}(D)$ and    
     \begin{align*}
   	 \psi^{-1}(B^{d}(\epsilon+\frac{\alpha}{2})\times B^{m-d-1}(\alpha)\times[0,\alpha])\subseteq\nn_{\eta}(D\cap L).
     \end{align*}
    If we then let $\xi^{\eta}:=\tilde{\xi}^{\alpha}$ and take $\chi^{\eta}$ to be the extension by the identity of the pull-back of $\hat{\chi}^{\alpha}$ by $\psi$, then it follows from the lemma above that 
    $\xi^{\eta}$ and $\chi^{\eta}$ satisfy properties (\ref{tubular}) and (\ref{diffeotopy}) of the statement of \ref{l: non-destructive isotopies}.
   	
    The existence of $\tilde{G}^{\eta}$ satisfying (\ref{connecting isotopy}) is somewhat more delicate. Write $\chi$ for $\chi^{\eta}$ and $\hat{\Theta}$ for $\hat{\Theta}^{\alpha}$. If we define $\Theta: N\times I\to M$ by
   	$$
   	 \Theta(p,s) := 
   	 \begin{cases*}
   	   \psi^{-1}(\hat{\Theta}(\phi(p)),s) & for $p\in U,\, s\in I$ \\
   	   H(p,t^{\star}+(s+1)\delta')  & for $p\notin U,\,s\in I$,   	    
   	 \end{cases*}
   	$$
   	then it follows from (\ref{eq:def f}) and from (\ref{cuatro}) in Lemma \ref{l: Gamma} that $\Gamma$ is a compactly supported isotopy of $N$ in $M$, which is $L$-clean, by (\ref{tres}).
   	
   	 Consider the isotopy $\Xi$ of $N$ in $M$ given by
     \begin{equation}
     	\label{eq:xi}\tag{$\flat\flat$} \Xi_{s}:=G_{t^{\star}}G_{t^{\star}+(1+s)\delta'}^{-1}\circ\Theta_{s}.
     \end{equation}   

      For all $s\in I$ and $p\in N\setminus U$ the definition of $\Theta$ implies:
       $$
         \Xi(p,s):=G_{t^{\star}}G_{t^{\star}+(s+1)\delta'}^{-1} G_{t^{\star}+(s+1)\delta'}\circ\iota(p)=G_{t^{\star}}\circ\iota(p).
       $$
       Since $\Theta_{s}(U)\subseteq V$ for all $s\in I$, Observation \ref{o: disjointness} implies that $\Xi_{s}(U)\subseteq V'$. By Lemma \ref{l: isotopy extension} there is a diffeotopy $\Omega$ of $Id_{M}$ supported on 
       $V'$ with the property that $\Xi_{s}=\Omega_{s}\circ\Xi_{-1}$ for all $s\in I$.
       Combined with (\ref{eq:xi}) and the equality
	     \begin{align*}
	      	  \label{ximinus}\Xi_{-1}&=G_{t^{\star}}G_{t^{\star}}^{-1}\circ \Theta_{-1}=\Theta_{-1}=\chi_{1} G_{t^{\star}}\circ\iota,
	      \end{align*}
      this implies that $\Theta=\iota^{*}(\Upsilon^{1})$, where  
      $$     
      \Upsilon^{1}:=(G_{t^{\star}+(s+1)\delta'}G_{t^{\star}}^{-1}\Omega_{s}\chi_{1}G_{t^{\star}})_{s\in I}.
      $$
%
      Consider, furthermore, the following diffeotopies:  
      \begin{align*}
      	\Upsilon^{2}:&= (G_{t_{\star}+2\delta'+\epsilon(s+1)}G_{t^{\star}}^{-1}\Omega_{1}\chi_{1}G_{t^{\star}})_{s\in I},\\  	
      	\Upsilon^{3}:&= (G_{t_{\star}+2(\delta'+\epsilon)}G_{t^{\star}}^{-1}\Omega_{-s}\chi_{-s}G_{t^{\star}})_{s\in I}.
      \end{align*}
     Notice that  
     \begin{align*}
      \Upsilon^{1} _{-1}=\chi_{1}G_{t^{\star}},\quad\quad \Upsilon^{1}_{1}=G_{t_{\star}+2\delta'}G_{t^{\star}}^{-1}\Omega_{1}\chi_{1}G_{t^{\star}}=\Upsilon^{2}_{-1},
      \\\Upsilon^{2} _{1}=G_{t_{\star}+2(\delta'+\epsilon)}G_{t^{\star}}^{-1}\Omega_{1}\chi_{1}G_{t^{\star}}=\Upsilon^{3}_{-1},\quad\quad\Upsilon^{3}_{1}:=G_{t_{\star}+2(\delta'+\epsilon)},
     \end{align*}
     so that $$\tilde{G}^{\eta}:=\Upsilon^{1}*\Upsilon^{2}*\Upsilon^{3}*G_{\restriction M\times(t_{\star}+2(\delta'+\epsilon),1]}$$
      is a diffeotopy from 
     $\chi_{1}G_{t^{\star}}$ to $G_{1}$. 
     
     Recall that $\iota^{*}(\Upsilon^{1})=\Theta$ is $L$-clean, by construction.
     Likewise, the isotopy
       \begin{align*}
       	\iota^{*}(\Upsilon^{2})=&(G_{t_{\star}+2\delta'+\epsilon(s+1)}G_{t_{\star}+2\delta'}^{-1}\Theta_{1})_{s\in I}=\\
       	&(G_{t_{\star}+2\delta'+\epsilon(s+1)}G_{t_{\star}+2\delta'}^{-1}H_{t_{\star}+2\delta'})_{s\in I}=(H_{t_{\star}+2\delta'+\epsilon(s+1)})_{s\in I},
       \end{align*}
     is $L$-clean. Finally, for all $s\in I$ the diffeomorphism $\Omega_{s}\chi_{s}$ is supported on $V'$, from which it follows that
     $\Upsilon^{3}$ is constant on some neighbourhood of $(\Upsilon_{-1}^{3})^{-1}(L)$. 
     Since $\chi_{1}G_{t^{\star}}\circ\iota\transv L$, it follows from Observation \ref{o: compatibility and strictness} that the isotopy $\iota^{*}(\Upsilon^{3})$ is also $L$-clean. 
     So $\iota^{*}(\tilde{G}^{\eta})$ is also $L$-clean, as needed.
   \end{proof}
   
     We record the following simple algebraic consequence of Lemma \ref{l: non-destructive isotopies} for later use.
  \begin{corollary}
  	 \label{c: non-destructive isotopies} Let $(N,M,L)$ be a nice $m$-triple, $G$ a compactly supported diffeotopy from $Id_{M}$ to some element $g\in \D_{c0}(M)$, and $\iota:N\to M$ a proper embedding such that
  	  $\iota^{*}(G)$ is an $L$-compatible non-destructive isotopy of $N$ in $M$ with $r$ tangency points. Let also $\mathscr{X}\subseteq\D_{c0}(M)$ be a set such that for every
  	 embedded $k$ ball $D$ in $M$, $k\leq m-1$ we have $\cmp{\eta}{D}\subseteq\mathscr{X}$ for some $\eta>0$.
  	 Then there are  
  	 $h_{0},\dots h_{r-1}\in\mathscr{X}$ as well as elements $g_{L}\in \D_{c0}^{\{L\}}(M)$, $g_{N}\in \D_{c0}^{\{\iota(N)\}}(M)$ such that $g=g_{L}h_{r-1}\cdots h_{0}g_{N}$.        
  \end{corollary}
  \begin{proof}
	  We prove the statement by induction on $r$, where $r=0$ is just Lemma \ref{l: stability application}. Let $t_{0},\dots t_{r-1}$ be the tangency times of $\iota^{*}(G)$. For any $u\in\D(M)$ it is clear that the set $\mathscr{X}^{u^{-1}}$ also satisfies the property of the statement.
	  As a result, by applying Lemma \ref{l: non-destructive isotopies} to the restriction of $G$ to a suitable final subinterval of $I$, we obtain $t^{\star}\in(t_{r-2},t_{r-1}]$ such that for any given $u\in D(M)$ we can modify $G$ on $[t^{\star},1]$ so that for some $t'\in(t^{\star},1)$ we have:
	  \begin{itemize}
	  	\item \label{condi}$G_{1}=\tilde{g}hG_{t^{\star}}$, where $h\in\mathscr{X}^{u^{-1}}$ ($\chi_{1}$ in Lemma \ref{l: non-destructive isotopies}),
	  	\item \label{condii}$\tilde{g}$ is of the form $\Theta_{1}$, for some compactly supported diffeotopy of $Id_{M}$ such that
	  	$\Theta_{s}hG_{t^{\star}}\circ\iota\transv L$ for every $s\in I$.
	  \end{itemize}   
    It follows from the second condition above that one can write
    $\tilde{g}=g'_{L}g'_{N}$, where, letting $N':=hG_{t^{\star}}\circ\iota(N)$, we have
    $g'_{L}\in\D_{c0}^{\{L\}}(M)$ and $g'_{N}\in\D_{c0}^{\{N'\}}(M)$. 
	 
    The induction hypothesis allows us to write $G_{t^{\star}}=g''_{L}h_{r-2}\cdots h_{0}g''_{N}$, where 
	  $$g''_{L}\in\D_{c0}^{\{L\}}(M), \quad\quad g''_{N}\in\D_{c0}^{\{\iota(N)\}}(M),\quad\quad h_{i}\in\mathscr{X},\,\,0\leq i\leq r-2.$$
	  We choose a decomposition $g=\tilde{g}hG_{t^{\star}}=g'_{L}g'_{N}hG_{t^{\star}}$, as described in the first paragraph for the choice $u:=g''_{L}$. Then
	    \begin{align*}
	    	g=g'_{L}hG_{t^{\star}}(g''_{N})^{hG_{t^{\star}}}=&g'_{L}hg''_{L}h_{r-2}\cdots h_{0}g''_{N}(g'_{N})^{hG_{t^{\star}}}=\\
	    	&g'_{L}g''_{L}h^{g''_{L}}h_{r-2}\cdots h_{0}g''_{N}(g'_{N})^{hG_{t^{\star}}}.
	    \end{align*}
	  Since $h_{r-1}:=h^{g''_{L}}\in\mathscr{X}$ and $(g'_{N})^{hG_{t^{\star}}}\in\D_{c0}^{\{\iota(N)\}}(M)$, the result follows.  	  
  \end{proof}
   
    The following is a crucial observation.  
   \begin{observation}
   	\label{o: conjugation}Let $(N,M,L)$ be a nice $m$-triple and $\iota: N\to M$ a proper embedding. Suppose that we are given a compactly supported diffeotopy $G$ of $Id_{M}$ such that the isotopy $\iota ^{*}(H)$ of $N$ in $M$ is $L$-compatible and has type $\sigma$. Let also $G'$ be a compactly supported diffeotopy $Id_{M}$ preserving $L$ and $\iota(N)$ setwise at all times, and write $g:=G_{1}'$. Then the isotopy $H' :=(gG_{t}g^{-1}\circ\iota)_{t\in I}$ of $N$ in $M$ is also $L$-compatible and has the same type as $H$. Moreover, the diffeotopy $G'':=(G'_{-t}G_{1}(G'_{-t})^{-1})_{t\in I}$ from $gG_{t}g^{-1}$ to $G_{1}$ is such that $\iota^{*}(G'')$ is $L$-clean. 
   \end{observation}
   \begin{proof}
   	 Notice that $g^{-1}\circ\iota=\iota\circ\hat{g}^{-1}$ for some $\hat{g}\in\D_{c0}(N)$. 
   	 If $(p,t)\in N\times I$, clearly $gG_{t}g^{-1}\circ\iota(p)=gG_{t}\circ\iota\circ \hat{g}^{-1}(p)$ is in $L$ if and only if $G_{t}\circ\iota(p)\in L=g^{-1}(L)$. For any $(p,t)\in N\times (-1,1)$ for which the latter holds we also have $im(D(G_{t}\circ\iota)\loc{p})=T_{H_{t}(p)}L$ if and only if $im(DH'_{t}\loc{\hat{g}^{-1}(p)})=T_{H'_{t}(\hat{g}^{-1}(p))}L$ and precomposing the charts $\phi$ and $\psi$ in Definition \ref{d: compatible isotopy} with $\hat{g}$ and $g^{-1}$ respectively yields the required local expression. The final assertion regarding $\iota(G'')$ follows from a similar argument.
   \end{proof}

   \begin{lemma}
   	\label{l: commutation move} Let $(N,M,L)$ be a nice $m$-triple, $\iota$ a proper embedding of 
   	$N$ in $M$ and $G$ a compactly supported diffeotopy of $Id_{M}$ such that
   	$H:=\iota^{*}(G)$ is $L$-compatible with type $\sigma$. Assume that $\sigma$ contains a subword $\sigma_{0}$ of one of the following two forms:
   	\begin{enumerate}
   		[(a)]
   		\item \label{big step} $(d,d')$ where $d\geq d'+2$,
   		\item \label{small step}$(m-1,\dots m-1,m-2)$.
   	\end{enumerate}   	
     Then there is a compactly supported diffeotopy $G'$ of $Id_{M}$ such that
   	\begin{enumerate}[(i)]
   		\item $G_{1}=G'_{1}$,
   		\item and the type of $\iota^{*}(G)$ is the result of replacing in $\sigma$ the subword $\sigma_{0}$ by:
       \begin{itemize}
       	\item $(d',d)$ in case (\ref{big step}),
       	\item some word of the form $(d'_{1},\dots d'_{q},m-2,m-1,\dots m-1)$, where $d'_{l}\in\{0,m-2\}$ for all $1\leq l\leq q$
       	in case (\ref{small step}) (see Figure \ref{angelus} below).
       \end{itemize}   		
   	\end{enumerate}  
   \end{lemma}
   \begin{proof}
   	By using Fact \ref{f: concatenation} and Observation \ref{o: compatibility and strictness}, one may assume $\sigma=\sigma_{0}$. Write $\sigma:=(d_{-r},\dots d_{0},d_{1})$ and denote by $(t_{-r},\dots t_{0},t_{1})$ the corresponding increasing sequence of tangency times.	Lemma \ref{l: non-destructive isotopies} implies the existence of some $t^{\star}_{1}\in (t_{0},t_{1})$ and some  $d_{1}$-bridge $D_{1}$ between $H_{t^{\star}_{1}}$ and $L$ such that for all $\eta>0$ we may modify the restriction of $G$ to $N\times(t^{\star}_{1},1)$, without changing the type of $H$, in such a way that, keeping the notation $t_{1}$ for the new tangency time, for some $s_{1}>t_{1}>t^{\star}_{1}$, the restriction of $H$ to $G\times[t^{\star}_{1},s_{1}]$ is of the form 
   	$(\chi^{1}_{t}\circ G_{t^{\star}_{1}})_{t\in I}$ for some diffeotopy $(\chi^{1}_{t})_{t\in I}$ of $Id_{M}$.
   	
   	    Observation \ref{o: niceness reversal} allows us to apply Lemma \ref{l: non-destructive isotopies} to $H^{op}$ and $G^{inv}$, which yields $t^{\star}_{0}\in(t_{0}.t^{\star}_{1})$ and some $(m-d_{0}-1)$-bridge $D_{0}$ in $M$ between $H_{t^{\star}_{0}}(N)$ and $L$, so that for all $\eta>0$ one can modify the restriction of $G$ to $N\times(-1,t^{\star}_{0})$, without altering the type of $\iota^{*}(G_{M\times\restriction[-1,t^{\star}_{0}]})$, so that, as a result, for some $s_{0}<t_{0}$ the restriction of $G$ to $N\times[s_{0},t^{\star}_{0}]$ is a time reparametrization of 
   	$(\chi^{0}_{-t}G_{t^{\star}_{0}})_{t\in I}$
   	for some diffeotopy $(\chi^{0}_{t})_{t\in I}$ of $Id_{M}$ supported on $\nn_{\eta}(D_{0})$. 
  
   	By Lemma \ref{l: stability application}, the restriction of $G$ to $N\times[t^{\star}_{0},t^{\star}_{1}]$ can be assumed to be of the form $(G_{t-t^{\star}_{0}-1}^{L}G_{t^{\star}_{0}}G^{N}_{t-t^{\star}_{0}-1})_{t\in I}$ for a compactly supported diffeotopy $G^{L}$ of $Id_{M}$ preserving $L$ and a compactly supported diffeotopy $G^{N}$ of $Id_{M}$ preserving $\iota(N)$.\footnote{Here a direct application of \ref{l: stability application} really shows that for $t\in[t^{\star}_{0},t^{\star}_{1}]$ we have $G_{t}G_{t_{0}}^{-1}=G^{L}_{t-t^{\star}_{0}}K^{N}_{t-t^{\star}_{0}}$ where the diffeotopy $K^{L}$ preserves $L$ and the diffeotopy $K^{N}$ preserves $G_{t^{\star}_{0}}(\iota(N))$.}  	
   	Let $g^{L}:=G^{L}_{t^{\star}_{1}-t^{\star}_{0}-1}$ and $g^{N}:=G^{N}_{t^{\star}_{1}-t^{\star}_{0}-1}$. We can replace the 
   	restriction of $G$ to $N\times[t^{\star}_{0},s_{1}]$ with the concatenation of the following diffeotopies:
   	
   	\begin{itemize}
   	   \item 	$((g^{L})^{-1}\chi^{1}_{t}G_{t^{\star}_{1}})_{t\in I}=((g^{L})^{-1}\chi^{1}_{t}g^{L}G_{t^{\star}_{0}}g^{N})_{t\in I}$,
   	which, by the same argument as in the proof of Observation \ref{o: conjugation}, pulls back by $\iota$ to an isotopy of $N$ in $M$ which is $L$-compatible, and has the same type as $(\chi^{1}_{t}\circ H_{t^{\star}_{1}})_{t\in I}$,
   	  \item $((G^{L}_{-t})^{-1}\chi^{1}_{1}G_{t^{\star}_{1}})_{t\in I}$, which 
   	  clearly pulls back by $\iota$ to an $L$-clean isotopy of $N$ in $M$.
   	\end{itemize} 
      Recall that here we are using Fact \ref{f: concatenation} and Observation \ref{o: compatibility and strictness} and we will continue to do so tacitly.
   	  Since we can choose $\chi^{1}$ so that $((g^{L})^{-1}\chi^{1}_{t}g^{L})_{t\in I}$ is supported on any given neighbourhood of the bridge $(g^{L})^{-1}(D_{1})$, we might as well assume $t^{\star}_{0}=t^{\star}_{1}=:t^{\star}$. 
   	  Write $g_{\star}:=G_{t^{\star}}$, $f_{\star}:=H_{t^{\star}}$, $N_{\star}:=f_{t^{\star}}$. 
   	\begin{lemma}
   		\label{l:disjoint hemispheres} Under the general hypotheses of \ref{l: commutation move}, we may assume that for $Z,W\in\{L,N_{\star}\}$ the $Z$-hemisphere of $D_{0}$ and the $W$-hemisphere of $D_{1}$ are disjoint. 
   	\end{lemma}
   	\begin{subproof}
    	Observation \ref{o: conjugation} allows us to replace $D_{0}$ (or, equivalently, $D_{1}$) with its image by any $h\in\D_{c0}^{\{L,N_{\star}\}}(M)$ (see \ref{d: manifoldstabilizer}), something we will use repeatedly below. Let $C_{i}$ be the equator of $D_{i}$ (recall Definition \ref{d: core}). Notice that $C_{0},C_{1}$ are embedded spheres of dimension 
   	$m-d_{0}-2$ and $d_{1}-1$, respectively, in the $(m-2)$-dimensional manifold $P:=L\cap N_{\star}$. Since $m-d_{0}+d_{1}-3<m-3$, it follows from Lemma \ref{l: wlog transverse} that there exists $h\in\D_{c0}(P)$ such that $h(C_{0})\cap C_{1}=\emptyset$. 
   	By the last assertion of Lemma \ref{l: isotopy extension}, the map $h$ extends to an element of $\D_{c0}^{\{N_{\star},L\}}(M)$. This allows us to assume $C_{0}\cap C_{1}=\emptyset$.  
   	
   	For $Z\in\{L,N_{\star}\}$ let now $E^{Z}_{i}$ be the $Z$-hemisphere of $E_{i}$.
   	The sum of the dimensions of $E^{Z}_{0}$ and $E^{Z}_{1}$ is $(m-d_{0}-1)+d_{1}<m-1$.
   	Since the boundaries of $C_{0}$ and $C_{1}$ are disjoint, by using the same argument as before, we can find $h_{Z}\in\D_{c0}^{P}(Z)$ such that 
   	$h_{Z}(E^{Z}_{0})\cap E^{Z}_{1}=\emptyset$ for $Z\in\{L,N_{\star}\}$, which then extends to some element $g_{Z}\in\D_{c0}^{\{Z\},Z'}(M)$, where $\{Z,Z'\}=\{N_{\star},L\}$. By replacing $D_{0}$ with $g_{L}g_{N_{\star}}(D_{0})$, we can assume that
   	$E^{Z}_{0}\cap E^{Z}_{1}=\emptyset$ for $Z\in\{L,N_{\star}\}$, concluding the proof of \ref{l:disjoint hemispheres}.
   	\end{subproof}
   	
   	\subsection*{Case (\ref{big step})}  
   	We work under the conclusions of Lemma \ref{l:disjoint hemispheres}. In this case the sum of the dimensions of $D_{0}$ and $D_{1}$ is less than $m$. Therefore, by Lemma \ref{l: wlog transverse}, it is possible to find 
    $h\in\D_{c0}^{L\cup N_{\star}}(M)$ such that $h(D_{1})\cap D_{0}=\emptyset$. By using Observation \ref{o: conjugation}, we may thus assume $D_{0}\cap D_{1}=\emptyset$. Since we can take $\chi^{l}$ to be supported in an arbitrary neighbourhood of $D_{l}$ for $l=0,1$ (the choice of $\chi^{0}$ and $\chi^{1}$ are entirely independent), we can assume
    $supp(\chi^{1})\cap supp(\chi^{0})=\emptyset$. It now suffices to take as $G'$ the concatenation of 
    $$G_{\restriction M\times[-1,s_{0}]},\quad 
    (\chi^{1}_{t}\chi^{0}_{1}G_{t^{\star}})_{t\in I}=(\chi^{0}_{1}\chi^{1}_{t}G_{t^{\star}})_{t\in I},\quad (\chi^{0}_{-t}\chi^{1}_{1}G_{t^{\star}})_{t\in I},\quad G_{\restriction M\times[s_{1},1]}.$$
        
    \subsection*{Case (\ref{small step})}  Here the sum of the dimensions of $D_{0}$ and $D_{1}$ is exactly $m$, so after applying Lemma \ref{l: wlog transverse} and Observation \ref{o: conjugation}, we may only assume that $D_{0}$ and $D_{1}$ intersect transversely in finitely many points in $\mathring{D_{0}}\cap\mathring{D_{1}}$. 
    This turns out to be good enough, since we are able to move $D_{0}$ out of the way first, by applying a sequence of moves of dimension $0$ and $m-2$ only, then perform the deformation along $D_{1}$ and finally apply a twisted copy of the deformation along the image of $D_{0}$, which still results in a sequence of moves of dimension $m-1$ (see Figure \ref{angelus}).  
    As a matter of fact, we need to take care of all the $L$-tangencies of dimension $(m-1)$ in $\sigma$, with their associated arcs ($0$-bridges) simultaneously. This requires some care. We spell out the details below. 
      	
   Fix some chart 
   	$$\phi=\pmb{y}:U\cong B^{m-2}(3)\times(-3,3)^{2}\subseteq\R^{m}$$
   	 such that
   	\begin{enumerate}[(I)]
   		\item \label{dul}$\phi(U\cap N_{\star})=\phi(U)\cap\{(x_{1},\dots x_{m})\,|\,x_{m-1}=\frac{x_{1}^{2}+\dots x_{m-2}^{2}}{2}-2\}$,
   		\item \label{sed}$\phi(L\cap U)=\phi(U)\cap\{y_{m-1}=0\}$,
   		\item \label{ned}$ \phi(D_{1})=\{(x_{1},\dots x_{m})\,|\,\frac{x_{1}^{2}+\dots x_{m-2}^{2}}{2}-2\leq x_{m-1}\leq 0,\,x_{m}=0\}$. 
   	\end{enumerate}
   	
    We think of the $m$-th factor as the ``thickness" of the box, the direction along which some tubular sections of the image of $N$ will cross it, and of its $(m-1)$-th factor as its ``height", the direction along which these tubular sections will be deformed later on.
    Consider also the sets 
   	\begin{align*}
   		 T':=&B^{m-2}(1)\times(-1,1)\times [-2,2],\quad\quad T:= \phi^{-1}(T'),\\
   		Q' :=&B^{m-2}(1)\times(-1,0)\times [-2,2],\quad\quad Q:=\phi^{-1}(Q').
   	\end{align*}
   	Choose some curve
   	$$
	   	\gamma:(-3,3)\to Q',\quad\quad\gamma(s):=(\underline{0},\theta(s)),
   	$$
   	 where $\underline{0}\in\R^{m-2}$ in the expression above and $\theta=(\theta_{1},\theta_{2}):[-3,3]\to  (-1,0)\times [-2,2]$ is a regular curve 
   	 satisfying the following conditions:
   	\begin{itemize}
   		\item $\theta(-s)=(\theta_{1}(s),-\theta_{2}(s))$ for all $s$, 
   		\item $\theta_{2}$ is non-decreasing, 
   		\item $\theta_{2}(3\nu)=2\nu$ for $\nu\in\{-1,1\}$, 
   		\item $\theta_{2}$ is constant (i.e., $\theta$ is vertical) on $[1,2]$ (hence, also on $[-2,-1]$),  
   		\item $\theta_{1}$ is constant (i.e., $\theta$ is horizontal) in some neighbourhood of $\{-3,3\}$,  
   		\item $(\theta_{1})'(s)\leq 0$ for $s\in [0,3)$, 
   		\item $(\theta_{1})''(s)<0$ for $s\in (-1,1)$, 
   		\item $\theta_{1}(2)\geq\frac{1}{2}$.
   	\end{itemize}
    Write $\alpha:=\theta_{1}(1)$ and $\beta:=\theta_{1}(2)$, so that  $\beta<\alpha$. 
    
   	Given isometries $\mu_{s}$ between $\R^{m-1}$ and the space of vectors orthogonal to $\gamma'(s)$ in $\R^{m}$ for $s\in[-3,3]$, we have that for $\epsilon>0$ small enough the map 
   	$$\xi:V:=[-3,3] \times \overline{B}^{m-1}\to Q',\quad \xi(s,v):=\gamma(s)+\epsilon\mu_{s}(v),$$ 
   	is a smooth embedding.
   	If we denote by $\pi_{l}$ the projection on the $l$-th coordinate, we may also assume that and $v\in \overline{B}^{m-1}$ we have:
   	\begin{itemize}
   		\item  $\pi_{m-1}(\xi(\{s\}\times \overline{B}^{m-1})\subseteq [\alpha,0)$ for $s\in[-1,1]$, 
   		\item  $\pi_{m-1}(\xi(\{s\}\times \overline{B}^{m-1})\subseteq (-1,\beta)$ for $s\notin (-2,2)$.
   	\end{itemize}
         
    Given a closed interval $J\subseteq\R$, we say that an embedding $\kappa:J \times \overline{B}^{m-1}\to M$ is $\xi$-restricted if $im(\kappa)\subseteq im(\phi^{-1}\circ\xi)$ and there is some diffeomorphism
    $\rho:J\cong[-3,3]$, $\nu\in(0,1)$ and $u\in B^{m-1}$ such that 
    $\phi(\kappa((s,v))=\xi(\rho(s),\nu v+u)$ for all $s\in J$ and $v\in\overline{B}^{m-1}$.  
    We say that $\kappa$ is adapted to $\xi$ if there are finitely many disjoint closed intervals $\{J_{l}\}_{l=1}^{n}$ such that $\kappa^{-1}(Q)=\bigcup_{l=1}^{n}J_{l}\times\overline{B}^{m-1}$ and 
    $\kappa_{\restriction J_{l}\times\overline{B}^{m-1}}$ is $\xi$-restricted for all $1\leq l\leq n$. 
    By a $\xi$-spherical shell we mean a set of the form $\phi(\kappa(J\times S^{m-1}))\subseteq Q'$, where $\kappa$ is $\xi$-restricted. 
    
    The following claim is left to the reader. The last assertion can be proven by exhibiting subspaces of the right dimension on which the Hessian is positive and negative definite, respectively.
    \begin{claim}
    	\label{c: big hessian claim} For $\epsilon>0$ small enough the map $\xi$ above is such that for any $\xi$-spherical shell $\Sigma$ in $Q'$ there are exactly two values $\lambda_{0}<\lambda_{1}$ in $(\beta,0)$ for which the plane $L_{\lambda}:=\{y_{m}=\lambda\}$ is tangent to $\Sigma$, and $L_{\lambda_{i}}$ is tangent to $\Sigma$ at a single point for $i=0,1$. Locally around $p_{i}$ the surface $\Sigma$ is the graph of a function $f$ with a unique critical point, corresponding to $p_{i}$, and a non-singular Hessian such that the number positive entries in its diagonal form is $0$ for $i=0$ and $m-2$ for $i=1$.    
    \end{claim}
    
    Fix $\xi$ as given by the claim.

   \begin{figure}
   	\captionsetup{width=\linewidth}
    \centering
    \begin{subfigure}[t]{0.4\textwidth}
        \centering
        \includegraphics[width=\linewidth]{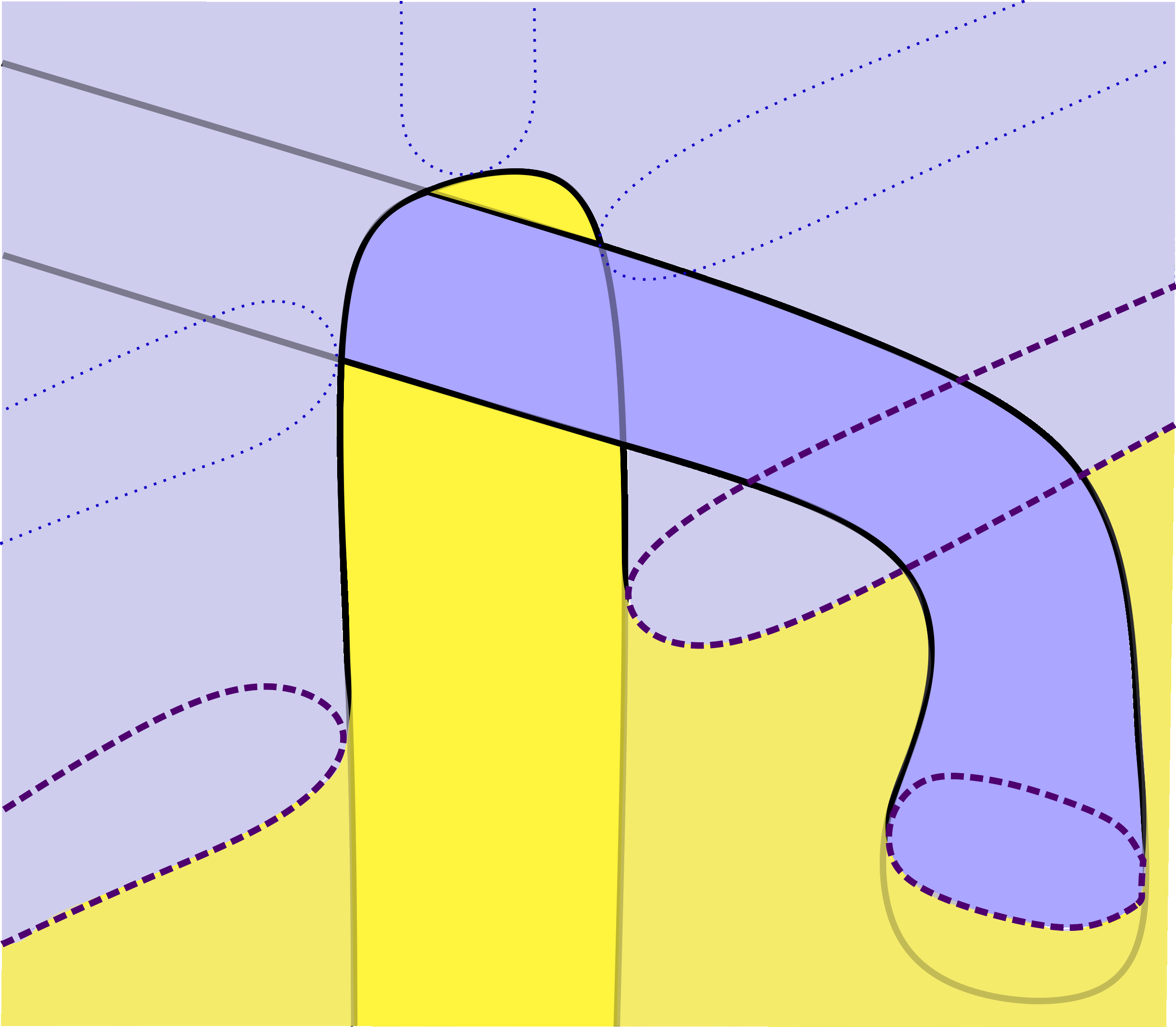} 
        \caption{Initial configuration.} \label{fig:a1}
    \end{subfigure}
    \hfill
    \begin{subfigure}[t]{0.4\textwidth}
        \centering
        \includegraphics[width=\linewidth]{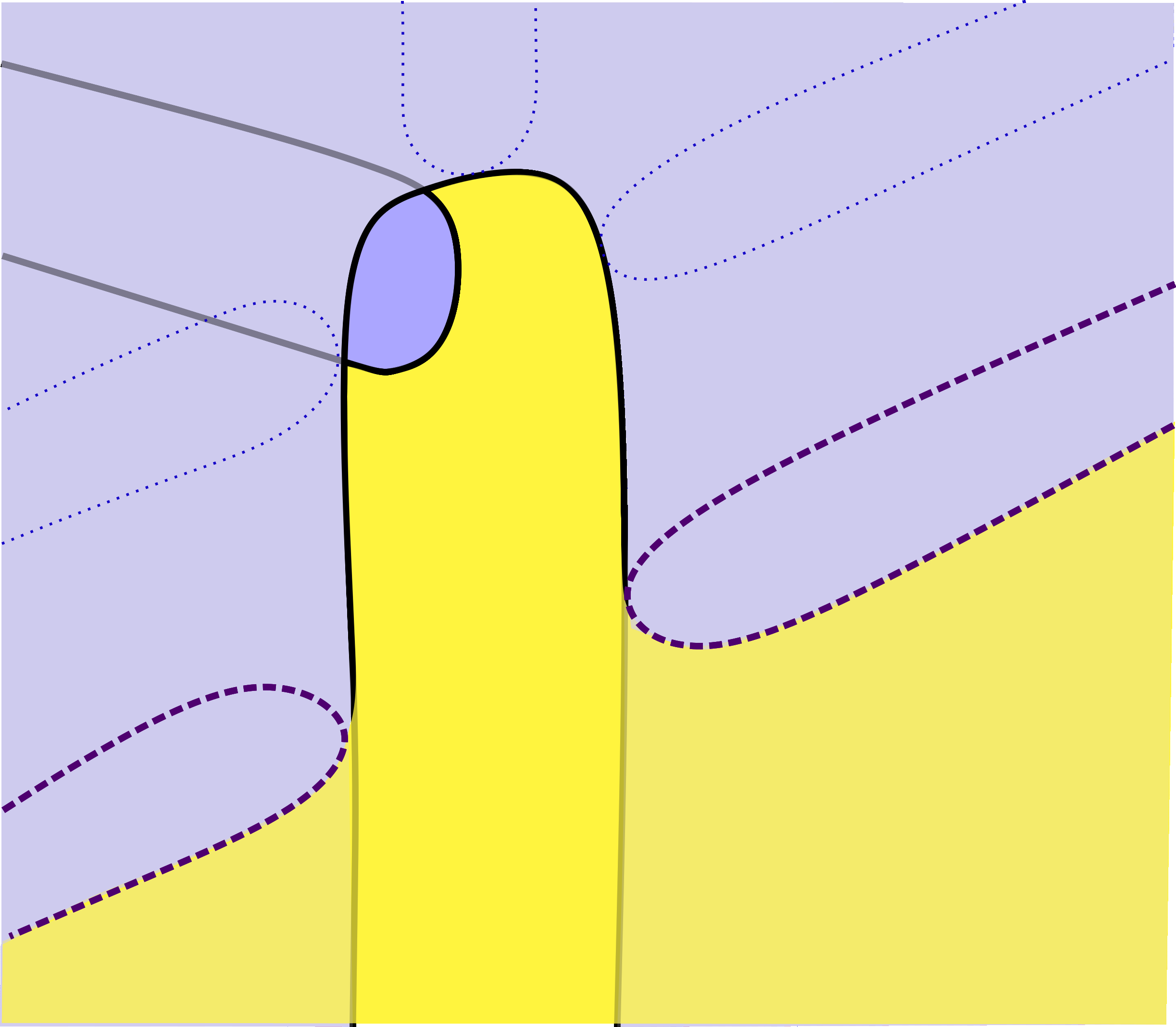} 
        \caption{After a $2$-move on \subref*{fig:a1}.} \label{fig:a2}
    \end{subfigure}
   \vspace{0.1cm}
   
     \begin{subfigure}[t]{0.4\textwidth}
        \centering
        \includegraphics[width=\linewidth]{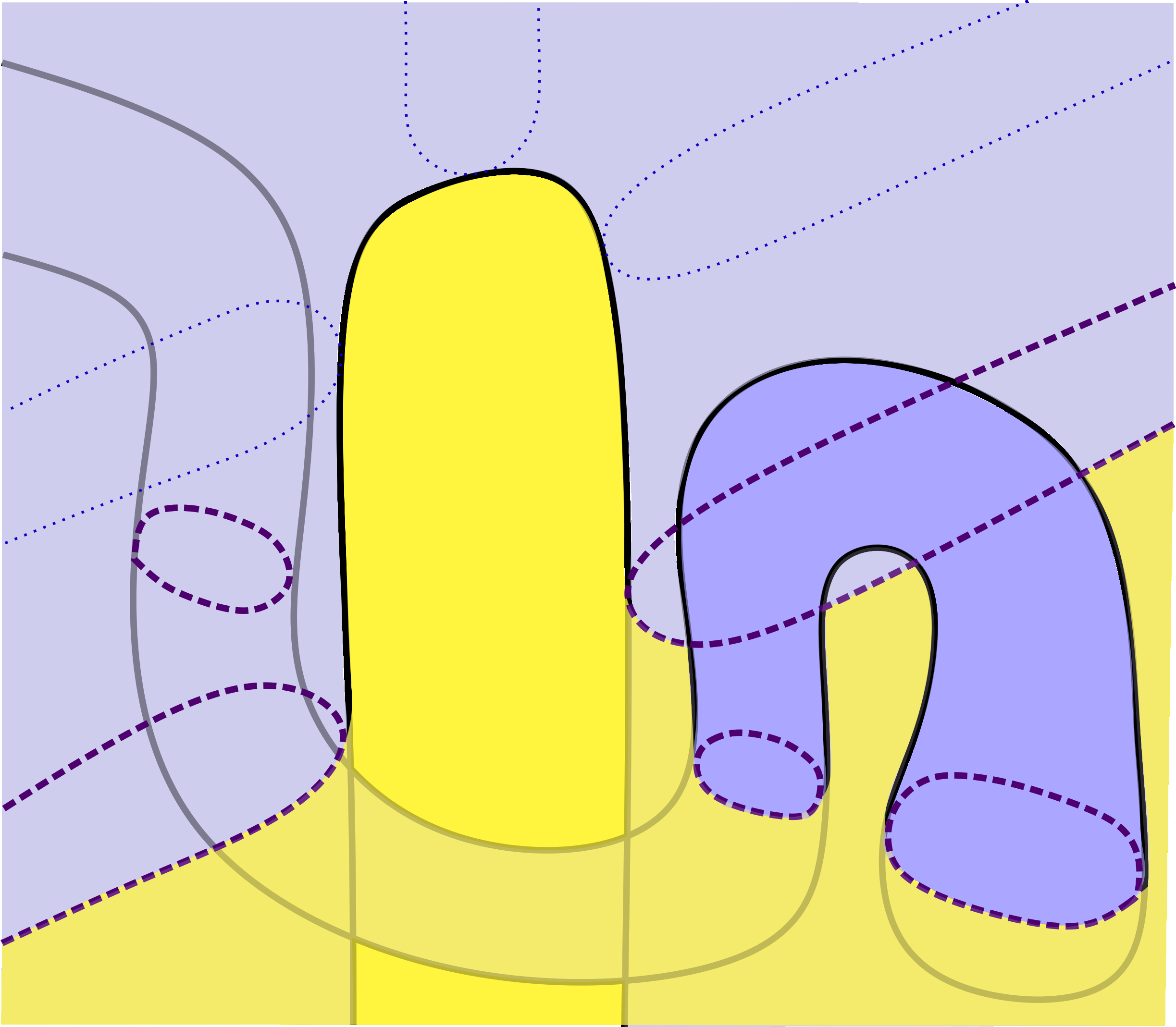} 
        \caption{After a $(0,1)$-isotopy on \subref*{fig:a1}.} \label{fig:a2b}
    \end{subfigure}
    \hfill
    \begin{subfigure}[t]{0.4\textwidth}
        \centering
        \includegraphics[width=\linewidth]{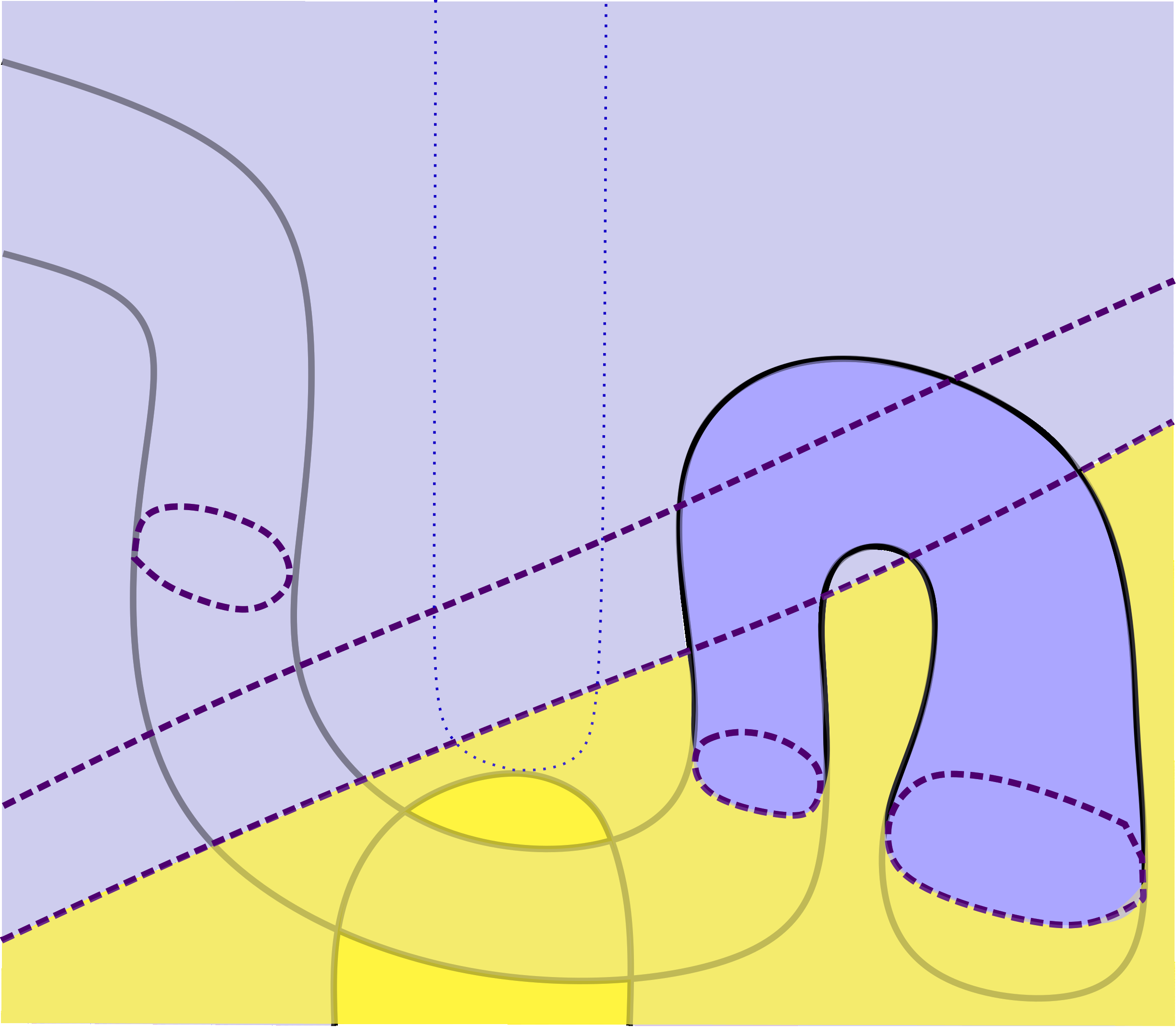} 
        \caption{After a $2$-move on \subref*{fig:a2b}.} \label{fig:a3b}
    \end{subfigure}
   
    \vspace{0.1cm}
    \begin{subfigure}[t]{0.4\textwidth}
    \centering
        \includegraphics[width=\linewidth]{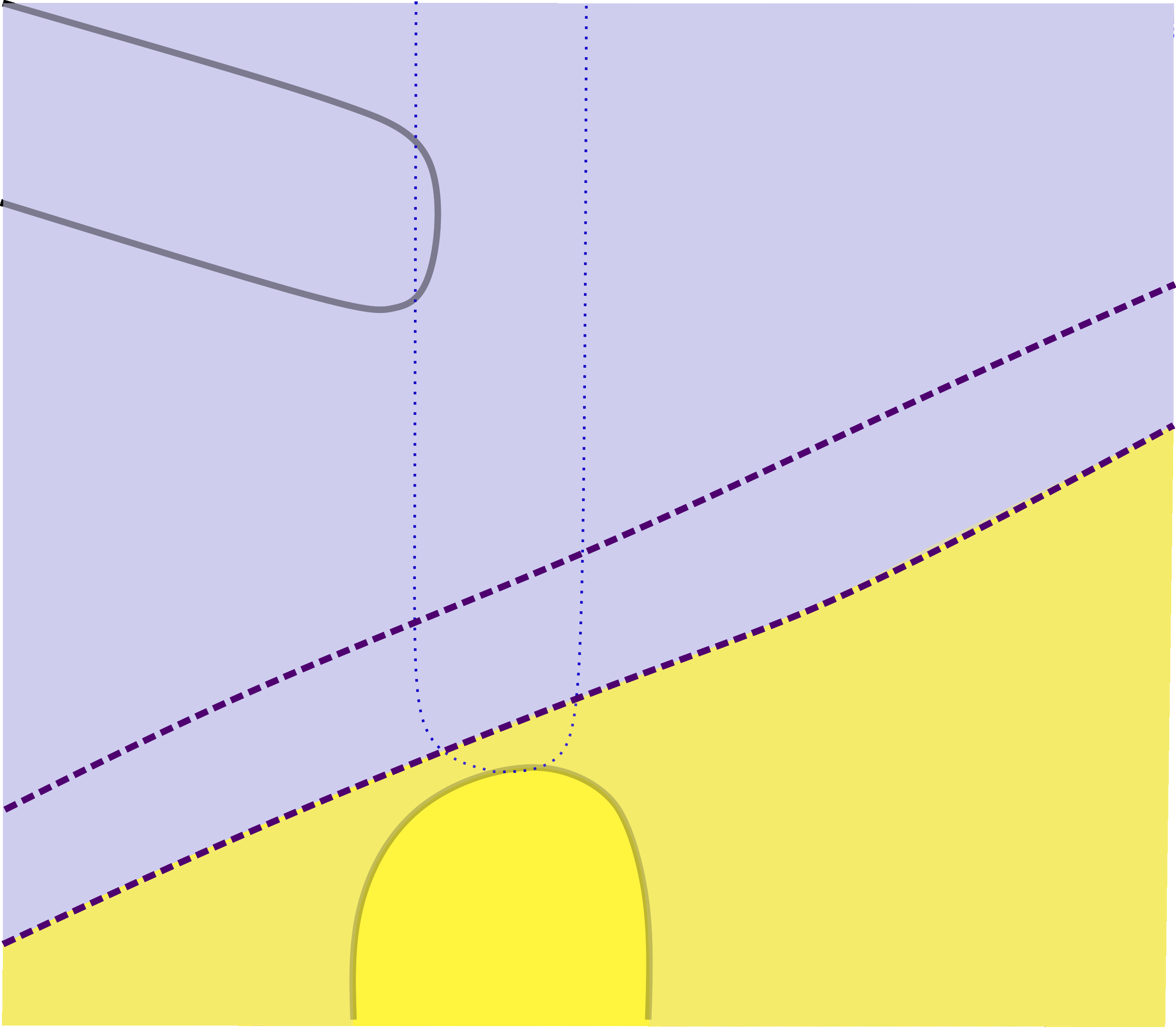} 
        \caption{Final configuration.} \label{fig:af}
    \end{subfigure}
    \hfill  
    \begin{subfigure}[t]{0.4\textwidth}
    \centering
        \includegraphics[width=\linewidth]{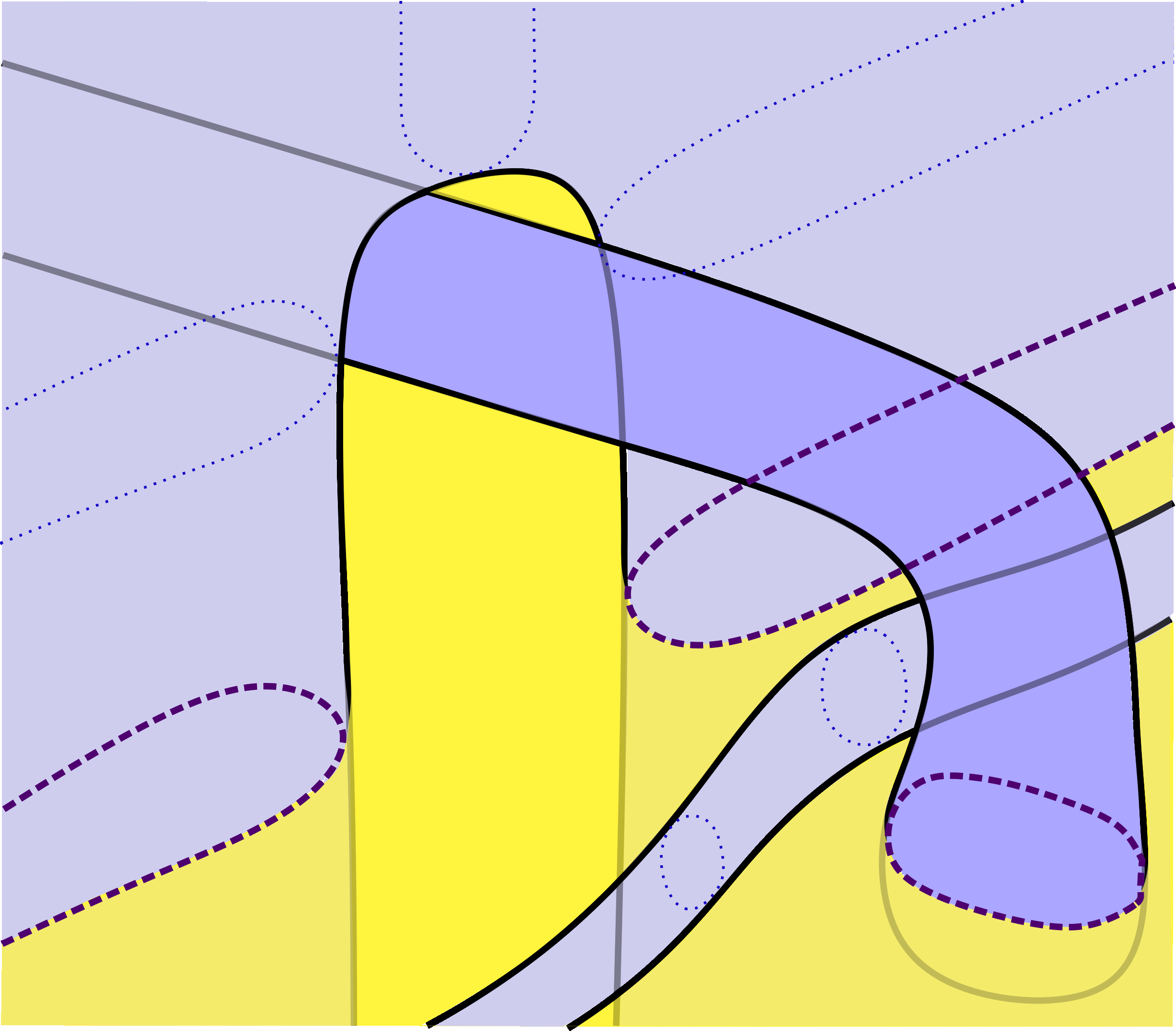} 
        \caption{Additional component.} \label{fig:d}
    \end{subfigure}

    \caption{\label{angelus} Case (\ref{small step}) of Lemma \ref{l: commutation move}. The submanifold $L$ is drawn as a horizontal plane (yellow) and the intersection of the image of $N$ (blue) with $L$ is indicated by a thick dashed line (the thinner dotted lines serve only to highlight the position of $N$). We are given an isotopy of $N$ with type $(2,1)$, which traverses stages \subref*{fig:a1},\subref*{fig:a2},\subref*{fig:af} and Lemma \ref{l: commutation move} provides an alternative isotopy with type $(0,1,1,2,2,2)$ from \subref*{fig:a1} to \subref*{fig:af} that passes through intermediate configurations \subref*{fig:a2b} and \subref*{fig:a3b}. In \subref*{fig:a1} the image of $N$ can be described as consisting of the boundary of an arch whose two pillars cross the plane and a disk in the shape of a prong that passes under the arch and then descends to intersect the plane $L$ in a circle. So far these moves are inessential, but the presence of an additional component, as in \subref*{fig:d}, suffices to render them essential. Notice that the vertical direction is here reversed with relative to the text for the sake of visual clarity.      }
\end{figure}   
    
    \renewcommand{\t}[0]{T}
    \newcommand{\q}[0]{Q}
    \begin{lemma}
    	\label{l: big arranging pushes sublemma} We may replace the restriction of $G$ to $M\times[-1,t^{\star}]$ by some concatenation $\tilde{G}*G^{-r}\dots*G^{0}$ which has the same end-points and is such that $\iota^{*}(\tilde{G})$ is $L$-clean and for $-r\leq i\leq 0$ the isotopy $\iota^{*}(G^{i})$ has a single tangency point of dimension $m-1$.\footnote{Equivalently, $(\iota^{*}(G^{-i}))^{op}$ has a single tangency of dimension $0$.} Moreover, if for $-r\leq i\leq 0$ we let 
    	$$ \tilde{g}:=\tilde{G}_{1},\quad \tilde{f}:=\tilde{g}\circ\iota, \quad\tilde{N}:=\tilde{f}(N),\quad g_{i}:=G^{i}_{1}, \quad f_{i}:=g_{i}\circ\iota,\quad N_{i}:=f_{i}(N),$$
    	then we have 
    	\begin{itemize}
    		\item a $0$-cylinder (i.e., an arc) $D_{i}$ in $M$ bridging $N_{i}$ and $L$,
    		\item a tubular neighbourhood $\xi_{i}:(-1,1) \times B^{m-1}$ of $\mathring{D}_{i}$,
    		\item a diffeotopy $\chi^{i}$ of $Id_{M}$ such that 
    		 $$supp(\chi_{i})\cap U \subseteq im(\xi_{i}),$$
    	\end{itemize}
    	 and the following holds for $-r\leq i\leq 0$:
    	\begin{enumerate} [(i)]
    		\item \label{isotopy} $G^{i}_{t}=\chi^{i}_{-t}g_{i}$ for all $t\in I$,
    		\item \label{compatible} $\xi_{i}$ and $(\chi^{i}_{t}\circ f_{i})_{t \in I}$ are compatible in the sense of Lemma \ref{l: non-destructive isotopies},
    		\item \label{adapted} $\xi_{i}$ is adapted to $\xi$,  
    		\item \label{nice annuli} the intersection $\mathcal{A}:=N_{i}\cap \bar{T}$ (respectively, $\tilde{N}\cap\bar{T}$) is mapped by $\phi$ to a finite union of $\xi$-spherical shells,
    	 \item \label{different critical points} no set of the form $\pi_{m-1}^{-1}(\nu)\cap Q'$ for $\nu\in(\beta,0)$ is tangent to $\phi(\mathcal{A})$ at two different points. 
    	\end{enumerate}
    \end{lemma}
    \begin{subproof}
   	   We inductively construct $G^{0},G^{-1},\dots$, modifying $G_{\restriction M\times[-1,t^{\star}]}$ as we go along. Suppose we are given $G'*G^{-k+1}*G^{-k+2}\dots G^{0}$, 
   	   where $G^{-j}$, $j\leq k$ satisfy the assumptions of the conclusion and $H':=\iota^{*}(G')$ is $L$-compatible, with a type consisting $r-k+1$ many entries equal to $m-1$.
   	    
   	   Since we have the freedom to choose $\chi^{-k}$ in an arbitrary neighbourhood of $D^{-k}$, the same arguments as in the beginning of the proof of Lemma \ref{l: commutation move} provides some $0$-bridge (arc) $D_{-k}$ between $N_{-k}$ and $L$, some tubular neighbourhood $\xi_{-k}:I\times B^{m-1}$ of $D_{-k}$ and a diffeotopy 
   	   $\chi^{-k}$ compatible with $\xi_{-k}$ such that $(\chi^{-k}_{t}\circ f_{-k})_{t\in I}$ has a single tangency point
   	   of dimension $0$ and there is some compactly supported diffeotopy $G''$ of $Id_{M}$ such that  $G''_{1}=\chi_{1}^{-k}g_{-k}$ and $\iota^{*}(G'')$ is $L$-compatible, with $r-k-1$ tangency points of dimension $m-1$.
   	   
   	   By part (\ref{nice annuli}) of our inductive assumption, the intersection of $N_{-k}$ with $\bar{T}$ is the union of some finite collection of $\xi$-spherical shells.
   	   This, together with Corollary \ref{c: refined transversality} and Observation \ref{o: conjugation}, allows us to add the following assumptions one by one.
   	   \begin{itemize}
   	   	\item \label{1} The two points in $\partial D_{-k}=N_{-k}\cap D_{-k}$ lie outside of $\t$. Notice that this might involve sliding $\partial D_{k}$ along some of the finitely many $\xi$-spherical shells in $N_{-k}\cap \t$. 
   	   	\item The intersection $D_{-k}\cap\t$ is contained in $\q$.
   	   	\item \label{2} The arc $D_{-k}$ intersects $D_{1}$ transversely. 
   	   	\item \label{2.24} We have $D_{-k}\cap D_{1}\subseteq D_{1}\cap im(\xi)$.  
   	   	\item \label{2.5} There is $\nu>0$ such that for each $p\in D_{-k}\cap D_{1}$ some neighbourhood of $p$ in $D_{1}$ is of the form $\xi((-\nu,\nu)\times\{v\})$. This can be achieved, for instance, using the inverse function theorem and Lemmas \ref{l: isotopy extension} and \ref{l: tubular neighbourhoods}.
   	   \end{itemize}
   	    
   	    \begin{lemma}
   	    	We may assume $D_{-k}\cap \t=D_{-k}\cap\q$ is the intersection of finitely many arcs of the form 
   	   	 $\xi([-3,3]\times\{v\})$.  
   	    \end{lemma}
        \begin{subsubproof}
        	Using (\ref{dul})-(\ref{ned}) one can easily find an open set $U'\subseteq\overline{U}'\subseteq B^{m-1}(2)$
        	and some embedding $\zeta:[-4,4]\times \overline{B}^{m-1}(2)\to U\setminus L$ such that:
        	\begin{itemize}
        		\item $\zeta_{\restriction [-3,3]\times\overline{B}^{m-1}}=\xi_{\restriction [-3,3]\times\overline{B}^{m-1}}$ 
        		\item $D_{-k}\cap T\subseteq\zeta([-3,3]\times U')$,
        		\item  $\zeta([-4,4]\times U')\cap N_{-k}=\emptyset$. 
        	\end{itemize}
         It is then easy to find an element $h\in\D_{c0}^{N_{-k}\cup L}(M)$ of $Id_{M}$ such that there is some orientation preserving diffeomorphism $\qopa$ of $[-4,4]$ mapping $[-\epsilon,\epsilon]$ onto $[-3,3]$ for which   
   	   	 $\zeta(h(\zeta^{-1}(s,v)))=(\qopa(s),v)$ for any $(s,v)\in[-\nu,\nu]\times U'$. We obtain the desired result by replacing $D_{-k}$ by its image by $h$ (using Observation \ref{o: conjugation}).
        \end{subsubproof} 
   	   For any $h\in\D_{c0}^{N_{-k}\cup L}(M)$ we may also replace $\xi_{-k}$ with $h\circ\xi^{-k}$ and 
   	   $\chi_{t}^{-k}$ with $h\chi^{-k}_{1}h^{-1}$. By combining this with the equivalence of tubular neighbourhoods (Lemma \ref{l: tubular neighbourhoods}) and Lemma \ref{l: isotopy extension}, as well as Observation \ref{o: conjugation}, one may assume $\xi_{-k}$ is adapted to $\xi$ (\ref{adapted}). We can also require that  $\phi(N_{-k-1}\cap\t)=\phi(\chi^{-k}_{1}(N_{-k})\cap\t)$ is a finite union of preimages of $\xi$-spherical shells (\ref{nice annuli}). We leave to the reader the observation that the induction above can be carried on in such a way as to ensure property (\ref{different critical points}) is also satisfied.
    \end{subproof}
    
    \renewcommand{\O}[0]{\Omega}
    \newcommand{\GV}[0]{\Omega^{\uparrow}}
    \newcommand{\HGV}[0]{\hat{\Omega}^{\uparrow}}
    \newcommand{\gv}[0]{g_{\uparrow}}
    We keep the notation of \ref{l: big arranging pushes sublemma}.
   	It is easy to find some diffeotopy $\GV$ of $Id_{M}$ supported on $T$ such that $\HGV_{t}:=\phi\circ G_{t}\circ\phi^{-1}$ satisfies:
   	\begin{itemize}  
   		\item $\HGV_{t}(\{y_{m}\leq\beta\})\subseteq\{y_{m}\leq\beta\}$ for all $t\in I$, 
   	 	\item there is a diffeomorphism $\lambda_{t}$ of $(-1,1)$, restricting on $(\frac{\alpha+\beta}{2},0)$ to the translation by $\frac{(\alpha+\beta)(t+1)}{4}$, such that 
   	 	$$\HGV_{t}(\underline{y})=(y_{1},\dots y_{m-1},\lambda_{t}(y_{m}))$$
   	        for any $t\in I$ and $\underline{y}\in (-1,1)^{m-2}\times(0,\beta)\times(-2,2)$.  	  
   	 \end{itemize}
    
    For simplicity, write $\gv:=\GV_{1}$. Notice that $supp(\GV)\cap N_{\star}\subseteq T\cap N_{\star}=\emptyset$, so that 
    \begin{equation}
    	\label{eq:omega and star}\tag{$\dagger$} \GV_{t}\circ f_{\star}=f_{\star}\quad\quad\text{for all $t\in I$ and, in particular, $\quad \gv\circ f_{\star}=f_{\star}$.}
    \end{equation}    
    The set $\gv(im(\xi))$ is disjoint from $D_{1}$, so we may also assume 
   	\begin{equation}
   	\label{disjoint supports} \tag{$\ddag$} supp(\gv\chi^{i}_{t}\gv^{-1})\cap supp(\chi^{1})=\emptyset\quad \text{for all $-r\leq i\leq 0$ and $t\in I$}.
   	\end{equation}
   	
   		\newcommand{\GG}[0]{\Upsilon^{+}}
   	\newcommand{\GGG}[0]{\Upsilon^{-}}
   	   	
   	It follows from Lemma \ref{l: compatible redefinition}, Claim \ref{c: big hessian claim} and item (\ref{nice annuli}) in Lemma \ref{l: big arranging pushes sublemma} that the isotopy $\tilde{f}^{*}(\GV)$ is $L$-compatible and its type contains only entries of type $0$ and $m-2$.\footnote{It is hard to assert more about this sequence other than the first entry is a $0$ and the final one $m-2$.}  
    Consider next the diffeotopy 
    $$\GG:=(\chi^{1}_{t}\gv \tilde{g})_{t\in I}.$$ 
      
    \renewcommand{\HH}[0]{\Xi}
     Consider also the following diffeotopy, where the concatenation starts with $i=-r$ and ends with $i=0$:  
     $$\HH:=\textstyle{\bigast_{i=-r}^{0}}(\chi^{i}_{-t}g_{i}g_{\star}^{-1})_{t\in I}.$$ 
    Recall that in the notation of Lemma \ref{l: big arranging pushes sublemma} we have
    $g_{\star}=G^{0}_{1}=g_{0}$, so that  
    $$\HH_{-1}=\tilde{g}g_{\star}^{-1},\quad\quad \HH_{1}=Id_{M}, \quad\quad supp(\HH)\subset\textstyle\bigcup_{l=0}^{r}supp(\chi^{-l}).$$
     Next, consider the compactly supported diffeotopy 
   	$$\GGG:=(\gv\Xi_{t}\gv^{-1}\chi^{1}_{1}\gv g_{\star})_{t\in I}.$$
   	
   	Property  (\ref{disjoint supports}) implies that $\gv \chi^{i}_{t}\gv^{-1}$ commutes with $\chi^{i}_{1}$
   	for all $-r\leq i\leq 0$ and $t\in I$. Therefore, so does the element $\HH_{-1}$ and we have
    $$
        \GG_{1}=\chi^{1}_{1}\gv \tilde{g}=\chi^{1}_{1}\gv\HH_{-1}\gv^{-1}\gv g_{\star}=\gv \HH_{-1}\gv^{-1}\chi^{1}_{1}\gv g_{\star}=\GGG_{-1}.
    $$ 
    By using $\HH_{-1}\circ f_{\star}=\tilde{f}$ in the first equality below, (\ref{eq:omega and star}) in the second and (\ref{disjoint supports}) in the third, we also get
    \begin{align*}
    (\gv\circ \tilde{f})_{\restriction f_{\star}^{-1}(supp(\chi^{1}))}
    =&(\gv\Xi_{-1}\gv^{-1}\gv \circ f_{\star})_{\restriction f_{\star}^{-1}(supp(\chi^{1}))}=\\
    &(\gv\Xi_{-1}\gv^{-1}\circ f_{\star})_{\restriction f_{\star}^{-1}(supp(\chi^{1}))}=(f_{\star})_{\restriction f_{\star}^{-1}(supp(\chi^{1}))},
    \end{align*}
    and since $f_{\star}^{*}(\chi^{1})$ is $L$-compatible with type $(m-2)$, so must be $$\iota^{*}(\GG)=(\chi^{1}_{t}\gv\circ\tilde{f})_{t\in I}=(\gv\circ\tilde{f})^{*}(\chi^{1}).$$ 
          
   	For $-r\leq j\leq 0$ the properties of $\xi_{j}$, the description of $\GV$ and item (\ref{adapted}) of Lemma \ref{l: big arranging pushes sublemma} imply that for any point $q\in \gv(D_{j})\cap L$ the  
   	the fiber image through $q$ of the tubular neighbourhood $\gv\circ\xi_{j}$ is contained in $L$. 
   	Property (\ref{eq:omega and star}) implies
   	 $$\iota^{*}(\GGG)=(\gv\Xi_{t}\gv^{-1}\chi^{1}_{1}\circ f_{\star})_{t\in I},$$
   	and (\ref{disjoint supports}) implies: 
   	$$(f_{\star})_{\restriction  f_{\star}^{-1}(\gv supp(\HH))}=(\chi_{1}^{1}\circ f_{\star})_{\restriction f_{\star}^{-1}(\gv supp(\HH))}.$$ 
 
   	The above remarks, Lemma \ref{l: compatible redefinition} and properties (\ref{first property}) and (\ref{second property}) in Lemma \ref{l: non-destructive isotopies}, imply that $\iota^{*}(\GGG)$ is $L$-compatible and purely destructive. For each $-r\leq i\leq 0$ the type contains one $(m-1)$-entry corresponding to the original tangency of $\iota^{*}(G^{i})$ and  	 
   	two more for each of the points in the intersection $D_{i}\cap D_{1}$.
   	
   	\newcommand{\Gd}[0]{\Upsilon^{\downarrow}}
   	Finally, consider the diffeotopy 
   	$$\Upsilon^{\downarrow}:=(\chi^{1}_{1}\GV_{-t}g_{\star})_{t\in I}.\, $$
   	 Notice that $\GGG_{1}=\chi^{1}_{1}\gv g_{\star}=\Gd_{-1}$. 
    It follows from (\ref{eq:omega and star}) that $\iota^{*}(\Gd)_{t}$ is constant and thus $L$-clean.
    This concludes the proof of Lemma \ref{l: commutation move} in case (\ref{small step}), since as $G'$ we may now take the concatenation of the diffeotopies: 
	    $$
	    \tilde{G},\quad\quad \tilde{g}^{*}(\GV),\quad\quad \GG,\quad\quad \GGG,\quad\quad \Gd,\quad\quad G_{\restriction M\times[s_{1},1]} .
	    $$
   \end{proof}

   \begin{corollary}
   	\label{c: compatibility and decompositions}Suppose we are given a nice $m$-triple $(N,M,L)$, a compactly supported diffeotopy $G$ of $Id_{M}$ and $\iota:N\to M$ such that $\iota^{*}(G)$ is $L$-compatible. Then there are compactly supported diffeotopies $G',G''$ such that:
   	\begin{itemize}
   		\item $G'_{-1}=Id_{M}$, $G'_{1}=G''_{-1}$, $G''_{1}=G_{1}$
   		\item $\iota^{*}(G')$ is $L$-compatible and purely constructive,
   		\item $\iota^{*}(G'')$ is $L$-compatible and purely destructive.
   	\end{itemize}
   \end{corollary}
   \begin{proof}
   	We may assume the type of $H$ is of the form $$\sigma^{\frown}(m-1,m-1,\dots m-1,d)^{\frown}\tau,$$ where  $\sigma$ contains no occurrences of $m-1$ and $d\leq m-2$. The proof is an easy induction on the length of $\tau$ using Lemma \ref{l: commutation move}.
   \end{proof}

\section{Neighbourhoods of the identity contain fix-point stabilizers of finite embedded simplicial complexes}
  
  \newcommand{\kk}[0]{\mathfrak{H}} 
  \newcommand{\subg}[1]{\langle #1 \rangle}
  \label{s: neighbourhoods are rich}
  \newcommand{\mylabel}[2]{#2\def\@currentlabel{#2}\label{#1}}
\begin{notation}
	  \label{n: G} Let $M$ be a smooth $m$-dimensional manifold. We consider $\hh$, $\gg$ as in one of the following two situations: 
  \begin{enumerate}
    \item[\mylabel{conddiff}{$(\mathsf{Df})$}] $\hh:=\D_{c0}(M)\leq\gg\leq\D(M)$,
  	\item[\mylabel{condhom}{$(\mathsf{Tp})$}] $\hh:=\H_{c0}(M)\leq\gg\leq\H(M)$.
  \end{enumerate}
  We will also denote by $\T_{co}$ the restriction of the compact-open topology to $\gg$ and $\T$ a group topology on $\gg$ coarser than the compact-open topology. 
  
  In the setting \ref{conddiff}
  by an embedded ball in $M$ we will mean a smoothly embedded ball. In case \ref{condhom} by such term we will mean a collared ball in $M$. 
\end{notation}
  
  All the results of this section will take place by default in either of the two settings of \ref{n: G}. 
   
  \begin{observation}
  \label{o: generation compact-open topology} If $\T$ is strictly coarser than $\tc$, then there do not exist $m$-balls $D,D'$ in $M$ with $D\subseteq D'$
  and $\mathcal{V}\in\nd$ such that $g\cdot D\subseteq D'$ for all $g\in \mathcal{V}$.
  \end{observation}
  \begin{proof}
  Assume for the sake of contradiction such $D$, $D'$ and $\V$ exist. By two applications of either Fact \ref{f: transitive on disks} in case \ref{conddiff} or Fact \ref{f: transitive on disks in homeo} in case \ref{condhom} one can easily see that, given $D,D'$, any $p\in M$ and any $\epsilon>0$ small enough there is $h\in\kk$ such that  $p\in h\cdot \mathring{D}\subseteq h\cdot D'\subseteq \nn_{\epsilon}(p)$. Then $g\cdot D_{0}\subseteq \nn_{\epsilon}(p)$ for any $g\in\V^{h^{-1}}$, where $D_{0}=h\cdot D$. It follows easily that the conjugates of $\V$ by the action of $\hh$ generate a system of neighbourhoods of $\tc$ at the identity (here one uses that basic neighbourhoods of the identity are of the form $\V_{K,\epsilon}$ for compact $K$).
  \end{proof}
   
  \begin{lemma}
  \label{l: mixing disks}Let $D,E_{1},\dots E_{k}$ be $m$-balls in $M$ with $E_{j}\nsubseteq D$ for $1\leq j\leq k$. Let also $K\subseteq M$ be a compact subset and $\epsilon$ a positive real. Then there exists $h_{1},\dots h_{d}\in \kk$, which fix $D$ and such that for any $1\leq j_{0},j_{1},\dots j_{d}\leq k$ and any connected component $C$ of the complement of $\bigcup_{l=1}^{d}h_{l}\cdot E_{j_{l}}$ either:
  \begin{itemize}
  \item $K\cap C=\emptyset$,
  \item $diam(C)<\epsilon$,
  \item or $C\subseteq \nn_{\epsilon}(D)$.
  \end{itemize}
  \end{lemma}
  \begin{proof}
  Consider first case \ref{conddiff} of \ref{n: G}. We may assume that the $E_{j}$ are mutually disjoint and disjoint from $D$. 
  There exists some $m$-ball $D'$ in $M$ with $D\subseteq\mathring{D}'\subseteq D'\subseteq\nn_{\epsilon}(D)$, as well as some compact codimension $0$ submanifold $N\subseteq M$ such that $D'\subseteq K\subseteq N$. Let $N'=N\setminus\mathring{D}'$. 
  
  Let $\delta=\min\{\epsilon,d(D,M\setminus \mathring{D}')\}$. 
  By compactness, there exist finite collections of $m$-balls $\{C_{i}\}_{i=1}^{r},\{C_{i}'\}_{i=1}^{r}$ in $M$ 
  such that:
  \begin{itemize}
  	\item $C_{i}\subseteq\mathring{C}'_{i}\subseteq M\setminus \mathring{D}$, 
  	\item $N'\subseteq\bigcup_{i=1}^{r}C_{i}$,
  	\item and $diam(C'_{i})<\delta$.
  \end{itemize}
  For $1\leq i\leq r$ choose two families of disjoint $m$-balls $\{E_{i,j}\}_{1\leq j\leq k},\{E_{i,j}'\}_{1\leq j\leq k}$ in  $U_{i} := \mathring{C}'_{i}\setminus C_{i}$ such that
  for any $1\leq j_{1},j_{2}\leq r$ the set $E_{i,j_{1}}\cup E'_{i,j_{2}}$ separates $C_{i}$ from $M\setminus\mathring{C}'_{i}$. This can be done, for instance, by taking a diffeomorphism $\phi: \overline{U}_{i}\to S^{m-1}\times I$, choosing some vector $v_{0}\in S^{1}$, some small $\delta>0$, considering the disks 
  $$K_{j}^{\nu}:=\{(v,f_{\nu,j}(v))\,|\,\norm{v+\nu v_{0}}>\delta\}, \quad f_{\nu,j}(v)=-\frac{\nu\delta j}{3k}+\frac{\norm{v+\nu v_{0}}}{2}$$ 
  for $1\leq j\leq k$, $\nu\in\{1,-1\}$, and then letting $E_{i,j}$ and $E'_{i,j}$ be given by suitable thickenings of  $K_{j}^{1}$ and $K_{j}^{-1}$.
   
  By Fact \ref{f: transitive on disks}, there are $h_{i},h'_{i}\in\kk$
  such that $h_{i}\cdot E_{j}=E_{i,j}$, $h'_{i}\cdot E_{j}=E'_{i,j}$. We claim that the set 
  $\{h_{i},h'_{i}\}_{1\leq i\leq r}$ satisfies the desired properties ($d=2r$).  
  Indeed, fix some arbitrary choice of $j_{i},j'_{i}\in\{1,\dots k\}$, $1\leq i\leq d$ and let $\{\mathcal{C}_{\lambda}\}_{\lambda\in\Lambda}$ be the collection of connected components of $X:=M\setminus\bigcup_{i=1}^{r}(E_{i,j_{i}}\cup E'_{i,j'_{i}})$.
  If $p\in X\cap N'$, then $p\in \mathcal{C}_{\lambda(p)}$ for some $\lambda(p)\in\Lambda$ such that     
  $diam(\mathcal{C}_{\lambda(p)})<\delta$. If we write $Y :=\bigcup_{p\in N'}\mathcal{C}_{\lambda(p)}$, then any connected component of $ X\setminus Y$ is contained in some connected component of the superset $Z:=\mathring{D}\cup (M\setminus N)$ of $X\setminus Y$. Since one of the latter is contained in $D'$ and the other one is disjoint from $K$, the desired conclusion follows.
  
  In case \ref{condhom} of \ref{n: G} we can use an almost identical proof, with the only difference that the existence of $N$ follows from triangulability, $\phi$ is taken to be a homeomorphism and $h_{i},h'_{i}$ can be found using Fact \ref{f: transitive on disks in homeo}.   
  \end{proof}
   
  \begin{remark}
  	We believe one may take $d=m$ in the lemma to be a function of $m=dim(M)$ only.
  \end{remark}

  \begin{lemma}
  \label{l: spread all over}Assume that $\T$ is strictly coarser than $\tc$. Let $D,E_{1},\dots E_{k}$ be $m$-balls in $M$ and $\V\in\nd$. Then for any $\V\in\nd$ there exists $g\in\V$ such that $g\cdot D\cap E_{i}\neq\emptyset$ for all $1\leq i\leq k$.
  \end{lemma}
  \begin{proof}  
   Up to making $D$ smaller, we may assume that $E_{i}\nsubseteq D$ for all $1\leq i\leq k$. 
   Suppose that some $\V\in\nd$ fails to satisfy the property, so that for any $g\in\V$ there is 
  some $1\leq i\leq k$ such that $g\cdot D\cap E_{i}\neq\emptyset$.

  Pick any arbitrary $g^{*}\in\kk\setminus\{1\}$ with $supp(g^{*})\subseteq D$ 
  and let $\V_{0}=\V_{0}^{-1}\in\nd$ be such that $g^{*}\nin\V_{0}^{3}$. Choose some compact set $K\subset M$ and $\epsilon>0$ such that $\V_{K,\epsilon}\subseteq\V_{0}$ and $\nn_{\epsilon}(D)$ is contained in a ball $D'$.
  
  Let $h_{1},\dots h_{d}$ be the elements resulting from applying Lemma \ref{l: mixing disks} to the compact $K$, the real $\epsilon$ and the balls $D,E_{1}\dots E_{k}$ and let $\V_{1}:=\bigcap_{i=1}^{d}\V^{h_{i}^{-1}}$. Consider the intersection $\V_{2}:=\V_{0}\cap\V_{1}$.
  Our assumption on $\V$ and the conclusion of Lemma \ref{l: mixing disks} implies that for any given $g\in\V_{1}$ at least one of the following possibilities holds:
  \begin{itemize}
  \item $g\cdot D\subseteq D'$,
  \item $diam(g\cdot D)<\epsilon$,
  \item or $g\cdot D\cap K=\emptyset$.
  \end{itemize}
  If $g\in\V_{2}\subseteq\V_{1}$ satisfies the second or third possibility above, then $(g^{*})^{g^{-1}}\in\V_{K,\epsilon}\subseteq
  \V_{0}$ and thus $g^{*}\in\V_{0}^{3}$, contradicting the choice of $\V_{0}$. Hence, the first alternative must always hold, which, by Observation \ref{o: generation compact-open topology}, is contrary to the assumption that $\T\subsetneq\tc$.
  \end{proof}
  
  \begin{definition}
  	\label{d: intrusive} We will refer to any group topology $\T$ that is coarser than the compact-open topology and satisfies the conclusion of Lemma \ref{l: spread all over} as intrusive. 
  \end{definition}
    
   \begin{lemma}
  	\label{l: intrusive homeos}In the setting of \ref{condhom} of \ref{n: G} assume that the group topology $\T$ is strictly coarser than $\tc$ on $\gg$. Then for any $0<\epsilon<\frac{1}{2}diam(M)$ and compact $K\subseteq M$ with $\mathring{K}\neq\emptyset$ there is no $\V\in\nd$ such that $\H_{c0}(M)\cap\V\subseteq\V_{K,\epsilon}$ . 
  \end{lemma} 
  \begin{proof}
  	Fix $K$, $\epsilon$ and $\V\in\nd$ as in the statement. Choose $\V_{0}=\V_{0}^{-1}\in\nd$ with $\V_{0}^{3}\subseteq\V$, then $\delta>0$ such that $\V_{\delta}\subseteq\V$ and finally a closed ball $D$ in $M$ which has diameter less than $\delta$ and is contained in $K$. Since $\epsilon<\frac{1}{2}diam(M)$, there is a pair of collared balls $E_{1},E_{2}\subseteq M$ satisfying the following two properties:
  	\begin{itemize}
  		\item $d(E_{1},E_{2})>\epsilon$, where $d$ stands here for the infimum distance,
  		\item $E_{1}\subseteq K$. 
  	\end{itemize}
  	 Lemma \ref{l: spread all over} applied to $\V_{0}$, $D$ and $E_{1},E_{2}$ implies the existence of $g\in\V_{0}$ such that $g\cdot D\cap E_{i}\neq\emptyset$ for $1=1,2$. This condition, together with the two properties listed above, imply the existence of some element $h\in\H(M)$ supported on $g\cdot D$ such that $h\notin\V_{K,\epsilon}$ (for instance, by Fact \ref{f: transitive on disks in homeo}).   
     We have $h\in\H_{c0}(M)$, by Alexander's trick \cite{alexander1923deformation}. 
     Finally, notice that $h^{g}$ is supported on $D$. Since $diam(D)<\delta$, it follows that
     $$h=(h^{g})^{g^{-1}}\in\V_{\delta}^{g^{-1}}\subseteq\V_{0}^{3}\subseteq\V .$$   
  \end{proof}

  The proof of the following lemma involves a series of nested sublemmas and observations. 
  \begin{lemma}
  \label{l: nerve} Let $\gg,\hh,\tt$ be as in \ref{conddiff} of \ref{n: G} and assume furthermore the group topology $\T$ on $\gg$ satisfies the following:   
  \begin{enumerate}[(i)]
  	\item \label{compact case}$\T$ is intrusive (Definition \ref{d: intrusive}),
  	\item \label{non-compact case} one of the following two conditions holds: 
  	    \begin{itemize}
  	    	\item $M$ is compact and small deformations are close to the identity in $\T$ (see Definition \ref{d: small deformations are dense} and Remark \ref{r: small deformations}),
  	    	\item $\T$ is coarser than $\tc$. 
  	    \end{itemize}
  \end{enumerate}
  Then for any $\V\in\nd$ and any compact set $L\subseteq M$ there exists some finite embedded $(m-1)$-dimensional simplicial complex $\Delta$ such that $\PS{\hh}{\Delta}\subseteq\V$ and $L\cup\rl{\Delta}$ is contained in the closure of one of the connected components $U_{0}$ of $M\setminus\rl{\Delta}$. 
  \end{lemma}
  \begin{proof}
  \newcommand{\trr}[0]{\Omega}	
  Take $\V_{0}=\V_{0}^{-1}\in\nd$ with $\V_{0}^{9}\subseteq \V$. We proceed according to the case in the statement.
  \begin{itemize}
  	\item In case (\ref{compact case}) we let $\epsilon>0$ be such that $\w_{7\epsilon}\subseteq\V_{0}$
  	and set $M':=M$.
  	\item In case (\ref{non-compact case}) we pick some compact set $K\subseteq M$ and $\epsilon>0$ such that 
  	$\V_{K,7\epsilon}\subseteq\V_{0}$ and we let $M'\subseteq M$ be some codimension $0$ compact submanifold of $M$ containing $K\cup L$.
  \end{itemize} 
  In both cases we let $\tr$ be a triangulation of $M'$ in which each simplex has diameter less than $\epsilon$. 
     
  \begin{figure}[t]
  	\includegraphics[width=0.90\textwidth]{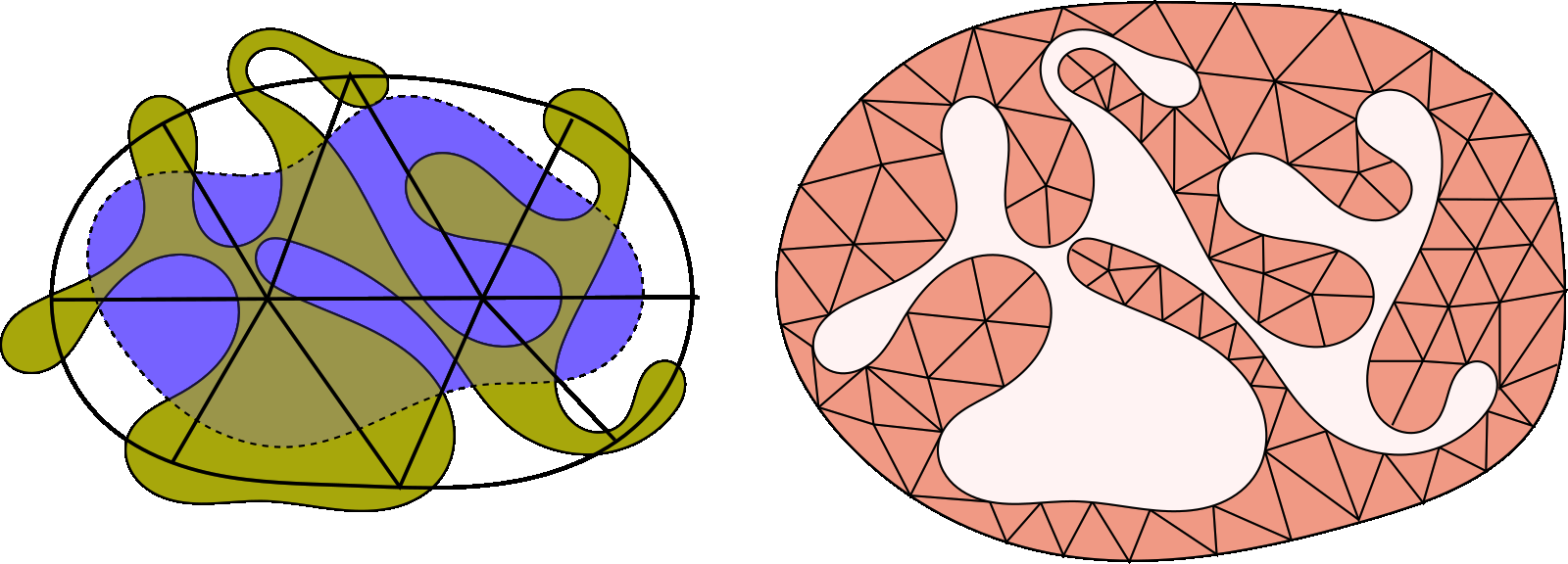}
  	\caption{\label{fig: triangulation}Left: the triangulation $\tr$ of the manifold $M'$ (roughly oval shaped) chosen to contain the set $K\cup L$ (blue area delimited by a dashed line), as well as the ball $g\cdot D$ (yellow). Right: the triangulation $\trr$ of the complement $N'$ of $g\cdot \mathring{D}$ in the submanifold $N\subseteq M$. }	
  \end{figure} 
  
  \newcommand{\A}[0]{\mathcal{A}}
  Fix also some embedded ball $D$ with $diam(D)<\epsilon$. 
  The assumption that $\T$ is intrusive implies that there is
  $g\in\V_{0}$ such that $g\cdot D\cap\mathring{\sigma}\neq\emptyset$ for each $\sigma\in\tr^{m}$. We may also assume $g\cdot\partial D\transv\tau$ for any $\tau\in\tr$, by Lemma \ref{l: wlog transverse}.   
  Let $\A_{1},\dots \A_{r}$ be an enumeration of 
  the connected components of all sets in the collection 
  $$\{\,\sigma\setminus g\cdot \mathring{D}\,|\,\sigma\in\tr^{m}\,\}.$$
  Notice that it is possible that $\sigma\subseteq g\cdot D$ for some $\sigma\in\tr^{m}$. 
  Our transversality assumption implies that for $\sigma\in\tr^{m}$ and $1\leq i\leq r$ we have $\A_{i}\subseteq\sigma$ if and only if $\A_{i}\cap\mathring{\sigma}\neq\emptyset$.   
  For $1\leq i\leq r$ choose $p_{i}\in (g\cdot\partial D)\cap\A_{i}$ and let $\sigma_{i}$ be a simplex of $\tr^{m}$ containing $p_{i}$. Our transversality assumption implies 
  \begin{itemize}
  	\item $\mathring{\A}_{i}$ is connected, 
  	\item $\sigma_{i}$ is uniquely determined by $i$. \footnote{ However, we could have $\sigma_{i}=\sigma_{j}$ for different $1\leq i<j\leq r$.} 
  \end{itemize}  
  Let $N\subseteq M$ be a compact submanifold such that $M'\cup g\cdot D\subseteq int(N)$ (so that $N=M$ in case (\ref{compact case})). 
  Pick $\delta>0$ smaller than each of the following three values: 
  \begin{enumerate}[(a)]
  	\item \label{t1} $\epsilon$,
  	\item \label{t2} $\frac{1}{2}\{d(\sigma,\sigma')\,|\,\sigma,\sigma'\in\tr^{m},\,\sigma\cap\sigma'=\emptyset\}$,
  	\item \label{t3} and $\{d(p_{i},\partial\sigma_{i})\}_{i=1}^{r}$.
  \end{enumerate}
  Let also $\trr$ be a triangulation of $N':=N\setminus g\cdot \mathring{D}$ in which every simplex has diameter less than $\delta$. The general setting is illustrated in Figure \ref{fig: triangulation} above. 
   
	  By perturbing $\trr$ slightly, we may also assume the following property holds: 
	  \begin{equation}
	  	\tag{$\Diamond$}\label{pr:diamond}\text{For any $\sigma$ in $\trr^{m}$ and $1\leq i\leq r$ we have $\sigma\cap\A_{i}\neq\emptyset\leftrightarrow\mathring{\sigma}\cap\mathring{\A_{i}}\neq\emptyset$. }
	  \end{equation}

  \begin{lemma}
     \label{l: small retracting forest}There is some retracting forest $\mathcal{F}=\{(\Gamma_{i},\tau_{i})\}_{i=1}^{r}$ for $\trr$ (see Definition \ref{d: retracting forest}) with the following properties:
  \begin{itemize}
  	\item $diam(|\Gamma_{i}|)<7\epsilon$ for every $1\leq i\leq r$,  
  	\item and $N\subseteq g\cdot D\cup|\mathcal{F}|$. 
  \end{itemize}  	
  \end{lemma}
\begin{subproof}
	
   To begin with, for $1\leq i\leq r$ we choose some $\tau_{i}\in\trr^{m-1}$ such that $\tau_{i}\subseteq g\cdot \partial D\subseteq\partial N'$ and $p_{i}\in\tau_{i}$, and we denote by $\hat{\tau}_{i}$ the unique top-dimensional simplex of $\trr$ containing $\tau_{i}$.  
   We also write 
   $
   \mathscr{V}_{i}:=\{\,\sigma\in\trr^{m}\,|\,\sigma\cap\A_{i}\neq\emptyset\,\}
   $. Clearly, $\hat{\tau_{i}}\in\mathscr{V}_{i}$.
   \begin{observation}
   	\label{o: connectedness} The graph spanned by $\mathscr{V}_{i}$ in $\G(\trr)$ is connected.
   \end{observation}
    \begin{subsubproof}
      Pick some arbitrary $\sigma\in\mathscr{V}_{i}$.  By assumption, $M'\subseteq N$ and therefore $\sigma_{i}\subseteq\bigcup\mathscr{V}_{i}$. By (\ref{pr:diamond}) there is a smooth arc $\alpha$ inside $\mathring{\A_{i}}$ which connects $\sigma$ with a point in the interior of $\hat{\tau}_{i}$, while avoiding $\trr^{(m-2)}$ (take $\alpha$ to be in general position with respect to the simplices of $\trr^{m-2}$). The arc $\alpha$ determines a combinatorial path in $\G(\trr)$ between $\sigma$ and $\hat{\tau}_{i}$ entirely contained in $\mathscr{V}_{i}$. 
    \end{subsubproof}
   
   We construct sets $\mathscr{D}_{i}\subseteq\mathscr{V}_{i}$, $1\leq i\leq r$ as follows. Assuming $\mathscr{D}_{1},\dots\mathscr{D}_{i-1}$ are already given, we take 
   $\mathscr{D}_{i}$ to be the set of vertices in the connected component of $\mathscr{V}_{i}\setminus\bigcup_{l=1}^{i-1}\mathscr{D}_{l}$ that contains $\hat{\tau}_{i}$ (we are referring to the graph induced on $\mathscr{V}_{i}$ on $\G(\trr)$). Notice this set is non-empty by bound (\ref{t3}) on $\delta$.

  \begin{figure}[t]
  	\includegraphics[width=0.55\textwidth]{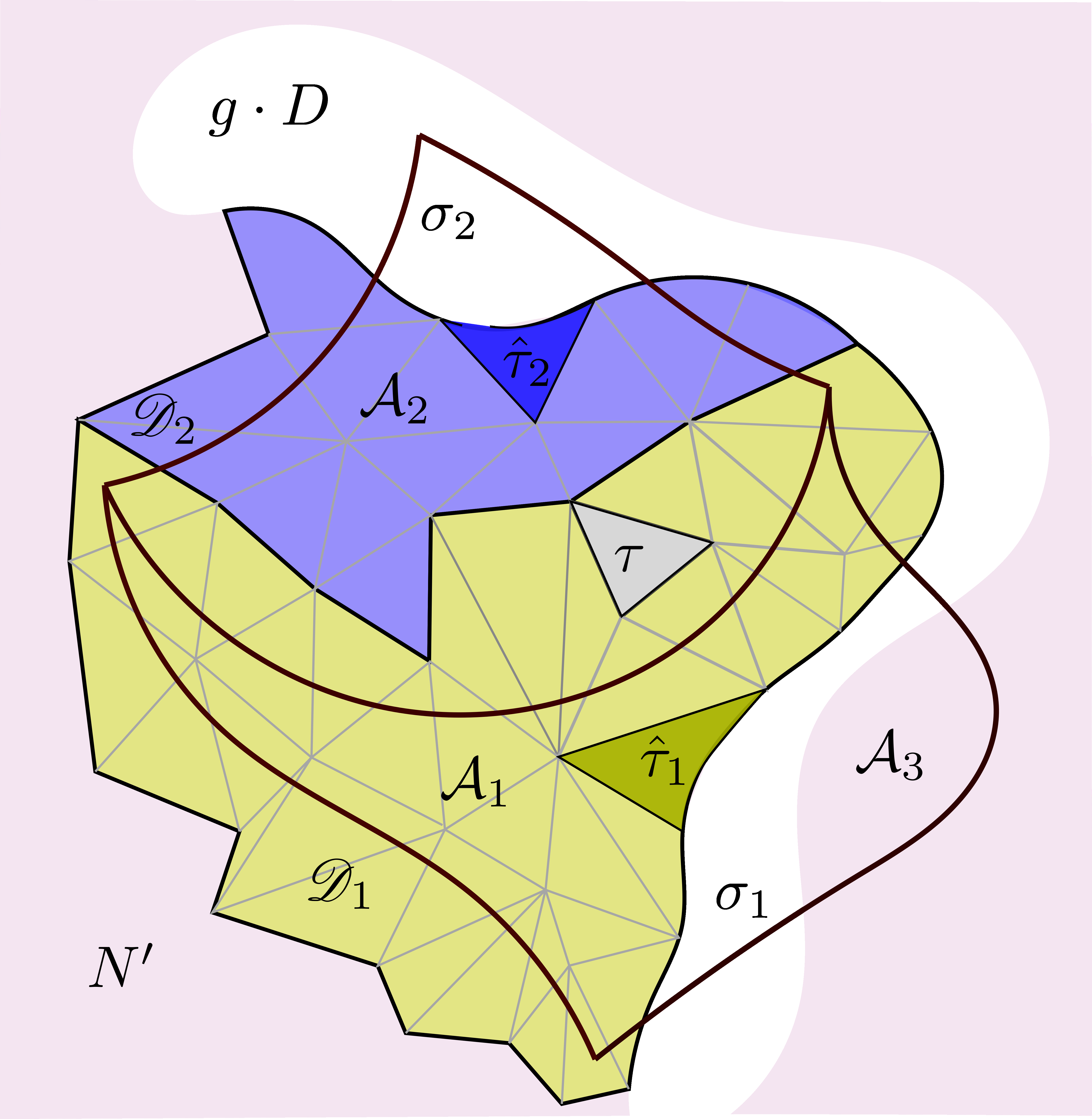}
  	\caption{\label{fig: leftover}The sets $\mathscr{D}_{1}$ (yellow) and $\mathscr{D}_{2}$ (blue), with the two simplices $\hat{\tau}_{1},\hat{\tau}_{2}\in\trr^{m}$ distinguished by their darker color. The simplex $\tau\in\trr^{m}$ is contained in $\sigma_{2}\in\tr^{m}$, however it belongs to $\mathscr{E}_{1}$ rather than $\mathscr{E}_{2}$ at the end of the construction. We have also drawn the simplices $\sigma_{1},\sigma_{2}\in\tr^{m}$ as two large triangles with curved edges (dark red). Notice that here $\sigma_{1}=\sigma_{3}$ but $\sigma_{2}=\sigma_{i}$ only if $i=2$. }
  \end{figure}
   
   \begin{observation}
   	 \label{o: leftover components}For $1\leq i\leq r$ let $\mathfrak{X}_{i}$ be the collection consisting of every $\mathscr{C}\subseteq\G(\trr)$ which is the set of vertices in one of the connected components of  $\mathscr{V}_{i}\setminus\bigcup_{l=1}^{i-1}\mathscr{D}_{l}$ not contained in $\mathscr{D}_{i}$. 
   	 Then for any $\mathscr{C}\in\mathfrak{X}_{i}$ there is an edge in $\G(\trr)$ between some $\tau\in\mathscr{C}$ and some $\tau'\in\mathscr{D}_{j}$ such that $j<i$ and $\sigma_{i}\cap\sigma_{j}\neq\emptyset$ (in particular, $\mathfrak{X}_{1}=\emptyset$).
   \end{observation}
   \begin{subsubproof}
      By Observation \ref{o: connectedness}, there is a path in the graph spanned by $\mathscr{V}_{i}$ in $\G(\trr)$ connecting $\mathscr{C}$ to $\hat{\tau}_{i}$. By construction, for some $1\leq j<i$ such path must intersect the set $\mathscr{D}_{j}\cap\mathscr{V}_{i}$, since otherwise $\sigma$ would have been included in $\mathscr{D}_{i}$ to begin with. Finally, bound (\ref{t2}) for $\delta$ ensures that $\sigma_{i}\cap\sigma_{j}\neq\emptyset$. 
   \end{subsubproof} 
  
   For any $2\leq i\leq r$ and $\mathscr{C}\in\mathfrak{X}_{i}$ we choose some $1\leq j<i$ satisfying the conclusions of Observation \ref{o: leftover components} and denote it by $\lambda(\mathscr{C})$. The set $\mathscr{E}_{j}:=\mathscr{D}_{j}\cup\bigcup_{\lambda(\mathscr{C})=j}\mathscr{C}$ spans a connected subgraph of $\G(\trr)$. As $\Gamma_{j}$ we take any maximal subtree of said subgraph.
   Clearly in case (\ref{non-compact case}) of the statement of the lemma we have
   $$L\cup K\subseteq \rl{\tr}\subseteq\textstyle\bigcup_{j=1}^{r}\mathscr{E}_{j}\subseteq\textstyle\bigcup_{j=1}^{r}\rl{\Gamma_{j}}=|\mathcal{F}|,$$
   and, similarly, $M=|\mathcal{F}|$ in case (\ref{compact case}).
   
   For any $1\leq i<j\leq r$ and $\mathscr{C}\in\mathfrak{X}_{i}$ with $\lambda(\mathscr{C})=j$ we have $\mathscr{C}\subseteq\mathscr{V}_{i}\subseteq\nn_{\delta}(\sigma_{i})$ and $diam(\sigma_{i})<\epsilon$, from which it follows that $diam(\mathscr{C}\cup\sigma_{i})<\epsilon+2\delta$. We also know that
   $\sigma_{i}\cap\sigma_{j}\neq\emptyset$, by Observation \ref{o: leftover components}.
   Since $\mathscr{E}$ is the union of a set contained in $\mathscr{V}_{j}$ and sets of the form $\mathscr{C}$ with $\lambda(\mathscr{C})=j$, and since $\delta<\epsilon$ we conclude: 
   $$diam(\textstyle\bigcup_{\sigma\in\mathscr{E}_{j}}\sigma)<diam(\sigma_{j})+2\max\{\delta,\epsilon+2\delta\}<7\epsilon.$$ 
\end{subproof}
   
  Let $\Delta=\Sp(\mathcal{F})$, where $\mathcal{F}$ is given by Lemma \ref{l: small retracting forest}.
  Denote by $U_{0}$ the connected component of $M\setminus\rl{\Delta}$ containing $g\cdot D$.  
  \begin{lemma}
  \label{l: out of neighbour}For any neighbourhood $W$ of $\rl{\Delta}$ there exists some $h\in\w_{7\epsilon}$ such that $h\cdot (U_{0}\setminus W)\subseteq g\cdot D$.
  \end{lemma}
  \begin{subproof}
    From Fact \ref{f: spine} we get some isotopy $G$ of $N'$ in itself with $G_{-1}=Id_{N'}$ and
    supported on the complement some neighbourhood $V$ of $\rl{\Tsp(\mathcal{F})}$ with $V\subseteq W$
    and such that $G_{1}(\rl{\mathcal{F}})\subseteq W$. 
    Since $diam(|\Gamma_{i}|)<7\epsilon$ for all $i$, the covering diffeotopy 
    $G$ provided by Lemma \ref{l: isotopy extension} can also be taken to be supported on a disconnected union of open sets of diameter less than $7\epsilon$, from which it follows that $h:=G_{1}^{-1}\in\w_{7\epsilon}\subseteq\V_{0}$. 
    We now have $h\cdot (N'\setminus g\cdot D)=h\cdot\rl{\mathcal{F}}\subseteq W$, and so we are done.
  \end{subproof}
 
  In order to verify $\PS{\hh}{\Delta}\subseteq\V$, pick some arbitrary $h\in\PS{\hh}{\Delta}$. Lemma \ref{l: supported away} provides some $g_{0}\in \w_{\epsilon}\subseteq
  \V_{0}$ with the property that $g_{0}h$ equals $\tilde{G}_{1}$ for some compactly supported diffeotopy $\tilde{G}$ of $Id_{M}$ such that $supp(\tilde{G})\cap W=\emptyset$ for some neighbourhood $W$ of $\Delta$, while Lemma \ref{l: out of neighbour}  yields $g_{1}\in\V_{0}$ such that
  $g_{1}\cdot (U_{0}\setminus W)\subseteq g\cdot D$.
 We can now write $(g_{0}h)^{g_{1}^{-1}}=h_{0}h_{1}$,
 where $h_{0}\in\hh$ is diffeotopic to the identity through a diffeotopy supported on $g\cdot D$ and 
 \begin{itemize}
 	\item in scenario (\ref{compact case}) $h_{1}=Id$ (since here $M=|\mathcal{F}|$), 
 	\item in scenario (\ref{non-compact case}) $h_{1}\in\hh$ is supported outside of $K$.\footnote{In fact, also diffeotopic to the identity through a diffeotopy compactly supported on $M\setminus\mathring{U}_{0}$.}
 \end{itemize}

 This also implies that $h_{0}^{g}$ is diffeotopic to the identity through a diffeotopy supported on $D$. Since
   $diam(D)<\epsilon$, we get $h_{0}^{g}\in\w_{\epsilon}$, and thus $h_{0}\in\w_{\epsilon}^{g^{-1}}\subseteq\V_{0}^{3}$. On the other hand, $h_{1}\in\V_{0}$ since it is either trivial or satisfies $supp(h_{1})\cap K=\emptyset$. This concludes the proof of Lemma \ref{l: nerve}, since we can write
 $$
 h=g_{0}^{-1}h_{0}^{g_{1}}h_{1}^{g_{1}}\in\V_{0}^{9}\subseteq\V.
 $$
  
  \end{proof}
  
  \begin{corollary}
  \label{c: control on vertices}Under the assumptions of Lemma \ref{l: nerve}, for any $\V\in\nd$, any open $U\subseteq M$ 
  and any compact submanifold $L\subseteq M$, there exists some embedded simplicial complex 
  $\Delta$ such that:
  \begin{itemize}
  	\item $\PS{\hh}{\Delta}\subseteq\V$,
  	\item $\Delta\transv L$,
  	\item and $\Delta^{(0)}\subseteq U$.
  \end{itemize}
  \end{corollary}
  \begin{proof}
  On the one hand, for any finite set $\mathcal{F}$ of points, any $m$-ball $D$ in $M$ and any
  $\W\in\nd$ there is, by Lemma \ref{l: nerve}, some finite embedded simplicial complex $\Delta$ such that $\PS{\hh}{\Delta}\subseteq \W$ and $\mathcal{F}\cup B\subseteq\overline{U}_{0}$ for some connected component $U_{0}$ of $M\setminus\rl{\Delta}$. For $\epsilon>0$ small enough $\w_{\epsilon}\subseteq\W$ and there exists some $g_{0}\in\w_{\epsilon}$ such that $g_{0}\cdot\mathcal{F}\subseteq U_{0}$. One can also find
   $g_{1}\in\PS{\kk}{\Delta}\subseteq\W$ such that $g_{1}\cdot(g_{0}\cdot\mathcal{F})\subseteq B$, by Fact \ref{f: transitive on disks}. 
  
  On the other hand, for $L\subseteq M$  as above, any finite embedded simplicial complex $\Delta$ and any finite collection of points $\mathcal{F}$ contained in some open set $U$ there exists some $\delta>0$ such that $g\cdot\mathcal{F}\subseteq U$ for any $g\in\V_{\delta}$. By Lemma \ref{l: wlog transverse} for any $\delta>0$ there is $g\in\w_{\delta}$ such that $L\transv g\cdot\Delta$. The result easily follows from all of the above and Lemma \ref{l: nerve} using the continuity of multiplication in $\T$.

  \end{proof}
  
  \section{Rotating cubes and the density of arc compressions} 
  \label{s: compression moves}
   In this section we will work in the setting \ref{conddiff} of \ref{n: G}, and we will furthermore assume that the group topology $\T$ on $\gg$ is strictly coarser than $\tc$, so that Corollary \ref{c: control on vertices} applies.   
   \begin{definition}
    	\label{d: compressions}Given embedded closed subset $D$ of $M$ and $\eta>0$ we write
    	$\cmp{D}{\eta}$ for the group of elements in $\hh$ diffeotopic to the identity through a diffeotopy compactly supported on $\nn_{\eta}(D)$.
      Let $\T$ be a group topology on $\gg$. We say that $\T$ is compressive if 
    	for any embedded $k$-disk $C$ in $M$, $k\leq m-1$, and any $\V\in\nd$ there is some $\eta>0$ such that $\cmp{\eta}{C}\subseteq\V$.       	
     	If the former is true for $k=0$ we say that arc compressions are dense around the identity in $\T$.   
    \end{definition}
    
    The following easily follows from the continuity of multiplication. 
    \begin{observation}
    \label{o: intrusiveness} If arc compressions are dense around the identity in $\T$ (with the usual assumptions), then $\T$ is intrusive. 
    \end{observation}
         
  Let $Q :=I^{m}$. We consider the CW-structure on $Q$ (which can be refined to a simplicial structure) in which the cells are the sets of the form $\{\underline{x}\,|\,x_{i_{1}}=\epsilon_{1},x_{i_{2}}=\epsilon_{2},\dots x_{i_{r}}=\epsilon_{r}\}$, where $1\leq i_{1}<i_{2}\dots <i_{r}\leq m$ and $\epsilon_{l}\in\{\pm 1\}$. 
  We suspect that the arguments below are well-known, but we have been unable to find a reference.
  \begin{lemma}
  	\label{l: intersection point} Let $f_{i}:Q\to I,\,\,1\leq i\leq m$ be functions such that
  	$$f_{i}(I^{i-1}\times\{\epsilon\}\times I^{m-i})=\{\epsilon\}$$ for $\epsilon\in\{-1,1\}$ and $1\leq i\leq m$. 
  	Then there is $\underline{x}\in Q$ such that $f_{i}(\underline{x})=0$ for all $i$.
  \end{lemma}
  \begin{proof}
  	Assume the conclusion fails for some tuple $(f_{i})_{i=0}^{m}$ as in the hypotheses. Consider the function $f:Q\to Q$ given by $f(q):=(f_{i}(q))_{i=1}^{m}$ and let $E$ be a small open ball around $\underline{0}$
  	such that $im(f)\cap E=\emptyset$. 
    By assumption, $f$ is the identity on $Q^{(0)}$ and preserves each of the faces of $Q$.
  	We construct by induction on $0\leq k\leq m-1$ some homotopy $F^{k}:Q^{(k)} \times I\to Q^{(k)}$ such that:
  	\begin{itemize}
  		\item $F^{k}_{-1}=f_{\restriction Q^{(k)}}$,
  		\item $F^{k}_{1}=Id_{Q^{(k)}}$,
  		\item $F^{k}$ preserves each of the faces of $Q^{(k)}$ setwise at every point in time.
  	\end{itemize}
    
    As $F^{0}$ we can simply take $Id_{Q^{(0)}}\times Id_{I}$. Assume now $F^{k}$ is given for some $0\leq k\leq m-2$. By the homotopy extension property applied to $(E,\partial E)$ for each of the $(k+1)$ dimensional faces $E$ of $Q$, there is some homotopy $\tilde{F}^{k}:Q^{(k+1)}\times I\to Q^{(k+1)}$ such that: 
    \begin{itemize}
    	\item $\tilde{F}^{k}_{-1}=f_{\restriction Q^{(k+1)}}$,
    	\item $\tilde{F}^{k}_{\restriction Q^{(k)}\times I}=F^{k}$,
    	\item and $\tilde{F}^{k}$ preserves each of the $(k+1)$-dimensional faces of $Q$ setwise at every point in time.
    \end{itemize}
    
    A homotopy $G^{k+1}:Q^{(k+1)}\times I\to Q^{(k+1)}$ from $\tilde{F}^{k}_{1}$ to $Id_{Q^{(k+1)}}$ is easily seen to exist (for instance, using convex combinations) and we can take as $F^{k+1}$ the concatenation of $\tilde{F}^{k}$ and $G^{k+1}$.
   
   Let $E$ be a small collared ball around $\underline{0}$ such that $f(Q)\cap E=\emptyset$.
   A final application to the homotopy extension property to the pair $(Q\setminus E,\partial Q)$ the starting map $f$ and the homotopy $F^{m-1}$, yields $f':Q\to Q\setminus E$ 
   such that $f'_{\restriction Q^{(m-1)}}=Id_{Q^{(m-1)}}$. This is impossible, since for any retraction $r:Q\setminus E\twoheadrightarrow\partial Q= Q^{(m-1)}$ the map $r\circ f'$ would then witness of the triviality of the $(m-1)$-th homology group of $S^{m-1}$. A contradiction.       	
  \end{proof}

    \newcommand{\F}[0]{\mathcal{F}}
	  One can endow $Q$ with a smooth structure that makes it diffeomorphic to $\overline{B}^{m}$, in such a way that any homeomorphism of $Q$ given by coordinate permutation is also a diffeomorphism. By a smoothed cube we mean an embedding $\phi:Q\to M$ which is a diffeomorphism with respect to this structure. We will refer to the images of the faces of $Q$ as the faces of $\phi$ or $\phi(D)$. In particular, we will write $\F_{i}^{\pm 1}(\phi):=\phi(I^{i-1}\times\{\pm 1\}\times I^{m-i})$. 

  \begin{lemma}
   \label{l: forcing vertices} Let $\phi:Q\cong D\subseteq M$ be a smoothed $m$-cube in $M$. 
  Suppose that $\Delta$ is a finite embedded $(m-1)$-dimensional simplicial complex
  in general position with respect to $\partial D$.
  Then at least one of the following holds:
  \begin{enumerate}[(a)]
  \item \label{option1}$D\cap\Delta^{(0)}\neq\emptyset$,
  \item \label{option2}there exists some $1\leq i\leq m$ and some component $W$ of  $D\setminus\rl{\Delta}$ with
  $$\F_{i}^{1}(\phi)\cap \overline{W}\neq\emptyset, \quad \F_{i}^{-1}(\phi)\cap \overline{W}\neq\emptyset.$$.
  \end{enumerate}	
  \end{lemma}
  \begin{proof}	
  	\newcommand{\ow}[0]{\overline{W}}
  	Write $\F_{i}^{\nu}:=\F_{i}^{\nu}(\phi)$.	Assume for the sake of contradiction neither (\ref{option1}) nor (\ref{option2}) holds. Let 
  	$U_{i}$ be the interior in $D$ of the union of $\overline{W}$ of all the connected components $W$ of 
  	$D\setminus\rl{\Delta}$ such that $\F_{i}^{-1}\cap\overline{W}\neq\emptyset$. 
  	On the one hand, $\F^{-1}_{i}\subseteq U_{i}$, since $D\setminus\rl{\Delta}$ can have only finitely many connected components (this is true locally at every point in $D$, where for the points in $\partial D$ we use the transversality assumption). On the other hand, the negation of (\ref{option2}) implies that $\overline{U}_{i}\cap \F^{1}_{i}=\emptyset$ for all $i$. 
  	Clearly, the frontier $B_{i}$ of $U_{i}$ in $D$ is contained in $\rl{\Delta}$.  
  	
  	By applying the strong version of Urysohn's lemma (\cite{munkrestopology}, Ex 33.5) 
  	to $\overline{U_{i}}$ and $\overline{D\setminus U^{i}}$, one obtains continuous functions $f_{i}:D\to I$ such that $f_{i}(\F_{i}^{\lambda})=\{\lambda\}$ for $\lambda\in\{1,-1\}$, and  
  	\begin{equation}
  		\label{eq: intersection}\tag{$\blackdiamond$} B_{i}=f_{i}^{-1}(0)=\overline{f_{i}^{-1}([-1,0))}\cap \overline{f_{i}^{-1}((0,1])}.
  	\end{equation}
  	Lemma \ref{l: intersection point} applied to the map $(f_{i}\circ\phi)_{i=1}^{m}:Q\to Q$ implies $\bigcap_{i=1}^{m}B_{i}\neq\emptyset$.
   \newcommand{\C}[0]{\mathcal{C}}
  	We also have the following.
  	\begin{observation}
  		\label{l: nice boundary} For every simplex $\sigma\in\Delta$ and every connected component $\mathcal{C}$ of $\sigma\cap \mathring{D}$
  		either $\mathcal{C}\cap B_{i}=\emptyset$ or $\mathcal{C}\subseteq B_{i}$, and thus $\overline{\mathcal{C}}\subseteq B_{i}$. 
  	\end{observation}  	
    \begin{subproof}
    
    	Let $\C_{i}^{-1}:=\mathcal{C}\cap \overline{f_{i}^{-1}([-1,0))}$ and $\C_{i}^{1}:=\mathcal{C}\cap \overline{f_{i}^{-1}((0,1])}$. Clearly, $\C_{i}^{\pm 1}$ is closed in $\C$. 
    	We also claim that it is open. Indeed, since $B_{i}$ is contained in a codimension $1$ embedded simplicial complex, any point $p\in\mathring{D}\cap\C$ admits a neighbourhood $V$ in $\mathring{D}$ such that $V\setminus B_{i}$ consists of finitely many components $\mathcal{E}_{1},\dots \mathcal{E}_{k}$ with the property that $\C\cap V\subseteq\mathcal{E}_{l}$ for any $1\leq l\leq k$. By construction, the sign of $f_{i}$ is constant on each $\mathcal{E}_{l}$. If $p\in\C_{i}^{\epsilon}$, then this sign must be $\epsilon$ for at least one value of $l$, from which it follows that $\mathcal{C}\cap V\subseteq \C_{i}^{\epsilon}$. We conclude that for $\epsilon\in\{\pm 1\}$ either $\C_{i}^{\epsilon}=\C$ or $\C_{i}^{\epsilon}=\emptyset$. If we assume  $\C\cap B_{i}\neq\emptyset$, then by (\ref{eq: intersection}) we have $\C_{i}^{1},\C_{i}^{-1}\neq\emptyset$, from which it follows that $\C\subseteq \overline{f_{i}^{-1}([-1,0))}\cap \overline{f_{i}^{-1}((0,1])}=B_{i}$, as needed.   
    \end{subproof}
    
  	It follows from the negation of (\ref{option1}), the fact that $\bigcap_{i=1}^{m}B_{i}\neq\emptyset$ and the observation above that there must be some positive-dimensional simplex $\sigma$ of $\Delta$ and some connected component $\mathcal{C}$ of 
  	$\mathring{D}\cap\sigma$ such that $\mathcal{C}\subseteq\bigcap_{i=1}^{m}B_{i}$. 
  	Then necessarily $\overline{\mathcal{C}}\cap\partial D\neq\emptyset$ (a path in $\sigma$ from $\mathcal{C}$ to a vertex of $\sigma$ eventually hits $\partial D$). Therefore, 
  	$\overline{\mathcal{C}}\cap(\F^{1}_{i}\cup \F^{-1}_{i})\neq\emptyset$ for some $1\leq i\leq m$, which is impossible, since $\overline{\mathcal{C}}\subseteq B_{i}\subseteq f_{i}^{-1}(0)$. 
  \end{proof}
  
  Given a smooth arc $\alpha$ and an open set $U$ containing $\alpha$, let $\mathcal{A}_{U}(\alpha)$ be the collection of arcs $\beta$ in $U$ such that $\alpha(s)=\beta(s)$ for $s\in\{-1,1\}$.
  
  \begin{lemma}
  \label{l: many point pushing maps}Let $\alpha$ be a smooth arc and $U$ an open set containing $\alpha$. Then for any $\V\in\nd$ there exists some smooth arc $\beta\in\mathcal{A}_{U}(\alpha)$ and some neighbourhood $V$ of $\beta$ such that $\hh[V]\subseteq\V$. 
  \end{lemma}
  \begin{proof}
  We first claim that for every smoothed $m$-cube $\phi:Q\to D$ and every $\V_{0}\in\nd$ 
  there exists some finite embedded simplicial complex $\Delta\subseteq M$ such that:
  \begin{itemize}
  	\item  $\partial D\transv\Delta$,
  	\item  $\PS{\hh}{\Delta}\subseteq\V_{0}$,
  	\item  and the closure of some connected component of $D\setminus\rl{\Delta}$ intersects both $\F^{1}_{1}(\phi)$ and $\F^{-1}_{1}(\phi)$. 
  \end{itemize}
  
  Write $\F_{i}^{\nu}:=\F_{i}^{\nu}(\phi)$. If the conclusion does not hold, then there exists some $\V_{0}\in\nd$ with the property that every finite embedded simplicial complex $\Delta$ satisfying the first two conditions with respect to $\V_{0}$ must fail to satisfy the third.
  
  By Fact \ref{f: transitive on disks}, for $i\geq 2$ there exists $g_{i}\in\G$ which preserves $D$ setwise and maps $\F_{i}^{\nu}$ onto $\F^{\nu}_{1}$ for $\nu\in\{-1,1\}$. If we let $\V^{*}:=\V_{0}\cap\bigcap_{i=2}^{m}\V_{0}^{g_{i}}$, then for every finite embedded simplicial complex $\Delta$ satisfying the first two conditions there cannot be any $1\leq i\leq m$ for which $\F_{i}^{1}$ and $\F^{-1}_{i}$ both intersect the closure of the same component of $D\setminus\rl{\Delta}$. By Lemma \ref{l: forcing vertices}, we must have $\Delta^{0}\cap D\neq\emptyset$, contradicting Corollary \ref{c: control on vertices}.   
  
  To conclude, let $\alpha,U$ and $\V$ be as in the hypotheses of the lemma. Pick $\V_{0}=\V_{0}^{-1}\in\nd$ with $\V_{0}^{3}\subseteq\V$ and let $\epsilon>0$ be such that $\V_{\epsilon}\subseteq\V_{0}$. One can find some smoothed $m$-cube with image 
  $D\subseteq U$ and embedded balls $C_{-1},C_{1}$ of diameter less than $\epsilon$  
  such that:
  \begin{itemize}
  	\item $\alpha\subseteq D\subseteq U$,
  	\item $\alpha(\nu)\in \F_{1}^{\nu}$ for $\nu\in\{-1,1\}$, 
  	\item $\F^{\nu}_{1}\subseteq C_{\nu}\subseteq U$ for $\nu\in\{-1,1\}$,
  \end{itemize} 
  The discussion at the beginning of the proof implies the existence of an arc $\beta: I\to D$ with $\beta(\epsilon)\in C_{\epsilon}$ for $\epsilon\in\{\pm 1\}$ and some neighbourhood $V$ of $\beta$ such that $\hh[V]\subseteq\V_{0}$. Conjugating $\hh[V]$ by $h^{-1}$ for a suitable element $h\in\V_{0}$ supported on $U$ and satisfying $h\beta(\epsilon)=\alpha(\epsilon)$ for $\epsilon\in\{\pm 1\}$ yields the desired result.
  \end{proof}

  \begin{corollary}
  	\label{c: 0-slices}If $m\neq 3$, then arc compressions are dense in any Hausdorff group topology $\T$ on $\gg$ with $\T\subsetneq\tc$.
  \end{corollary}
  \begin{proof}
  	 Let $\V\in\nd$ and choose $\V_{0}=\V_{0}^{-1}\in\nd$ such that $\V_{0}^{7}\subseteq\V$. 
  	 Let $\alpha$ be an arc in $M$. Lemma \ref{c: control on vertices} implies that there exists some finite embedded $(m-1)$-dimensional simplicial complex
  	 $\Delta$ with $\alpha\transv \Delta$, does not contain its endpoints and satisfies  
  	 $\PS{\hh}{\Delta}\subseteq\V_{0}$. 
  	 The intersection of $\alpha$ with $\Delta$ consists of finitely many points $p_{1},\dots p_{r}$, traversed in that order by $\alpha$, where $p(t_{i})=p_{i}$. Extend this notation by letting $t_{0}=-1$ and $t_{r+1}=1$. 	   	 
  	 Condition $\alpha\transv \Delta$ and the inverse function theorem implies the existence of a neighbourhood $U$ of $\alpha$ and a diffeomorphism $\phi: U\cong (0,1)\times B^{m-1}$ exists mapping $\alpha$ onto the set $(0,1)\times\{\underline{0}\}$ and $U\cap|\Delta^{(m-1)}|=\emptyset$ to the union of the sets $\{t_{i}\}\times B^{m-1}$. 
 	  
  	 We claim that there is an arc $\beta\in\mathcal{A}_{U}(\alpha)$ intersecting $\Delta$ only at the points 
  	 $p_{1},\dots p_{r}$, in that order, and some element $h\in\V_{0}$ mapping an initial subarc of 
  	 $\beta$ in $U\setminus\rl{\Delta}$ onto $\beta$. 	   
  	 Take $\V_{1}=\V_{1}^{-1}\in\nd$ such that $\V_{1}^{2r}\subseteq\V_{0}$ and let $\delta>0$ be such that $\V_{\delta}\subseteq\V_{1}$ and $2\delta<\min_{0\leq i\leq r-1}d(p_{i},p_{i+1})$. For each $1\leq i\leq r$ let $q_{i}=\alpha(s_{i}),q'_{i}=\alpha(s'_{i})$, where 
  	 $s_{i}\in(t_{i-1},t_{i})$ and $s'_{i}\in(t_{i},t_{i+1})$ and $q_{i},p_{i},q'_{i}$ are in the interior of an $m$-ball of diameter less than $\delta$. 
  	 By Lemma \ref{l: many point pushing maps} there are arcs $\gamma_{i}$, $0\leq i\leq r$ in $U\setminus\rl{\Delta}$, where
  	 \begin{itemize}
  	 	\item $\gamma_{0}$ joins $p_{0}$ to $q_{1}$,
  	 	\item $\gamma_{i}$ joins $q'_{i}$ to $q_{i+1}$ for $1\leq i\leq r$,
  	 	\item and $\gamma_{r}$ joins $q'_{r}$ to $p_{r+1}$,
  	 \end{itemize}
  	 as well as neighbourhoods $W_{i}$ of $\gamma_{i}$ such that
  	 $\kk[W_{i}]\subseteq\V_{1}$.  
  	 Using Fact \ref{f: transitive on disks} one easily finds $h_{0}\in \kk[W_{0}]$ mapping an initial subarc of $\gamma_{0}$ onto $\gamma_{0}$, then $h'_{0}\in\V_{\delta}$ mapping $\gamma_{0}$ onto 
  	 some arc from $p_{0}$ to $q_{1}$ of the form $\gamma_{0}*\delta_{0}$, where $\delta_{0}$ is a sub-arc of $\alpha$ with $p_{0}$ in its interior. Similarly, we can find $h_{1}\in\kk[W_{1}]$ such that the image of $h'_{0}h_{0}\cdot\gamma_{0}$ is of the form $\gamma_{0}*\delta_{0}*\gamma_{1}$ and so forth.
  	 If we let 
  	 $$\beta:=\gamma_{0}*\delta_{0}*\gamma_{1}*\delta_{1}\cdots\gamma_{r+1},\quad\quad h:=h_{r+1}h'_{r}h_{r}\cdots h'_{0}h_{0}\in\V_{1}^{2r}\subseteq\V_{0},$$ then $h$ will map an initial subsegment of $\beta$ onto $\beta$.
  	 
  	 To conclude, we use the assumption $dim(M)\neq 3$ to find some $g\in\PS{\hh}{\Delta}$ fixing the endpoints of $\alpha$ such that $\beta=g\cdot\alpha$. Then the element $h^{g}$ maps some initial subarc of $\alpha$ disjoint from 
  	 $\Delta$ onto $\alpha$. We have
  	 $$
  	 (h^{g}\cdot\Delta)\cap\alpha=(g^{-1}h\cdot\Delta)\cap\alpha=g^{-1}h\cdot(\Delta\cap(g\cdot \alpha))=g^{-1}h\cdot(\Delta\cap h^{-1}\cdot\beta)=\emptyset.
  	 $$
  	 And we are done, since for some neighbourhood $W$ of $\alpha$ we have 
  	 $$\kk[W]\subseteq\PS{\hh}{\Delta}^{(h^{g})^{-1}}\subseteq\V_{0}^{7}\subseteq\V.$$
  	 
  \end{proof}

\section{Completing the proof of Theorem \ref{t: main}}
  \label{conclusion}  
        In this section we keep working with $M,\gg,\hh,\T$ as in case \ref{conddiff} of Notation \ref{n: G}.
        We start the section with a lemma that allows us to ignore the lower-dimensional skeleton of embedded simplicial complexes.
        \begin{lemma}
        	\label{l: cleaning}Let $C\subseteq M$ be an $m$-ball and $\Delta\subseteq M$ a finite embedded simplicial complex of dimension $(m-1)$. If $\T$ is compressive and small deformations are close to the identity in $\T$, then for any $\mathcal{V}\in\nd$ there exists some $g\in\mathcal{V}$ such that $g\cdot\rl{\Delta^{(m-2)}}\cap C=\emptyset$. 
        \end{lemma}
	        \begin{proof}
	          Pick some $m$-ball $D\subseteq \mathring{C}$ such that 
	        	$D\cap\Delta=\emptyset$ and some $m$-ball $C'$ such that $D\subseteq C\subseteq U$, where $U :=\mathring{C}'$.	Then there exists    
	        	$$
	        	\phi:U\setminus D\cong S^{m-1}\times(0,\infty)  
	        	$$
	          mapping $C\setminus D$ onto $S^{m-1}\times(0,1]$.   
	          We may assume every simplex in $\Delta^{(m-2)}$ is the face of some $(m-2)$-simplex and, after subdividing $\Delta$, also that any simplex in 
	        	$\Delta^{m-2}$ is either contained in $U$ or in $M\setminus C$. Write 
	        	$$\Theta:=\{\,\sigma\in\Delta^{m-2}\,|\,\sigma\subseteq U\,\} $$

	         	We will use the term $U$-lift to refer to the 
	         	extension by the identity on $D\cup (M\setminus U)$ of some element $h\in\D_{c0}(U\setminus D)$ with the property that for all $p\in S^{m-1}\times(0,\infty)$ the push-forward $\hat{h}$ of $h$ by $\phi$ to $S^{m-1}\times(0,\infty)$ satisfies
	          $$\pi_{(0,\infty)}(\hat{h}(p))\geq\pi_{(0,\infty)}(p), \quad \pi_{S^{m-1}}(\hat{h}(p))=\pi_{S_{m-1}}(p)$$
	          for all $p\in S^{m-1}\times(0,\infty)$. 
	          Notice that any $U$-lift $h$ satisfies $C\cap h(M\setminus C)=\emptyset$ and $D\cap h(M\setminus D)=\emptyset$. 
	          
	          Let $\{\sigma_{k}\}_{k=1}^{r}$ be an enumeration of $\Theta$ and pick $\V_{0}\in\nd$ such that $\V_{0}^{2r}\subseteq\V$. For every $1\leq k\leq r$ let $\iota_{k}:B^{m-2}\to U$ be some smooth embedding such that $\sigma_{k}=\iota_{k}(\tau_{k})$ for some affine simplex $\tau_{k}\subseteq B^{m-2}$.

	          For $1\leq k\leq r$ we inductively choose some $g'_{k}\in\V_{0}\cap \V_{0}$ supported on $U\setminus D$ and then a $U$-lift $g_{k}\in\V_{0}$ in such a way that if we write
	          $$\overline{g}_{0}:=1,\quad\overline{g}_{k}:=g_{k}g'_{k}\cdots g_{1}g'_{1},\quad \iota_{k,l}:= g'_{k}\overline{g}_{k-1}\circ\iota_{l}:B^{m-2}\to M,$$
	          then  $\overline{im(g_{k}\circ\iota_{l})}\cap C=\emptyset$ for all $1\leq l\leq k\leq m$.   
	          Clearly, the element $g :=\bar{g}_{r}\in\V_{0}^{2r}\subseteq\V$ will satisfy $C\cap g\cdot\rl{\Delta^{(m-2)}}=\emptyset$. 
	          Suppose $g'_{1},g_{1}\dots g'_{k-1},g_{k-1}$ have already been chosen. 
	          
	          Assume that $g'_{l},g_{l}$ have already been constructed for $1\leq l\leq k-1$.  	                   
	          Let $\epsilon$ be such that $\w_{\epsilon}\subseteq\V_{0}$. 
	          Corollary \ref{c: projecting simplicial complexes} allows us to find for any 
	          $\eta\in(0,\epsilon)$ some  
	          $g'_{k}\in\w_{\eta}$ supported on $U\setminus D$ such that there is a filtration $$L_{0}=B^{m-2}\supsetneq L_{1}\supsetneq\dots L_{n}\supsetneq L_{n+1}=\emptyset,$$
	           where $L_{l+1}$ is a closed submanifold of $L_{l}$ and $\pi_{S^{m-2}}\circ\phi\circ\iota_{k,k}$ restricts to an immersion on $L_{l}\setminus L_{l+1}$. Clearly, 
	           by choosing $\eta$ small enough we might also assume that 
	           $\overline{im(\iota_{k,l})}\cap C=\emptyset$ for all $1\leq l<k$. 
	           It remains to find a $U$-lift $g_{k}$ such that $\overline{im(g_{k}\circ\iota_{k,k})}\cap C=\emptyset$. 
	           
	          We start by noting the following easy consequence of the compressiveness of $\T$.
	          \begin{observation}
	          	\label{o: narrow lifts} Assume that $F\subseteq U\setminus D$ is a compact set for which $$(\pi_{S^{m-1}}\circ\phi)(F)\subseteq S^{m-1}$$ is contained in an embedded $d$-disk in $S^{m-1}$, $d<m-1$.
	          	Then for every $\U\in\nd$ there exists some $U$-lift $g$ such that $g\cdot F\cap C=\emptyset$.
	          \end{observation} 

	          In order to construct $g_{k}$, choose in advance some sequence $(\V_{l})_{l\geq 1}$, $\V_{l}=\V_{l}^{-1}\in\nd$, with the property that $\V_{l}^{2}\subseteq\V_{l-1}$ for all $l\geq 1$. 
	          Additionally, for $l\geq 1$ and $j\geq 2$ choose $\V_{l,j}=\V_{l,j}^{-1}\in\nd$ such that 
	          $(\V_{l,j})^{j}\subseteq\V_{l}$. 
	          
	          We will construct by induction on $l$ a sequence of $U$-lifts $h_{l}\in\V_{n-l+1}$ supported on $U\setminus D$ with the property that if we let $\tilde{h}_{l}:=h_{l}\cdots h_{0}$, then 
	          $\overline{(\tilde{h}_{l}\circ\iota_{k}')(V_{l})}\cap C=\emptyset$ for some neighbourhood $V_{l}$ of $\tau_{k}\cap L_{n-l}$. At the end we may set $g_{k}:=h_{n}\cdots h_{0}\in\V_{0}$ (obviously, a composition of $U$-lifts is again a lift).           
	          Write $V_{-1}=\emptyset$ and assume that we have already constructed $h_{l'}$ for $0\leq l'<l$. Since $(L_{n-l}\cap\tau_{k})\setminus V_{l-1}$ is compact, it can be covered by a collection $\{D_{l,s}\}_{1\leq s\leq r_{l}}$ of finitely many closed $dim(L_{n-l})$-balls in $L_{n-l}$.\footnote{As a matter of fact $dim(L_{l})=m-2(l+1)$, see \cite{arnold2014singularities}.}
	          We choose lifts $h_{l,s}$ successively for $1\leq s\leq r_{l}$.
	          Given $s$, the map 
	          $$\pi_{S^{m-1}}\circ\phi\circ h_{l,s-1}\cdots h_{l,0}\tilde{h}_{l-1}\circ\iota_{\restriction D_{l,s}}=\pi_{S^{m-1}}\circ\phi\circ\iota_{\restriction D_{l,s}}:D_{l.s}\to S^{m-1}$$
	          is a smooth embedding, where the equality holds because all the $h_{l,t}$ in the expression are $U$-lifts. Observation \ref{o: narrow lifts} provides some $U$-lift $h_{l,s}\in\V_{m-l+1,r_{l}}$ supported on $U^{*}$ such that  
	          $$(h_{l,s}\cdots h_{l,0}\tilde{h}_{l-1}\circ\iota_{k})(D_{l,s})\cap C=\emptyset.$$ 
	          If we write $h_{l}:=h_{l,r_{l}}\cdots h_{l,0}$, then we have
	          $\tilde{h}_{l}\cdot (L_{n-l}\cap\tau_{k})\cap C$ and the existence of $V_{l}$ follows by continuity.
	                    
	        \end{proof}
    
    \begin{lemma}
      \label{l: different stabilizers}For any finite $(m-1)$-dimensional simplicial complex $\Delta$ in $M$ and any $\epsilon>0$ we have that 
      $\SS{\hh}{\Delta}\subseteq\w_{\epsilon}\PS{\hh}{\Delta}\w_{\epsilon}\PS{\hh}{\Delta}$. Therefore, $\PS{\hh}{\Delta}$ can be replaced by $\SS{\hh}{\Delta}$ in the conclusion of in Corollary \ref{c: control on vertices} (by the continuity of multiplication). 
    \end{lemma}
    \begin{proof}
		    Let $\{\sigma_{i}\}_{i=1}^{r}$ be an enumeration of $\Delta^{m-1}$. 
		    For $1\leq l\leq 2$ choose collections of disjoint embedded $(m-1)$-balls 
		    $\{D_{i}^{l}\}_{i=1}^{r}$, so that if we write $N_{l}:=\bigcup_{i=1}^{r}\mathring{D}_{i}^{l}$, then
		    \begin{itemize}
		    	\item  $D^{1}_{i}\subseteq \mathring{D}^{2}_{i}\subseteq D^{2}_{i}\subseteq\sigma_{i}\setminus\partial\sigma_{i}$ for $1\leq i\leq r$,
		    	\item $\rl{\Delta}\subseteq N_{1}\cup U$ (and thus $N_{2}\setminus N_{1}\subseteq U$). 
		    \end{itemize}
	    	    
		    Take some (trivial) relatively compact tubular neighbourhood   
		    $$\xi:E:=N_{2}\times (-1,4)\to M,$$
		    where $\pi:E \to N_{2}$ is the obvious projection, such that
		    \begin{enumerate}
		    	\item \label{clean preimage}$\xi(\pi^{-1}(N_{2}))\cap|\Delta|=\xi(N_{2}\times\{0\})$,
		    	\item \label{fibers}for any $p\in N_{2}$ we have $diam(\xi(\{p\}\times(-1,4)))<\epsilon$.
		    \end{enumerate}  
	 
	      Let $f:N_{2}\to[0,2]$ some smooth map such that $f_{\restriction N_{1}}=2$, $f_{\restriction V}=0$ in the complement of some compact subset of $N_{2}$. It follows from (\ref{fibers}) that one can find $h_{1}$ diffeotopic to the identity by a fiber preserving diffeotopy compactly supported on $im(\xi)$ and such that $\phi\circ h_{1}\circ\phi^{-1}_{\restriction N_{2}\times\{0\}}$ maps $(\underline{x},0)$ to $(\underline{x},f(\underline{x}))$. Note that, in particular, $h_{1}\in\w_{\epsilon}$. 
        
        Fix also some smooth function $\rho:(-1,4)\to [0,1]$ with $\rho(2)=1$ and $\rho(s)=0$ outside of $(1,3)$ and consider the map $\hat{H}: E\times I\to I$ defined by 
	      \begin{align*}
		     \hat{H}((p,s),t):= (G(p,(t+1)\rho(s)-1),s).
	      \end{align*}
	      One can check that $\hat{H}$ is a compactly supported diffeotopy of $Id_{E}$. 
	      We denote by $H$ the compactly supported diffeotopy of $Id_{M}$ obtained by extending  $(\xi\circ\hat{H}_{t}\circ\xi^{-1})_{t\in I}$ by the identity.
	      For any $p\in N_{2}$ and $t\in I$ we have
	       \begin{equation}
		       \label{eq fix} \tag{\protect\pluma}H_{t}(p)=\xi(\hat{H}_{t}(p,0))=\xi(G(p,-(t+1)0-1),0)=\xi(G(p,-1),0)=\xi(p,0)=p.
	       \end{equation}
	       For $p\in N_{1}$ and $t\in I$ we also have: 
	       \begin{align}
	       \label{eq conj} \tag{\protect\plumaa}
	       \begin{split}
	       h_{1}^{-1}H_{t}h_{1}(p)=&h_{1}^{-1}\circ\xi(\hat{H}_{t}(p,2))=h_{1}^{-1}\circ\xi(G(p,(t+1)\rho(2)-1),2)=\\&h_{1}^{-1}\circ\xi(G_{t}(p),2)=\xi(G_{t}(p),0)=G_{t}(p)
	       \end{split}
	       \end{align}
	       Likewise, since $supp(G)\subseteq N_{1}$ and $h_{1}$ preserve fibers, we conclude that 
	       \begin{equation}
	       	\label{eq fixfix}  \tag{\protect\plumaaa} supp(H),supp(h_{1}^{-1}H_{t}h_{1})_{t\in I}\subseteq\xi(\pi^{-1}(N_{1})),
	       \end{equation}    
	       One deduces from (\ref{clean preimage}), (\ref{eq fix}) and (\ref{eq fixfix}) that $H$ is
	       supported in the complement of the union of $\rl{\Delta}$ and some neighbourhood of $\rl{\Delta^{(m-2)}}$. It follows that $h_{2}:=H_{1}$ belongs to $\PS{\hh}{\Delta}$. 
	       Consider now the map $G^{*}:=M\times I\to M$ given by $G_{t}^{*}:=G_{t}^{-1}h_{1}^{-1}H_{t}h_{1}$. This is a  compactly supported diffeotopy from $Id_{M}$ to $g^{-1}h_{1}^{-1}h_{2}h_{1}$.
	        We also have:
	       \begin{itemize}
	       	\item $supp(G^{*})$ is disjoint from $V=U\cap\rl{\Delta}$,
	       	\item and from (\ref{eq conj}) we deduce that $(G_{t})_{\restriction N_{1}}=Id_{N_{1}}$ for all $t\in I$, so that it is the identity on $\rl{\Delta}$, by the previous bullet point. 
	       \end{itemize}
	       Therefore, $G^{*}$ witnesses $g^{-1}h_{2}^{h_{1}}\in\PS{\hh}{\Delta}$, 
	       which implies that
	       $$g\in h_{2}^{h_{1}}\PS{\hh}{\Delta}\subseteq(\PS{\hh}{\Delta})^{h_{1}}\PS{\hh}{\Delta}\subseteq\w_{\epsilon}\PS{\hh}{\Delta}\w_{\epsilon}\PS{\hh}{\Delta},$$
	       settling the first assertion of the lemma. The second one follows by the continuity of multiplication.
	    \end{proof}

  \begin{proposition}
  	\label{p: slice induction}Assume that the group topology $\T$ on $\gg$ satisfies:
  	\begin{enumerate}
  		\item arc compressions are dense around the identity in $\T$ (Definition \ref{d: compressions}),
  		\item one of the following holds: 
  		\begin{itemize}
  			\item $\T\subseteq\tc$,
  			\item or $M$ is compact and small deformations are close to the identity in $\T$ (Definition \ref{d: small deformations are dense}).
  		\end{itemize}
  	\end{enumerate}
  	 Then for any $m$-ball $D$ in $M$ and any $g\in\D_{c0}(\mathring{D})\subseteq\D_{c0}(M)$ (where the identification is by extension by the identity outside of $\mathring{D}$) we have $g\in\V$ for all $\V\in\nd$. 
  \end{proposition}
  \begin{proof}
  	 Note that Corollary \ref{c: control on vertices} applies (by Observation \ref{o: intrusiveness}). We will use its conclusion several times. 
  	 We prove the result by induction on $m$, starting from the base case $m=1$, for which the desired conclusion is just a reformulation of the assumption that arc compressions are dense. 
     Suppose now the result is true for any $1\leq k<m$. We start with the following sublemma.    
     \begin{lemma}
	     \label{l: inductive slices}Under the assumptions of \ref{p: slice induction} compressions are dense around the identity in $\T$.
     \end{lemma}
   \begin{subproof}
	     Choose any embedded $(m-1)$-ball $D$ in $M$ and $\V\in\nd$ and let $\W=\W^{-1}\in\nd$ satisfy $\W^{3}\subseteq\V$. Using Whitney's extension theorem and the existence of tubular neighbourhoods one can easily construct an embedding $\phi:S^{m-1}\to M$ extending $D$. By the existence of tubular neighbourhoods and the simple-connectedness of $S^{m-1}$, there is a smooth embedding $\xi:S^{m-1} \times I\to M$ such that $\xi_{\restriction S^{m-1}\times\{0\}}=(\phi\circ\pi_{S^{m-1}})_{\restriction S^{m-1}\times\{0\}}$.   
	     Consider the collection $\mathcal{N}^{*}:=\{\V^{*}\}_{\V\in\nd}$ of subsets of $\D_{c0}(S^{m-1})$, where for $\V\in\nd$ we define 
	     $$
	      \V^{*}:=\{\,g\in\D_{c0}(S^{m-1})\,\,|\,\,\exists\hat{g}\in\V\cap\D_{c0}^{\{\phi(S^{m-1})\}}(M)\,\,\,\,\phi^{-1}\circ\hat{g}\circ\phi=g\,\}.
	     $$
	     Let $\T^{*}$ be the unique topology on $\D_{c0}(S^{m-1})$ in which a system of neighbourhoods of $h$ 
	     	 is given by the translate $h\mathcal{N}^{*}$. 
	     \begin{lemma}
	     	 \label{l: aux top}$\T^{*}$ is a group topology in which small deformations are close to the identity and in which arc compressions are dense around the identity.
	     \end{lemma}
	     \begin{subsubproof}
	     	  For all $\V_{0},\V_{1}\in\nd$ we have $(\V_{0}\cap\V_{1})^{*}\subseteq\V_{0}^{*}\cap\V_{1}^{*}$, which implies that $\mathcal{N}^{*}$ generates a filter. Similarly, $(\V^{-1})^{*}=(\V^{*})^{-1}$ for all $\V$, so continuity of the inverse map at the identity in $\T^{*}$ follows from the same property for $\T$.
	     	 It is also clear that $\V_{0}^{*}\V_{1}^{*}\subseteq(\V_{0}\V_{1})^{*}$, so that, likewise, continuity of multiplication at the identity follows from the same property for $\T$. 	     	 
	     	 Finally, notice that $\mathcal{N}^{*}$ is invariant under conjugation by elements from $\D_{c0}(S^{m-1})$. Indeed, any element $g\in\D_{c0}(S^{m-1})$ extends via $\phi$ to an element $\tilde{g}\in\D_{c0}(M)$ (Lemma \ref{l: isotopy extension}), and given $\V\in\nd$, it is easy to see that $(\V^{*})^{g}=(\V^{\tilde{g}})^{*}\in\nd$, 
	     	 from which invariance of $\mathcal{N}^{*}$ under conjugation follows. 
	     	 
	     	 The above discussion shows that $\T^{*}$ is a group topology. That small deformations are close to the identity in $\T^{*}$ follows from the assumptions $\T$ satisfies this property, together with Observation \ref{o: lifting small diffeotopies}, the density of arc compressions around the identity in $\T^{*}$ follows from the fact that $\T$ satisfies the property, together with Lemma \ref{l: isotopy extension}.
	     \end{subsubproof}

	     Given $\W=\W^{-1}\in\nd$ with $\W^{3}\subseteq\V$, Corollary \ref{c: control on vertices} implies that we can find some finite embedded simplicial complex $\Delta$ in $M$ such that $\PS{\hh}{\Delta}\subseteq\W$ and $\phi(S^{m-1})\transv \Delta$. 
	     The last condition implies $\rl{\Delta}\cap S^{m-1}$ avoids a small $(m-1)$-ball $E$ in $S^{m-1}$, so for any embedded $(m-1)$-ball $E'$ in $S^{m-1}$ with $\phi^{-1}(E)\subseteq\mathring{E}'$ there is $\hat{h}\in\D_{c0}(E')\subseteq\D_{c0}(S^{m-1})$, such that $\hat{h}\cdot\phi^{-1}(\Delta)\cap\phi^{-1}(E)=\emptyset$. 
	     The inductive assumption that Proposition \ref{p: slice induction} holds in all dimensions less than $m$, together with Lemma \ref{l: aux top}, implies that $\hat{h}\in\W^{*}$, which yields an extension $h\in\W$ 
	     for which  $h\cdot\rl{\Delta}\cap E=\emptyset$. This then implies the existence of $\eta>0$ such that:
	     $$
	      \cmp{\eta}{E}\subseteq\SS{\hh}{h\cdot\Delta}=\SS{\hh}{\Delta}^{h^{-1}}\subseteq\W^{3}\subseteq\V.
	     $$
   \end{subproof}

     We will derive the conclusion of the proposition from Lemma \ref{l: aux top}.
     Fix $m$-balls $D,D'$ in $M$, $D\subseteq\mathring{D'}$ in $M$ and $g\in \D_{c0}(\mathring{D})\subseteq\D_{c0}(M)$ (extending by the identity outside of $D$).
          
     Let $\V\in\nd$ and choose $\V_{0}=\V_{0}^{-1}\in\nd$ with 
     $\V_{0}^{15}\subseteq\V$. Let also $\delta>0$ be such that $\w_{\delta}\subseteq\V_{0}$. 
      Let $\Delta_{0}$ be a finite embedded simplicial complex with 
     $\SS{\hh}{\Delta_{0}}\subseteq\V_{0}$, as given by Lemmas \ref{l: nerve} and \ref{l: different stabilizers}. By Lemma \ref{l: cleaning} there is $h_{0}\in\V_{0}$ such that $h_{0}\cdot\rl{\Delta_{0}^{(m-2)}}\cap D'=\emptyset$.
     Lemma \ref{l: wlog transverse} implies that for any $\delta>0$ there is $h_{1}\in\w_{\delta}$ such that $\Delta:=h_{1}h_{0}\cdot\Delta_{0}$ satisfies the previous property and, additionally, $\partial D'\transv \Delta$. By using Lemma \ref{l: wlog transverse} and the existence of tubular neighbourhoods, as in the proof of Lemma \ref{l: different stabilizers}, for any $\delta>0$ one finds $h_{2}\in\w_{\delta}$ such that
     $\Theta:=h_{2}\cdot\Delta$ is still transverse to $\partial D'$ and also satisfies $\rl{\Theta}\cap\rl{\Delta}\cap D=\emptyset$. We have 
     $$\SS{\hh}{\Delta}\subseteq\V_{0}^{5},\quad\quad\SS{\hh}{\Theta}\subseteq\V_{0}^{7}.$$ 
     
     If we let $U :=\mathring{D}$, then $L:=|\Delta|\cap U$ and $N:=|\Theta|\cap U$ are disjoint closed submanifolds of $U$. By assumption, $g_{\restriction U}\in\D_{c0}(U)$. 
     Therefore, by corollaries \ref{c: Morse} and \ref{c: compatibility and decompositions} there is $g_{0}\in\w_{\delta}$ supported on some compact subset of $U$ 
     and a compactly supported diffeotopy $G$ from $Id_{M}$ to $g_{0}g$ of the form 
     $G=G^{c}*G^{d}$, where 
     \begin{itemize}
     	\item $G^{d}_{\restriction N\times I}$ is $L$-compatible and non-destructive, with $r_{c}$-many tangency points,
     	\item and $H:=G^{d}_{\restriction N\times I}$ is $L$-compatible and purely destructive, with $r_{d}$-many tangency points.  
     \end{itemize}
     Pick $\V_{1}=\V_{1}^{-1}\in\nd$ such that $\V_{1}^{\max\{r',r''\}}\subseteq\V_{0}$. Let $g_{c}:=G^{c}_{1}$ and consider the compactly supported diffeotopy $\tilde{G}:=(G^{d}_{-t}g^{-1}g_{0}^{-1})_{t}$ from 
     $Id_{M}$ to $g_{c}g^{-1}g_{0}^{-1}$. Notice that $\tilde{G}_{\restriction g_{0}g(N)}=H^{op}$,
     which is non-destructive. The conclusion of Lemma \ref{l: inductive slices} allows us to apply Corollary \ref{c: compatibility and decompositions} 
     with $$\mathscr{X}=\mathscr{X}^{-1}:=\V_{1}\cap\D_{c0}(U)$$ (with $\D_{c0}(U)\subseteq\D_{c0}(M)$ as above) to both $G^{c}$ and $\tilde{G}$.
     We obtain decompositions: 
     \begin{align*}
     	g_{c}=g_{L}h_{r_{c}}h_{r_{c}-1}\cdots h_{1}g_{N},\quad\quad g_{c}g^{-1}g_{0}^{-1}=\tilde{g}_{L}\tilde{h}_{r_{d}}\cdots\tilde{h}_{1}\tilde{g}_{N},
     \end{align*}
     where $g_{L},\tilde{g}_{L}\in\D_{c0}^{\{L\}}(U)$, $g_{N}\in\D_{c0}^{\{N\}}(U)$,
     $\tilde{g}_{N}\in\D_{c0}^{\{g_{0}g(N)\}}(U)$ and $h_{i},\tilde{h}_{j}\in\mathscr{X}$ for 
     $1\leq i\leq r_{c}$ and $1\leq j\leq r_{d}$.  
     Writing $\hat{g}_{N}:=\tilde{g}_{N}^{g_{0}g}$, we then have $\hat{g}_{N}\in\D_{c0}^{\{N\}}(M)$ and   
     \begin{align*}
     	\label{dagdag}
     	g_{0}g =& \tilde{g}_{N}g_{0}g\hat{g}_{N}^{-1}=
     	\tilde{g}_{N}(g_{c}g^{-1}g_{0}^{-1})^{-1}g_{c}\hat{g}_{N}=\\
     	&\tilde{h}_{1}^{-1}\cdots\tilde{h}_{r_{d}}^{-1} \tilde{g}_{L}^{-1}g_{L}h_{r_{c}}\cdots h_{1}g_{N}\hat{g}_{N}
     \end{align*} 
     Our usual identification yields inclusions $\D_{c0}^{\{L\}}(U)\subseteq\SS{\kk}{\Delta}$
     and $\D_{c0}^{\{N\}}(U)\subseteq\PS{\kk}{\Theta}$. We thus have
     $$
     g\in g_{0}^{-1}\V_{1}^{r_{d}}\SS{\hh}{\Delta}\V_{1}^{r_{c}}\SS{\hh}{\Theta}\subseteq\V_{0}^{15}\subseteq\V,
     $$   
     as needed.
   \end{proof}
 
  \begin{customthm}{\ref{t: main}}
  	\bodythmdiffeos
  \end{customthm}
  
  \begin{proof}
   Assume that we are given some Hausdorff group topology $\T$ on $\gg$ strictly coarser than $\tc$ (the restriction of the compact-open topology). Corollary \ref{c: 0-slices} implies that arc compressions are dense around the identity in $\T$ (Definition \ref{d: compressions}). We can then apply Proposition \ref{p: slice induction}, whose conclusion implies that $\T$ is not Hausdorff: a contradiction.
  \end{proof}
  
  \begin{remark}
  	Note that Corollary \ref{c: 0-slices} was the key step requiring the assumption that $m\neq 3$, while the use of the assumption $\partial M=\emptyset$ plays a role already in Observation \ref{o: generation compact-open topology}.
  \end{remark}
  
  \section{Proof of Theorem \ref{c: main} }
  \label{s: proof main results}
  We will now deduce our main theorem \ref{c: main} in the introduction from the more technical Theorem \ref{t: main} we just proved. We begin with a general property of minimal groups, Theorem 3.1 of \cite{stephenson1971minimal}.
  \begin{fact}
  	\label{f: density and minimality} Let $(G,\T)$ be a topological group and $H$ a dense subgroup of $G$. Then the restriction of $\T$ to $H$ is a minimal group topology on $H$ if and only if $\T$ is a minimal group topology on $G$ and $H$ intersects all non-trivial closed normal subgroups of $G$ non-trivially.
  \end{fact}
   
   All manifolds in dimension $\leq 3$ admit a smooth structure and any $g\in\H(M)$ can be approximated by diffeomorphisms, while in dimension at least $5$ the property depends only on the isotopy class of $g$ (see \cite{muller2014uniform}). In particular, we have the following. 
  \begin{fact}[\cite{munkres1960obstructions} and \cite{muller2014uniform}, Thm. 4]
	  \label{f: mullers approximation} Let $M$ be a smooth connected manifold of dimension $m\neq 4$ without boundary, (possibly non-compact). Then for any $g\in\H_{c}(M)\cap\H_{0}(M)$ there is some 
	  compact set $K\subseteq M$ and some sequence $(g_{n})_{n\in\N}\subseteq\D(M)$ of elements supported on $K$ converging to $g$ in the compact-open topology (hence, uniformly). 
  \end{fact}
   
  We are now ready to prove the main Theorem in the introduction. 
  \begin{customthm}{\ref{c: main}} \bodythmhomeos
  \end{customthm}
   \begin{proof}
	   Let $\gg'$ be the closure of $\hh':=\D(M)\cap\H_{c0}(M)\subseteq\gg$ in $\gg$ with respect to the compact-open topology. We claim that $\H_{c0}(M)\subseteq\gg'$. Choose any $g\in\H_{c0}(M)$ and let $(g_{n})_{n\in\N}\subseteq\D_{c}(M)$ be the approximating sequence provided by Fact \ref{f: mullers approximation}. It follows by the main theorem of \cite{edwards1971deformations} that for $n$ large enough $g_{n}$ is isotopic to $g$ through a compactly supported topological ambient isotopy, so that $g_{n}\in\hh'$ for all $n$ and, therefore, $g\in\gg'$. 
	   
	   The group $\hh'$ clearly satisfies the assumptions of Theorem \ref{t: main} and therefore the restriction of the compact-open topology to $\hh'$ is minimal.
	   By Fact \ref{f: density and minimality} applied to $(\hh',\gg')$, we conclude that any Hausdorff group topology $\T$ on $\gg$ coarser than $\tc$ must agree with $\tc$ on $\gg'$. 
	   In particular, for any $\epsilon>0$ and compact $K\subseteq M$ with non-empty interior there must be some $\V\in\nd$ such that $\V\cap\H_{c0}(M)\subseteq\V\cap\gg'\subseteq\V_{K,\epsilon}$, in direct contradiction with Observation \ref{l: intrusive homeos}.
   \end{proof}

  \section{Questions and remarks}
  
  \label{s: questions}
  \subsection*{Dimension 2} Lemma \ref{l: commutation move} is actually not needed in the proof of the dimension $2$ case of Theorem \ref{t: main} on account of the fact that given two systems of curves (arcs) in general position one of which can be isotoped to be disjoint from the other we can always assume the isotopy to be non-destructive (the reverse of a sequence of bigon removals).

  \subsection*{Manifolds with boundary}
  	  
  Let $M$ be a smooth manifold with $dim(M)\notin\{1,3,4\}$, and assume $\partial M\neq\emptyset$. 
  Although we know that the compact-open topology on $\H(M)$ is not minimal, the image of $\H(M)$ in $\H(int(M))$ by the restriction map does satisfies the assumptions of Theorem \ref{c: main}.
	\begin{corollary}
	  In the situation above the topology $\tau_{co}^{\partial}$ which is the pull-back by $\rho$ of the compact-open topology on $\H(int(M))$ is minimal. 
	\end{corollary}

  More in general, given a subset $\mathcal{X}$ of the set of connected components of $\partial M$, we can define a topology $\tc^{\mathcal{X}}$ as the pull-back of the compact-open topology by the restriction map corresponding to the inclusion $M\setminus\bigcup_{C\in\mathcal{X}}C\subset M$. We suspect that the proof of Theorem \ref{t: main} can be generalized to show that this is all there is:
  \begin{conjecture}
  	Let $M$ be a connected smooth manifold with boundary of dimension $m\neq \{3,4\}$. Then any group topology on $\H(M)$ coarser than the compact-open topology is of the form $\tc^{\mathcal{X}}$ for some $\mathcal{X}\subseteq\pi_{0}(\partial M)$.
  \end{conjecture}
  
  \subsection*{Further questions}
   The results here leave some major gaps open. To begin with, there is the issue of the dimension $3$ case, which we hope might just be a matter of ingenuity rather than heavy machinery (the reader might remember the rather naive way in which the assumption $m\neq 3$ was used). 
    
   \begin{question}
   	Suppose that $M$ is a $3$-manifold without boundary. Is the compact-open topology on $\H(M)$ minimal?   
   \end{question}
   
   A positive answer to the following question seems very plausible. 
   \begin{question}
   	Can the arguments in this work be generalized from smooth to $PL$-manifolds? 
   \end{question}
   
   Lastly, we have the most general version of the problem, which more powerful technical tools or an entirely new strategy. 
  \begin{question}
  	Is the compact-open topology on $\H(M)$ minimal for any connected topological manifold $M$ without boundary?
  \end{question}

  \bibliographystyle{plain}
  \bibliography{bibliography}
  
\end{document}

    \begin{proof}
	    Fix $g\in\SS{\hh}{\Delta}$. By assumption, there is a compactly supported diffeotopy $G: M\times I\to M$ of $Id_{M}$ preserving each of the simplices of $\Delta$ setwise at all times and whose support is contained in $M\setminus U$ for some neighbourhood $V$ of $\rl{\Delta^{(m-2)}}$ in $M$. Write $U:=V\cap\rl{\Delta}$.  
	           
	    Let $\{\sigma_{i}\}_{i=1}^{r}$ be an enumeration of $\Delta^{m-1}$. 
	    For $1\leq l\leq 2$ choose some collections of disjoint embedded $(m-1)$ balls 
	    $\{D_{i}^{l}\}_{i=1}^{r}$, so that if we write $N_{l}:=\bigcup_{i=1}^{r}\mathring{D}_{i}^{l}$, then
	    \begin{itemize}
	    	\item  $D^{1}_{i}\subseteq \mathring{D}^{2}_{i}\subseteq D^{2}_{i}\subseteq\sigma_{i}\setminus\partial\sigma_{i}$ for $1\leq i\leq r$,
	    	\item $\rl{\Delta}\subseteq N_{1}\cup U$ (and thus $N_{2}\setminus N_{1}\subseteq U$). 
	    \end{itemize}
    	    
	    The existence of tubular neighbourhoods allows us to find an embedding 
	    $$\xi:E:=N_{2}\times (-1,4)\to M,$$
	    agreeing on $N_{2}\times\{0\}$ with the projection onto the first factor, $\pi:E\to N_{2}$, such that: 
	    \begin{enumerate}
	    	\item \label{clean preimage}$\xi(\pi^{-1}(N_{2}))\cap|\Delta|=\xi(N_{2}\times\{0\})$ 
	    	\item \label{fibers}for any $p\in N_{2}$ we have $diam(\{p\}\times(-1,4))<\epsilon$.
	    \end{enumerate}  
 
      Let $f:N_{2}\to[0,2]$ some smooth map such that $f_{\restriction N_{1}}=2$, $f_{\restriction V}=0$ in the complement of some compact subset of $N_{2}$. By Lemma \ref{l: isotopy extension} applied to an isotopy of a point in $(-1,4)$ one can easily construct $\hat{h}_{1}\in\D_{c0}(E)$ diffeotopic to the identity by a compactly-supported fiber-preserving diffeotopy of $E$, and mapping $N_{2}\times\{0\}$ to the graph of $f$. 
      If we denote by $h_{1}$ the extension of the map $\xi\circ\hat{h}_{1}\circ\xi^{-1}$ to $M$ by the identity, then the push-forward of the diffeotopy above to $\xi(E)$ can be also extended by the identity, and we conclude that $h_{1}\in\w_{\epsilon}$, by (\ref{fibers}). 
      
      Fix now some smooth function $\rho:(-1,4)\to [0,1]$ with $\rho(2)=1$ and $\rho(s)=0$ outside of $(1,3)$ and consider now the map $\hat{H}: E\times I\to I$ of defined by 
      \begin{align*}
	     \hat{H}((p,s),t):= (G(p,(t+1)\rho(s)-1),s).
      \end{align*}
      It is easy to check that $\hat{H}$ is a diffeotopy of $Id_{E}$.        
       Moreover, for any $p\in N_{2}$ and $t\in I$ we have
       \begin{equation}
	       \label{eq fix} \tag{\protect\pluma}\hat{H}_{t}(p,0)=(G(p,(t+1)\rho(0)-1),0)=(G(p,-(t+1)0-1),0)=(p,0).
       \end{equation}
       For $p\in N_{1}$ and $t\in I$ we also have: 
       \begin{align}
       \label{eq conj} \tag{\protect\plumaa}
       \begin{split}
       \hat{h}_{1}^{-1}\hat{H}_{t}\hat{h}_{1}(p,0)=&\hat{h}_{1}^{-1}(\hat{H}_{t}(p,2))=\hat{h}_{1}^{-1}(G(p,(t+1)\rho(2)-1),2)=\\&\hat{h}_{1}^{-1}(G(p,t),2)=(G_{t}(p),0)
       \end{split}
       \end{align}
       Likewise, since $supp(G)\subseteq N_{1}$ and $\hat{h}_{1}$ is fiber-preserving, we conclude that 
       \begin{equation}
       	\label{eq fixfix}  \tag{\protect\plumaaa} supp(\hat{H}),supp(\hat{h}_{1}^{-1}\hat{H}_{t}\hat{h}_{1})_{t\in I}\subseteq\pi^{-1}(N_{1}),
       \end{equation}    
       One deduces from (\ref{clean preimage}), (\ref{eq fix}) and (\ref{eq fixfix}) that
       the diffeotopy $(\xi\circ\hat{H}_{t}\circ\xi^{-1})_{t\in I}$ of $\xi(E)$ extends by the identity to a compactly supported diffeotopy $H$ of $Id_{M}$ supported in the complement of the union of $\rl{\Delta}$ and some neighbourhood of $\rl{\Delta^{(m-2)}}$.  
       It follows that $h_{2}:=H_{1}$ belongs to $\PS{\hh}{\Delta}$. 
       
       Consider now $G^{*}:=M\times I\to M$ given by $G_{t}^{*}:=G_{t}^{-1}h_{1}^{-1}H_{t}h_{1}$. This is a  compactly supported diffeotopy (by Fact \ref{c: horizontal composition isotopies}). We also have:
       \begin{itemize}
       	\item $G^{*}_{-1}=Id_{M}h_{1}^{-1}Id_{M}h_{1}=1$ and $G^{*}_{1}=g^{-1}h_{1}^{-1}h_{2}h_{1}$, 
       	\item $supp(G^{*})$ is disjoint from the neighbourhood $U\cap\xi(\pi^{-1}(V))$ of $\rl{\Delta^{(m-2)}}$, by (\ref{eq fixfix}) and so, in particular, disjoint from $V=U\cap\rl{\Delta}$,
       	\item and from (\ref{eq conj}) we deduce that $G_{t}$ is the identity on $N_{1}$ for all $t\in I$, so that it is the identity on $\rl{\Delta}$, by the previous bullet point. 
       \end{itemize}
       Therefore, $G^{*}$ witnesses $g^{-1}h_{2}^{h_{1}}\in\PS{\hh}{\Delta}$, 
       which implies that
       $$g\in h_{2}^{h_{1}}\PS{\hh}{\Delta}\subseteq(\PS{\hh}{\Delta})^{h_{1}}\PS{\hh}{\Delta}\subseteq\w_{\epsilon}\PS{\hh}{\Delta}\w_{\epsilon}\PS{\hh}{\Delta},$$
       settling the first assertion of the lemma. The second one follows by the continuity of multiplication.
    \end{proof}